\documentclass{amsart}
\usepackage{amsfonts,amssymb,latexsym}
\usepackage{amsmath,amsthm}
\usepackage{eucal}
\usepackage[latin1]{inputenc}
\usepackage{color}

\newtheorem{theorem}{Theorem}[section]
\newtheorem{lemma}[theorem]{Lemma}
\newtheorem{proposition}[theorem]{Proposition}
\newtheorem{corollary}{Corollary}[section]

\theoremstyle{definition}
\newtheorem{definition}[theorem]{Definition}

\newtheorem{remark}{Remark}[section]
\def\indi{1\!\! 1}
\def\R{{\mathbb R}}
\def\N{{\mathbb N}}

\def\C{{\mathbb C}}
\newcommand{\T}{\mathbb{T}}
\newcommand{\Z}{\mathbb{Z}}
\def\supp{\mathop{\rm supp}\nolimits}

\numberwithin{equation}{section}

\begin{document}
\title[Unconditional well-posedness for mKdV]{On unconditional well-posedness for the periodic modified Korteweg-de Vries equation }


\author[L. Molinet, D. Pilod and S. Vento]{Luc Molinet$^*$, Didier Pilod$^\dagger$  and St\'ephane Vento$^*$}

\thanks{$^*$ Partially supported by the french ANR project GEODISP}
\thanks{$^{\dagger}$ Partially supported by CNPq/Brazil, grants 303051/2016-7 and 431231/2016-8.}

\address{Luc Molinet, Laboratoire de Math\'ematiques et Physique Th\'eorique (CNRS UMR 7350), Universit\'e Fran\c cois Rabelais-Tours, Parc Grandmont, 37200 Tours, France.}
\email{Luc.Molinet@univ-tours.fr}

\address{Didier Pilod, Instituto de Matem\'atica, Universidade Federal do Rio de Janeiro, Caixa Postal 68530, CEP: 21945-970, Rio de Janeiro, RJ, Brasil.}
\email{didier@im.ufrj.br}

\address{St\'ephane Vento, Laboratoire d'Analyse, G\'eom\'etrie et Applications, Universit\'e Paris 13, Institut Galil\'ee, 99 avenue J. B. Cl\'ement, 93430 Villetaneuse, France.}
\email{vento@math.univ-paris13.fr}

\date{\today}

\maketitle

\begin{abstract}
We prove that the modified Korteweg-de Vries equation is   unconditionally well-posed in $ H^s(\T) $ for $ s\ge 1/3 $. For this we gather the smoothing effect first discovered
 by Takaoka and Tsutsumi with an approach developed by the authors that combines the energy method, with Bourgain's type estimates, improved Strichartz estimates and the construction of modified energies.

\end{abstract}

\section{Introduction}

We consider the initial value problem (IVP) associated to the modified Korteweg-de Vries  (mKdV) equation
\begin{align}
& \label{mKdV}
\partial_tu+\partial_x^3u
\mp \partial_x(u^3)=0 \\
& \label{ini} u(\cdot,0)=u_0 \,,
\end{align}
where  $u=u(x,t)$ is a real valued function, $x \in \mathbb T=\R /\Z $ and $t \in \mathbb R$.

In \cite{Bourgain1993} Bourgain introduced the Fourier restriction norm method and proved that \eqref{mKdV} is locally well-posed in $ H^s(\T) $ for $ s \ge 1/2 $. Note that, by a change of variable, Bourgain substituted the mKdV equation \eqref{mKdV} by the renormalized mKdV equation
$$
 v_t + \partial_x^3 v \mp  3(v^2-P_0(v_0^2))v_x =0 , \quad v(\cdot,0)=v_0 \, ,
$$
 where $ P_0 w $ denotes the mean value of $ w$.
 This result  was then proved to be sharp if one requires moreover the smoothness or the uniform continuity on bounded sets of the solution-map  associated with the renormalized equation
  (see \cite{bourgain2}, \cite{KPV3}, \cite{CCT2003}).  This obstruction is related to the resonant term $ \sum_{k\in\Z} |\widehat{v}(k)|^2 \widehat{v}(k) e^{ikx} $ that appears in the nonlinear part of  this equation.
 However, in \cite{TT}, Takaoka-Tsutsumi proved that \eqref{mKdV} is locally well posed in $H^s(\T) $ for $ s>3/8 $. For this, they first establish a smoothing
 effect on the difference $ |{\mathcal F}_x(v(t))(k)|^2-|\widehat{v_0}(k)|^2 $ and then work in a Bourgain's space depending on the  initial data in order to treat the resonant
  term. This was improved in \cite{NTT} where the local well-posedness was pushed down to $H^s(\T) $ with $ s>1/3$.

The local well-posednesss results proved in these papers mean the following : for any initial data  $ u_0\in H^s(\T) $ there exists a time
     $ T=T(\|u_0\|_{H^s})>0 $ only depending on $\|u_0\|_{H^s} $ and a solution  $ u$ that satisfies the equation  at least in  some weak sense and is unique in some function space (called resolution space) $ X_T\hookrightarrow  C([0, T];H^s(\T))$ that can depend on the initial data. Moreover, for any $ R>0 $,  the flow-map $ u_0\mapsto u $ is continuous from the ball centered at the origin with radius $ R $ of $ H^s(\T) $ into $C([0,T(R)];H^s(\T)) $.

     On the other hand,
in \cite{KT2}, Kappeler and Topalov introduced the following notion of solutions which a priori does not always corresponds to   the solution in the sense of distributions:
     {\it A  continuous curve $ \gamma \, :\, (a,b) \to H^\beta(\T) $ with $ 0\in (a,b)$ and $ \gamma(0)=u_0 $ is called a solution of the mKdV equation  in $ H^\beta(\T) $ with initial data $ u_0 $ iff for any $ C^\infty $-sequence of initial data $ \{u_{0,n}\} $ converging to $ u_0 $ in $ H^\beta(\T) $ and for any $ t\in ]a,b[ $, the sequence of emanating solutions  $ \{ u_n\} $ of the mKdV equation satisfies : $ u_n(t) \to \gamma(t) $ in $H^\beta(\T) $}

     Note that a solution in the sense of this definition is necessarily unique. With this notion of solution they proved the global well-posedness of the  defocusing mKV equation (with a $+$ sign in front of the nonlinear term) in $ H^s(\T) $
     $ s\ge 0 $, with a solution-map which is continuous from  $L^2(\T) $ into $C(\R;L^2(\T)) $.  Their proof is based on the inverse scattering method and thus depends in a crucial way of the complete integrability of this equation. It is worth noticing that, by Sobolev embedding theorem, their solutions of the defocussing mKdV equation satisfy the equation  in the  distributional sense as soon as $ s\ge 1/6$.
In \cite{M1} Molinet proved that, actually,  the solutions constructed by Kappeler-Topalov always satisfy the equation at least in a weak sense. He also proved that the flow-map cannot be continuously extended in $ H^s(\T) $ as soon as $ s<0$. Therefore the result of Kappeler-Topalov is in some sense optimal. However it is not known to hold for the focusing equation. Moreover, it uses the integrability of the equation and is thus not suitable to solve perturbations of the defocusing mKdV equation. Also, the question  of the existence of a resolution  space where the uniqueness holds remains open at this low regularity level.

Another interesting question about uniqueness, even in higher regularity, is to know wether uniqueness also holds in some larger spaces that contain weak solutions.
 This kind of question was first raised by Kato \cite{Kato} in the Schr\"odinger equation context. We refer to such uniqueness in  $ L^\infty(]0,T[ ;H^{s})$, without intersecting with any auxiliary function space as \textit{unconditional uniqueness}.  This ensures  the uniqueness of the weak solutions to  the equation at the $ H^s$-regularity. This is useful, for instance, to pass to the limit on perturbations of the equation  as the perturbative coefficient tends to zero (see for instance \cite{M2} for such an application).

Unconditional uniqueness was proved for the mKdV equation to hold in $H^{1/2}(\T)$ by Kwon and Oh (\cite{KT}) following an approach developed in
\cite{BaIlTi}. In this paper we push down the local well-posedness and  the unconditional uniqueness for the mKdV equation to $ H^{1/3}(\T) $.

To obtain our unconditional uniqueness result we gather the approach developed in \cite{MoPiVe} based on the construction of modified energies with some ideas of \cite{TT} and \cite{NTT} to derive the smoothing effect. On the one hand, the absence of very  small frequencies enables to   simplify some estimates on the nonlinear term with respect to
  \cite{MoPiVe}.  On the other hand, because of true resonances, we need to derive a smoothing effect as in \cite{TT}.
 Actually this is the obtention of the smoothing effect that limits us to the Sobolev index $ s\ge 1/3 $ (see Remark \ref{opti}). It is also worth noticing that we do not succeed to get an estimate on the $ L^\infty_T H^s $-norm of the difference of two solutions with different initial data - this seems to be related to the fact  that the flow-map is not Lipschitz below $ s=1/2 $.
  Instead we will establish an a priori estimate in  $L^\infty_T H^{s'} $, for some $ s'	<s $, on the difference of two solutions emanating
  from the same initial datum. This estimate will lead to the unconditional uniqueness result. It will be also sufficient to prove the well-posedness result thanks to the smoothing effect which ensures that, given a sequence of solutions $ \{u_n\} \subset L^\infty(0,T;H^s(\T)) $ to \eqref{remKdV} associated with a sequence of initial data $ \{u_{0,n}\} $ relatively compact in $ H^s(\T) $, the set $ \{u_n(t), \, t\in [0,T]\} $ is relatively   compact in  $ H^s(\T) $.

\section{Notations, Functions spaces and statement of the result}
We will not work directly with the mKdV equation but with the renormalized mKdV equation defined by
 \begin{equation}\label{remKdV}
 u_t + \partial_x^3 u \mp \partial_x (u^3-3P_0(u^2)u)  =0 \, .
 \end{equation}
  We explain how to come back to the mKdV equation \eqref{mKdV} in Subsection \ref{back}.  In the sequel of this paper, we choose to  the take the sign  \lq\lq+\rq\rq \, since this  sign will not play any role in our analysis.
 Let us  start by giving our notion of solution.
\begin{definition}\label{def} Let $T>0$ and $ s\ge \frac{1}{6}$. We will say that $u\in L^\infty(0,T;H^s(\T)) $ is a solution to \eqref{mKdV} (resp. \eqref{remKdV} ) associated with the initial datum $ u_0 \in H^s(\T)$  if
  $ u $ satisfies \eqref{mKdV}-\eqref{ini}  (resp. \eqref{remKdV}-\eqref{ini}) in the distributional sense, i.e. for any test function $ \phi\in C_c^\infty(]-T,T[\times \T) $,  there holds
  \begin{equation}\label{weakmKdV}
  \int_0^\infty \int_{\T} \Bigl[(\phi_t +\partial_x^3 \phi )u +  \phi_x F(u) \Bigr] \, dx \, dt +\int_{\T} \phi(0,\cdot) u_0 \, dx =0
  \end{equation}
  where  $ F(u)=u^3 $ (resp. $ F(u)=u^3-3P_0(u^2) u  $).

 \end{definition}
 \begin{remark} \label{rem2} Note that for $u\in L^\infty(0,T;H^s(\T)) $,  with $ s\ge \frac{1}{6} $, $ u^3 $ is well-defined and
  belongs to $ L^\infty(0,T;L^1(\T))$.   Therefore \eqref{weakmKdV} forces $ u_t \in L^\infty(0,T; H^{-3}(\T)) $  and  ensures that  \eqref{mKdV}
   (resp. \eqref{remKdV}) is satisfied in $ L^\infty(0,T; H^{-3}(\T)) $.
  In particular, $ u\in C([0,T]; H^{-3}(\T))$ and \eqref{weakmKdV}  forces  the initial condition $ u(0)=u_0 $. Note that, since $ u\in L^\infty(0,T;H^s(\T)) $, this actually ensures that  $ u\in C_{w}([0,T]; H^s(\T)) $ and $ u\in C([0,T]; H^{s'}(\T))$ for any $ s'<s $. Finally, we notice that this also ensures that $ u $ satisfies the Duhamel formula associated with \eqref{mKdV}-\eqref{ini}) (resp. \eqref{remKdV}-\eqref{ini}).
 \end{remark}

 \begin{definition} Let  $ s\ge \frac{1}{6} $. We will say that the Cauchy problem associated with \eqref{mKdV} (resp. \eqref{remKdV}) is unconditionally locally well-posed in $ H^s(\T )$ if for any initial data $ u_0\in H^s(\T) $ there exists $ T=T(\|u_0\|_{H^s})>0 $ and a solution
 $ u \in C([0,T]; H^s( \T)) $ to \eqref{mKdV} (resp. \eqref{remKdV}) emanating from $ u_0 $. Moreover, $ u $ is the unique solution to  \eqref{mKdV}  (resp. \eqref{remKdV}) associated with $ u_0 $ that belongs to  $ L^\infty(]0,T[; H^s(\T) )$. Finally, for any $ R>0$, the solution-map $ u_0 \mapsto u $ is continuous from the ball of $ H^s(\T) $  with radius $ R $ centered at the origin  into $C([0,T(R)]; H^s( \T)) $.
 \end{definition}
\begin{theorem}\label{main}
The mKdV equation  \eqref{mKdV}  and the renormalized mKdV equation \eqref{remKdV} are unconditionally locally well-posed in $ H^s(\T ) $ for $ s\ge 1/3 $.
\end{theorem}

\subsection{Notation}\label{sect21} Throughout this paper, $ \N $ denotes the set of non negative integer numbers.
For any positive numbers $a$ and $b$, the notation $a \lesssim b$
means that there exists a positive constant $c$ such that $a \le c
b$. We also denote $a \sim b$ when $a \lesssim b$ and $b \lesssim
a$. Moreover, if $\alpha \in \mathbb R$, $\alpha_+$, respectively
$\alpha_-$, will denote a number slightly greater, respectively
lesser, than $\alpha$.

For real numbers $a_1, \ a_2, \ a_3  >0$, we define the quantities
$a_{max} \ge a_{med} \ge a_{min}$ to be the maximum, sub-maximum and minimum of $a_1, \ a_2 $  and $a_3$. Usually, we use $k_i$, $j_i$ to denote integers and $N_i=2^{k_i}$, $L_i=2^{j_i}$ to denote dyadic numbers.
For $ f=f(x)  \in L^2(\mathbb T)$, we denote its Fourier transform by $ \hat{f}\, : \, \Z\to \C $ by $ \widehat{f}(k)=\int_{\T} e^{-2i\pi k x} f(x)\, dx $ and or any integer $ k\in \N $ we  set
$$
P_k f =\widehat{f}(k) e^{2i\pi kx}, \; P_{\sim k} = \sum_{|q| \sim k}\widehat{f}(q) e^{2i\pi kqx} \mbox{ and } P_{\lesssim k} u = \sum_{|q| \lesssim k}\widehat{f}(q) e^{2i\pi qx}\; .
$$
In particular,
$$
P_0 f = \widehat{f}(0)=\int_{\T} f(x) \, dx \, .
$$
For $u=u(x,t) \in \mathcal{S}'(\mathbb R^2)$,
$\mathcal{F}_{tx}u=(u)^{\wedge_{tx}} $ will denote its space-time Fourier
transform, whereas $\mathcal{F}_x u=\widehat{u}$, respectively
$\mathcal{F}_tu=(u)^{\wedge_t}$, will denote its Fourier transform
in space, respectively in time. For $s \in \mathbb R$, we define the
Bessel and Riesz potentials of order $-s$, $J^s_x$ and $D_x^s$, by
\begin{displaymath}
J^s_xu=\mathcal{F}^{-1}_x\big((1+|k|^2)^{\frac{s}{2}}
\mathcal{F}_xu\big) \quad \text{and} \quad
D^s_xu=\mathcal{F}^{-1}_x\big(|k|^s \mathcal{F}_xu\big).
\end{displaymath}

We also denote by $U(t)=e^{-t\partial_x^3}$ the unitary group associated to the linear part of \eqref{mKdV}, \textit{i.e.},
\begin{displaymath}
U(t)u_0=e^{-t\partial_x^3}u_0=\mathcal{F}_x^{-1}\big(e^{itk^3}\mathcal{F}_x(u_0)(k) \big) \, .
\end{displaymath}
Throughout the paper, we fix a smooth cutoff function $\chi$ such that
\begin{displaymath}
\chi \in C_0^{\infty}(\mathbb R), \quad 0 \le \chi \le 1, \quad
\chi_{|_{[-1,1]}}=1 \quad \mbox{and} \quad  \mbox{supp}(\chi)
\subset [-2,2].
\end{displaymath}
We set  $ \phi(k):=\chi(k)-\chi(2k) $. For any  $l \in \N$, we define
\begin{displaymath}
\phi_{2^l}(k):=\phi(2^{-l}k), \end{displaymath} and
\begin{displaymath}
\psi_{2^{l}}(k,\tau)=\phi_{2^{l}}(\tau-k^3).
\end{displaymath}
By convention, we also denote
\begin{displaymath}
\phi_0(k)=\chi(2 k) \quad \text{and} \quad \psi_{0}(k,\tau):=\chi(2(\tau-k^3)) \, .
\end{displaymath}
Any summations over capitalized variables such as $N, \, L$, $K$ or
$M$ are presumed to be dyadic.We work with non-homogeneous dyadic decompositions
\textit{i.e.}, these variables range over numbers of the form $\{2^k
: k \in \mathbb N \} \cup \{0\}$. We call those numbers \textit{nonhomogeneous dyadic numbers}. Then, we have that $\sum_{N}\phi_N(k)=1$
\begin{displaymath}
 \mbox{supp} \, (\phi_N) \subset
I_N:=\{\frac{N}{2}\le |k| \le 2N\}, \ N \ge 1, \quad \text{and} \quad
\mbox{supp} \, (\phi_0) \subset I_0:=\{|k| \le 1\}.
\end{displaymath}

Finally, let us define the Littlewood-Paley multipliers $P_N$, $R_K$ and $Q_L$ by
\begin{displaymath}
P_Nu=\mathcal{F}^{-1}_x\big(\phi_N\mathcal{F}_xu\big), \quad R_Ku=\mathcal{F}^{-1}_t\big(\phi_K\mathcal{F}_tu\big) \quad \text{and} \quad
Q_Lu=\mathcal{F}^{-1}\big(\psi_L\mathcal{F}u\big),
\end{displaymath}
 $P_{\ge N}:=\sum_{K \ge N} P_{K}$,  $P_{\le N}:=\sum_{K \le N} P_{K}$, $Q_{\ge L}:=\sum_{K \ge L} Q_{K}$ and   $Q_{\le L}:=\sum_{K \le L} Q_{K}$.

 Sometimes, for the sake of simplicity and when there is no risk of confusion, we also denote $u_N=P_Nu$.
 \subsection{Function spaces} \label{spaces}

For $1 \le p \le \infty$, $L^p(\mathbb T)$ is the usual Lebesgue
space with the norm $\|\cdot\|_{L^p}$. For $s \in \mathbb R$,
the Sobolev space $H^s(\mathbb T)$  denotes the space of all  distributions of $(C^\infty(\mathbb T))'$ whose usual
norm $\|u\|_{H^s}=\|J^s_xu\|_{L^2}$ is finite.

 If $B$ is one of the spaces defined above, $1 \le p \le \infty$ and $T>0$, we
define the space-time spaces $L^p_ t B_x$, and $L^p_TB_x$  equipped with  the norms
\begin{displaymath}
\|u\|_{L^p_ t B_x} =\Big(\int_{\R}\|f(\cdot,t)\|_{B}^pdt\Big)^{\frac1p} , \quad
\|u\|_{L^p_ T B_x} =\Big(\int_0^T\|f(\cdot,t)\|_{B}^pdt\Big)^{\frac1p}
\end{displaymath}
with obvious modifications for $ p=\infty $.
For $s$, $b \in \mathbb R$, we introduce the Bourgain spaces
$X^{s,b}$   related to  linear KdV group as
the completion of the Schwartz space $\mathcal{S}(\mathbb R^2)$
under the norm
\begin{equation} \label{X1}
\|u\|_{X^{s,b}}:= \left(
\sum_{k} \int_{\mathbb{R}}\langle\tau-k^3\rangle^{2b}\langle k\rangle^{2s}|\mathcal{F}_{xt}u(k, \tau)|^2
 d\tau \right)^{\frac12},
\end{equation}
where $\langle x\rangle:=(1+|x|^2)^{1/2}$.
It is easy to check that
\begin{equation} \label{X2}
\|u\|_{X^{s,b}}\sim\| U(-t)u \|_{H^{s,b}_{x,t}} \quad \text{where} \quad \|u\|_{H^{s,b}_{x,t}}=\|J^s_xJ^b_tu\|_{L^2_{x,t}} \, .
\end{equation}
We defined the function space $ Z^s $, with $ s\ge 0 $, by
\begin{equation}\label{qq}
Z^s=X^{s-\frac{11}{10},1}\cap L^\infty_t H^s_x \,.
\end{equation}
Finally,  we will also use restriction in time versions of these spaces.
Let $T>0$ be a positive time and $ B $ be a normed space of space-time functions. The restriction space $ B_T $ will
 be the space of functions $u: \mathbb R \times
]0,T[\rightarrow \mathbb R$ or $\mathbb C$ satisfying
\begin{displaymath}
\|u\|_{B_{T}}:=\inf \{\|\tilde{u}\|_{B} \ | \ \tilde{u}: \mathbb R
\times \mathbb R \rightarrow \mathbb R \ \text{or} \ \mathbb C, \ \tilde{u}|_{\mathbb R
\times ]0,T[} = u\}<\infty \; .
\end{displaymath}
\subsection{Extension operator} The aim of this subsection is to construct a bounded linear operator from  $ X^{s-\frac{11}{10},1}_T\cap L^\infty_T H^s_x  $
 into $ Z^s $ with a bound  that does not depend on $ s $ and $ T $ .  For this we follow \cite{MN} and introduce the extension operator $ \rho_T $ defined by
\begin{equation}\label{defrho}
\rho_T (u)(t):= U(t)\chi(t) U(-\mu_T(t)) u(\mu_T(t))\; ,
\end{equation}
where $ \chi $ is the smooth cut-off function defined in Section \ref{sect21} and $\mu_T $ is the  continuous piecewise affine
 function defined  by
\begin{equation}\label{defext}
 \mu_T(t)=\left\{\begin{array}{rcl}
 0  &\text{for } &  t\not\in ]0,2T[ \\
 t  &\text {for }  & t\in [0,T] \\
 2T-t & \text {for } &  t\in [T,2T]
 \end{array}
 \right.
\end{equation}
\begin{lemma} \label{extension}
Let $0<T \le 1$ and $s \in \mathbb R$. Then,
\begin{displaymath}
\begin{split}
\rho_T : \ & X^{s-\frac{11}{10},1}_T\cap L^\infty_T H^s_x  \longrightarrow  Z^s\\
 &u \mapsto \rho_T(u)
\end{split}
\end{displaymath}
is a bounded linear operator, \textit{i.e.}
\begin{equation} \label{extension.1}
\begin{split}
\|\rho_T(u)\|_{L^{\infty}_t H^s_x} + \|\rho_T(u)\|_{X^{s-\frac{11}{10},1}}\lesssim \|u\|_{L^\infty_T H^s_x}+\|u\|_{X^{s-\frac{11}{10},1}_T} \, ,
\end{split}
\end{equation}
for all $u \in X^{s-\frac{11}{10},1}_T\cap L^\infty_T H^s_x$.

Moreover, the implicit constant in \eqref{extension.1} can be chosen independent of $0<T \le 1$ and $s\in \R $.
\end{lemma}
\begin{proof}
On  one hand, the unitarity of the free group $ U(\cdot) $ in $ H^s(\T) $ easily leads to
$$
\|\rho_T(u)\|_{L^\infty_t H^s} \lesssim \|u(\mu_T(\cdot))\|_{L^\infty_t H^s} \lesssim \|u\|_{L^\infty_T  H^s} \vee \|u_0\|_{H^s}
$$
On the other hand,  the definition of the $ X^{s,b}$-norm and the continuity of  $ \mu_T $  lead to
\begin{align*}
\|\rho_T(u)\|_{X^{s-\frac{11}{10},1}} & = \|\chi\, U(-\mu_T(\cdot)) u(\mu_T(\cdot))\|_{H^{{s-\frac{11}{10}},1}_{x,t} }\\
& \lesssim  \|\chi\,  U(-\mu_T(\cdot)) u(\mu_T(\cdot))\|_{L^2_t H^{s-\frac{11}{10}} }+\Bigl\|\partial_t \Bigl( \chi\,  U(-\mu_T(\cdot))
u(\mu_T(\cdot))\Bigr)\Bigr\|_{ L^2_t H^{s-\frac{11}{10}} } \\
& \lesssim \|u(0)\|_{H^{s-\frac{11}{10}}}+ \|U(-\cdot) u\|_{L^2(]0,T[; H^{s-\frac{11}{10}})} +\| U(-\cdot) (u_t+\partial_x^3u)\|_{L^2(]0,T[; H^{s-\frac{11}{10}})} \\
&\hspace*{9mm}+\|U(T-\cdot) \Bigl(-u_t (T-\cdot)-\partial_x^3u(T-\cdot)\Bigr)\|_{L^2(]T,2T[; H^{s-\frac{11}{10}})}\\
& \lesssim \|u(0)\|_{H^{s-\frac{11}{10}}}+ \|u\|_{X^{s-\frac{11}{10},1}_T} \; .
\end{align*}
Now, since it is well-known (see for instance \cite{ginibre}), that $X_T^{s-\frac{11}{10},1} \hookrightarrow C([0,T]:H^{s-\frac{11}{10}}(\mathbb R))$,
 we infer that $u\in C([0,T]; H^{s-\frac{11}{10}}(\mathbb R))\cap L^{\infty}_T H^s_x$
 and  we claim that
\begin{equation} \label{tg100}
\|u(0)\|_{H^s} \le \|u\|_{L^{\infty}_TH^s_x} \, .
\end{equation}
Indeed, if $\|u(0)\|_{H^s} > \|u\|_{{L^{\infty}_T}H^s_x}$ there would exist $\epsilon>0$ and a decreasing sequence $\{t_n\} \subset (0,T)$ tending to $0$ such that for any $n \in \mathbb N$, $\|u(t_n)\|_{H^s} \le \|u(0)\|_{H^s}-\epsilon$.  The continuity of $ u $ with values in $ H^{s-1}(\R) $ then ensures that $ u(t_n) \rightharpoonup u(0) $ in $H^s(\mathbb R)$, which forces $ \|u(0)\|_{H^s}\le \liminf \|u(t_n)\|_{H^s} $   and yields a contradiction.

Gathering the two above  estimates, we thus  infer that for any $ (T,s)\in ]0,+\infty[\times \R $, $ \rho_T $ is a bounded linear operator from
 $ L^\infty_T H^s \cap X^{s-\frac{11}{10},1}_T $ into $ L^\infty_t H^s\cap X^{s-\frac{11}{10},1} $ with a bound that does not depend on $ (T,s)$.
\end{proof}

\section{A priori estimates on  solutions}
\subsection{Preliminaries}

\begin{definition} \label{def.pseudoproduct}
Let $\eta$ be a (possibly complex-valued) bounded function on $\mathbb Z^3$. We define the \textit{pseudo-product} $\Pi_\eta$ (that will also be denoted by $\Pi$ when there is no risk of confusion)  in Fourier variable by
\begin{equation} \label{def.pseudoproduct1}
\mathcal{F}_x\big(\Pi_{\eta}(f,g,h) \big)(k)=\sum_{\tiny{k_1+k_2+k_3=k}}\eta(k_1,k_2,k_3)\widehat{f}(k_1)\widehat{g}(k_2)\widehat{h}(k_3) \, .
\end{equation}

Moreover  for  any dyadic integer $M\ge 1 $ and any $ j\in \{1,2,3\} $, we also denote $\Pi_{\eta,M}^j$ (or $\Pi_M^j$ when there is no risk of confusion) the operator defined in Fourier variable by
\begin{equation} \label{def.pseudoproduct2}
\mathcal{F}_x\big(\Pi_{\eta,M}^j(f,g,h) \big)(k)=\sum_{\tiny{k_1+k_2+k_3=k }}\eta(k_1,k_2,k_3)\phi_M(\sum_{ 1 \le q \le 3 \atop q\neq j}k_q)\widehat{f}(k_1)\widehat{g}(k_2)\widehat{h}(k_3) \, .
\end{equation}
\end{definition}

The following technical lemma  corresponds to integration by parts for some  pseudo-products  (cf. \cite{MoPiVe}).
\begin{lemma} \label{technical.pseudoproduct}
Assume that the  (possibly complex-valued) bounded mesurable  function $ \eta $ satisfies on $\mathbb Z^3$
\begin{equation}\label{symmetry}
\eta(k_1,k_2,k_3)=\eta(k_2,k_1,k_3)= \Lambda(k_1,k_2)\Theta(|k_2+k_3|, |k_1+k_3|,|k_1+k_2| )
\end{equation}
for some functions  $\Theta $ on $ \Z^3$ and $ \Lambda $ on $ \Z^2$ with $\Lambda(k_1,k_2)=\Lambda(k_2,k_1) , \; \forall (k_1,k_2)\in \Z^2$.
Let $N$ and $M$ be two nonhomogeneous dyadic numbers satisfying $N \gg 1$. For any real-valued functions $f_1, \, f_2, \, g \in L^2(\mathbb R)$, we define
\begin{equation} \label{technical.pseudoproduct.1}
T_{\eta,M,N}(f_1,f_2,g,g) =\int_{\mathbb T} \partial_x P_N \Pi_{\eta,M}^3(f_1,f_2,g)P_N g \, dx \, .
\end{equation}
Then, for $M \ll N$, it holds
\begin{equation} \label{technical.pseudoproduct.2}
T_{\eta, M,N}(f_1,f_2,g,g)= M \int_{\mathbb T} \Pi_{\eta \eta_2,M}^3(f_1,f_2,P_{\sim N}g)P_Ng \, dx ,
\end{equation}
where $\eta_2$ is a function of $(k_1,k_2,k_3)$ whose $l^{\infty}-$norm is uniformly bounded in $N$ and $M$.
\end{lemma}

\begin{proof} From Plancherel's identity we may rewrite $T_{\eta,M,N}$ as
\begin{equation*}
\begin{split}
&T_{\eta,M,N}(f_1,f_2,g,g)\\ & \quad= \sum_{(k_1,k_2,k_3,k_4)\in \Z^4 \atop \tiny{k_1+k_2+k_3-k=0 }} (ik)\eta(k_1,k_2,k_3)\phi_M(k_1+k_2) \phi_N(k)^2\widehat{f}_1(k_1)\widehat{f}_2(k_2)\widehat{g}(k_3)\overline{\widehat{g}(k)} \; .
\end{split}
\end{equation*}
We  decompose $T_{M,N}(f_1,f_2,g,g)$ as follows:
\begin{equation} \label{technical.pseudoproduct.3}
\begin{split}
&T_{\eta,M,N}(f_1,f_2,g,g) \\
&= M\sum_{(k_1,k_2,k_3,k)\in \Z^4 \atop \tiny{k_1+k_2+k_3-k=0 }} (ik) \eta(k_1,k_2,k_3)\phi_M(k_1+k_2)  \Bigl(\frac{\phi_N(k)-\phi_N(k_3)}{M}\Bigr)\widehat{f}_1(k_1)\widehat{f}_2(k_2)\widehat{g}(k_3)\overline{\widehat{g}_N(k)} \\
&\quad + \sum_{(k_1,k_2,k_3,k)\in \Z^4 \atop \tiny{k_1+k_2+k_3-k=0 }} (ik)\eta(k_1,k_2,k_3)\phi_M(k_1+k_2) \widehat{f}_1(k_1)\widehat{f}_2(k_2)\widehat{g}_N(k_3)  \overline{\widehat{g}_N(k)} \\
& =: I_1+I_2 \; .
\end{split}
\end{equation}
We notice that $ I_1 $ can be rewritten as
$$
I_1= M \int_{\mathbb T} \Pi_{\eta \eta_1,M}^3(f_1,f_2,P_{\sim N}g)P_{N}g\, dx
$$
with
\begin{displaymath}
\eta_1(k_1,k_2,k_3)=i \frac{\phi_N(k_1+k_2+k_3)-\phi_N(k_3)}M (k_1+k_2+k_3)\, \indi_{\supp \phi_M}(k_1+k_2) \,
\end{displaymath}
and that it clearly follows from  the mean value theorem and the frequency localization that
$\eta_1 $ is
uniformly bounded in $M$ and $N$.

Next, we deal with $ I_2$. Since $g$ is real-valued, we have $\overline{\widehat{g_N}(k)}=\widehat{g_N}(-k)$, so that
\begin{displaymath}
I_2=\sum_{(k_1,k_2,k_3,k)\in \Z^4 \atop \tiny{k_1+k_2+k_3-k=0 }}(ik)\eta(k_1,k_2,k_3)\phi_M(k_1+k_2)\widehat{f_1}(k_1)\widehat{f_2}(k_2)\overline{\widehat{g_{N}}(-k_3)}\widehat{g_N}(-k) \, .
\end{displaymath}
Performing the change of variables  $(\hat{k}_3,\hat{k}) =(-k,-k_3) $ so that $\hat{k}=k_1+k_2+\hat{k}_3$, we get
\begin{displaymath}
I_2= \sum_{(k,k_1,\hat{k}_2,\hat{k}_3)\in \Z^4 \atop k_1+k_2+\hat{k}_3-\hat{k}=0}(-i\hat{k}_3)\eta(k_1,k_2,-\hat{k})\phi_M(k_1+k_2)\widehat{f_1}(k_1)\widehat{f_2}(k_2) \overline{\widehat{g_N}(\hat{k})}\widehat{g_{N}}(\hat{k}_3)\, .
\end{displaymath}
Now, observe that $|k_1-\hat{k}|=|k_2+\hat{k}_3|$ and $|k_2-\hat{k}|=|k_1+\hat{k}_3|$. Thus, according to \eqref{symmetry}
$$
\eta(k_1,k_2,-\hat{k})=\eta(k_2,k_1,-\hat{k})=\Lambda(k_1,k_2)\Theta(|k_2+\hat{k}_3|, |k_1+\hat{k}_3|,|k_1+k_2|)=\eta(k_1,k_2,\hat{k}_3)
$$
so that
\begin{displaymath}
I_2=\sum_{(\hat{k},k_1,k_2,\hat{k}_3)\in \Z^4 \atop k_1+k_2+\hat{k}_3-\hat{k}=0} i(k_1+k_2)\eta(k_1,k_2,\hat{k}_3)\phi_M(k_1+k_2)\widehat{f_1}(k_1)\widehat{f_2}(k_2) \overline{\widehat{g_N}(\hat{k})}\widehat{g_{N}}(\hat{k}_3)-I_2
\end{displaymath}
Setting
$$
\tilde{\eta}_1(k_1,k_2,k_3)= \frac{i(k_1+k_2)}{M} \indi_{\supp \phi_M}(k_1+k_2) \, \phi_N(k_3),
$$
this leads to
\begin{equation} \label{technical.pseudoproduct.4}
T_{\eta,M,N}(f_1,f_2,g,g)= M\int_{\mathbb T} \Pi_{\eta \eta_2,M}^3(f_1,f_2,P_{\sim N}g)P_Ng\, dx
\end{equation}
where $\eta_2= \eta_1+ \frac{1}{2} \tilde{\eta}_1 $ is also uniformly bounded function in $M$ and $N$ and completes the proof of the Lemma.
\end{proof}

The following proposition gives suitable estimates for the pseudo-products $\Pi_M$.
\begin{proposition} \label{pseudoproduct}
Let $M$ be a dyadic number with $M \ge 1$ and $\eta$ is a bounded mesurable function. Then for all
   $j=1,2,3$  and all $ f_i \in L^2(\T) $ it holds that
\begin{equation} \label{pseudoproduct.1}
\Big|\int_{\mathbb T}\Pi_{\eta,M}^j(f_1,f_2,f_3)f_4dx \Big| \lesssim M\prod_{i=1}^4\|f_i\|_{L^2} \, ,
\end{equation}
where  the implicit constant only depends on the $L^{\infty}$-norm of the function $\eta$.
\end{proposition}

\begin{proof} By symmetry we can assume that $ j=3$.
 Since the norms in the right-hand side only see the size of the modulus of the Fourier transform, we can assume that all the functions have non negative Fourier transform. By using Plancherel's formula, H\"older and Bernstein inequalities, we get that
\begin{equation*}
\begin{split}
\Big|\int_{\mathbb T}&\Pi_{\eta,M}^3(f_1,f_2,f_3)f_4dx \Big|   \\ & \lesssim
\sum_{(k_1,k_2,k)\in \Z^3}\phi_M(k_1+k_2)\widehat{f}_1(k_1)\widehat{f}_2(k_2)\widehat{f}_3(k-k_1-k_2)\widehat{f}_4(k)dk_1dk_2dk \\
&  =   \int_{\T} P_M(f_1 f_2) f_3 f_4 \, dx \\
 &  \lesssim \| P_{M} (f_1 f_2) \|_{L^\infty_x} \|f_3\|_{L^2_x} \|f_4 \|_{L^2_x} \\
  &\lesssim   M \prod_{i=1}^4\|f_i\|_{L^2_x} \; .
\end{split}
\end{equation*}
\end{proof}

Before stating the main result of this subsection,
let us define  the resonance function of order $3$ by
\begin{align}
\Omega_3(k_1,k_2,k_3) &= k_1^3+k_2^3+k_3^3-(k_1+k_2+k_3)^3 \notag\\
&= -3(k_1+k_2)(k_1+k_3)(k_2+k_3) \, \label{res3} \, .
\end{align}
 \begin{proposition} \label{L2trilin}
Assume that $0<T \le 1$, $\eta$ is a bounded mesurable function and $u_i$ are functions in $Z^0_T$ where $ Z^0 $ is defined in \eqref{qq}.  Assume also that $N \ge 2^8$,
  $N_1,\,N_2,\, N_3\ge 1 $, $M\ge 1$  and  $j\in\{1,2,3\}$.
We define
\begin{equation} \label{L2trilin.1}
G_{\eta,M}^T(u_1,u_2,u_3,u_4):= \int_{]0,T[\times  \mathbb \T}  \Pi_{\eta,M}^j(u_1,u_2,u_3)   u_4dxdt \, .
\end{equation}
Then for any $ K\gg 1 $  it holds
\begin{equation} \label{L2trilin.2}
\big| G_{\eta \indi_{|\Omega_3|\ge K},M}^T(P_{\lesssim N_1}u_1,P_{\lesssim N_2}u_2,P_{\lesssim N_3} u_3,P_N u_4) \big| \lesssim T^{\frac18}
 M \frac{ N_{max}^{\frac{11}{10}}}{K}\prod_{i=1}^4\|u_i\|_{Z^{0}_T} \, ,
\end{equation}
where $ N_{max}=\max(N_1,N_2,N_3)$.
Moreover, the implicit constant in estimates \eqref{L2trilin.2} only depend on the $L^{\infty}$-norm of the function $\eta$.
 \end{proposition}

\begin{remark}
Sometimes, when there is no risk of confusion, we also denote $$G_{M}^T(u_1,u_2,u_3,u_4)=G_{\eta,M}^T(u_1,u_2,u_3,u_4) \, .$$
\end{remark}

To prove Proposition \ref{L2trilin}, we need the following technical lemmas derived in \cite{MoVe}.
For any $0<T \le 1$, let us denote by $1_T$ the characteristic function of the interval $]0,T[$. One of the main difficulty in the proof of Proposition \ref{L2trilin} is that the operator of multiplication by $1_T$ does not commute with $Q_L$. To handle this situation, we follow the arguments introduced in \cite{MoVe} and use the decomposition
\begin{equation} \label{1T}
1_T=1_{T,R}^{low}+1_{T,R}^{high}, \quad \text{with} \quad \mathcal{F}_t\big(1_{T,R}^{low} \big)(\tau)=\chi(\tau/R)\mathcal{F}_t\big(1_{T} \big)(\tau) \, ,
\end{equation}
for some $R>0$ to be fixed later.
\begin{lemma} \label{QL}
 Let $L$ be a nonhomogeneous dyadic number. Then the operator $Q_{\le L}$ is bounded in $L^{\infty}_tL^2_x$ uniformly in $L$. In other words,
\begin{equation} \label{QL.1}
\|Q_{\le L}u\|_{L^{\infty}_tL^2_x} \lesssim \|u\|_{L^{\infty}_tL^2_x} \, ,
\end{equation}
for all $u \in L^{\infty}_tL^2_x$ and the implicit constant appearing in \eqref{QL.1} does not depend on $L$.
 \end{lemma}

\begin{lemma}\label{ihigh-lem} For any $ R>0 $ and $ T>0 $ it holds
\begin{equation}\label{high}
\|1_{T,R}^{high}\|_{L^1}\lesssim T\wedge R^{-1}
\end{equation}
and
\begin{equation}\label{low}
 \|1_{T,R}^{high}\|_{L^\infty}+ \|1_{T,R}^{low}\|_{L^\infty}\lesssim  1.
\end{equation}
\end{lemma}

\begin{lemma}\label{ilow-lem}
Assume that $T>0$, $R>0$  and $ L \gg R $. Then, it holds
\begin{equation} \label{ihigh-lem.1}
\|Q_L (1_{T,R}^{low}u)\|_{L^2}\lesssim \|Q_{\sim L} u\|_{L^2} \, ,
\end{equation}
for all $u \in L^2(\mathbb R^2)$.
\end{lemma}

\begin{proof}[Proof of Proposition \ref{L2trilin}]
 We take the extensions $ \tilde{u}_i =\rho_T(u_i) $ of $ u_i $ defined in \eqref{defext}  and to  lighten the notations
 we denote $ P_{\lesssim N_i} \tilde{u}_i $ by $\check u_i $, $ i=1,2,3$ and $ P_N \tilde{u}_4 $ by $ \check{u}_4 $.
 We  first notice that \eqref{res3} ensures that for  $ K > 2^8   N_{max}^3  $, the set $ |\Omega_3|\ge K $ is empty and thus \eqref{L2trilin.2} holds  trivially. We can thus
  assume that $ K\le  2^8   N_{max}^3  $. We set
  \begin{equation}\label{fc}
   R=( K/N_{max}^{11/10})^{8/7}= K \Bigl( \frac{K}{N_{max}^\frac{44}{5}}\Bigr)^{1/7} \le K  \Bigl(2^8 N_{max}^{-\frac{29}{5}}\Bigr)^{1/7} \ll K\, ,
  \end{equation}
  since $ N\ge 2^8 $ ensures that $ N_{max} \ge 2^{5} $,
    and decompose $ 1_T $ as $ 1_T = 1_{T,R}^{high}+1_{T,R}^{low} $.
The contribution of the first term is easily estimated thanks to
 H\"older's inequality in time, \eqref{pseudoproduct.1}, \eqref{high}  and \eqref{low} by
\begin{displaymath}
\begin{split}
\big|G_{\eta\indi_{|\Omega_3|\ge K},M}(1_{T,R}^{high} & \check{u}_1,\check{u}_2,\check{u}_3,\check{u}_4)\big|  \\
& \le T^{1/8} \|1_{T,R}^{high}\|_{L^{8/7}}\big\|\int_{\mathbb T}\Pi_M^j
( \check{u}_1,\check{u}_2,\check{u}_3)\check{u}_4dx\big\|_{L^{\infty}_t}
\\ & \lesssim  T^{1/8} R^{-7/8}M\prod_{i=1}^4\|\check{u}_i \|_{L^{\infty}_tL^2_x}  \, ,\\
& \lesssim  T^{1/8} \frac{M N_{max}^{\frac{11}{10}}}{K} \prod_{i=1}^4\|u_i \|_{L^{\infty}_T L^2_x} \, .
\end{split}
\end{displaymath}
To deal with the contribution of $ 1_{T,R}^{low} $, we first note \eqref{fc} ensures that we are in the hypotheses of Lemma \ref{ilow-lem}. Now,  by the definition of $ \Omega_3$, we may decompose the contribution of  $ 1_{T,R}^{low} $ as
\begin{eqnarray}
G_{\eta \indi_{|\Omega_3|\ge K},M}(_{T,R}^{low} \check{u}_1,\check{u}_2,\check{u}_3,\check{u}_4)&= &G_{\eta  \indi_{|\Omega_3|\ge K},M}(Q_{\gtrsim K}(1_{T,R}^{low} \check{u}_1),\check{u}_2,\check{u}_3,\check{u}_4)\nonumber \\
& & +G_{\eta \indi_{|\Omega_3|\ge K},M}(Q_{\ll  K}(1_{T,R}^{low} \check{u}_1), Q_{\gtrsim K} u_2,u_3,\check{u}_4) \nonumber \\
&  & +G_{\eta  \indi_{|\Omega_3|\ge K},M}(Q_{\ll  K}(1_{T,R}^{low} \check{u}_1), Q_{\ll K} \check{u}_2,  Q_{\gtrsim K} \check{u}_3,\check{u}_4) \nonumber \\
& &  +G_{\eta  \indi_{|\Omega_3|\ge K},M}(Q_{\ll K}(1_{T,R}^{low} \check{u}_1), Q_{\ll K} \check{u}_2,  Q_{\ll K} \check{u}_3,Q_{\gtrsim K} \check{u}_4)  \, .\label{TG}
\end{eqnarray}
The first term of the right-hand side of the above equality can be estimated thanks to \eqref{pseudoproduct.1} and  \eqref{ihigh-lem.1} by
\begin{displaymath}
\begin{split}
\big| G_{\eta  \indi_{|\Omega_3|\ge K},M}(Q_{\gtrsim K}(1_{T,R}^{low}) \check{u}_1),\check{u}_2,\check{u}_3,\check{u}_4) \big|
 &\lesssim T^{1/2}M  \|Q_{\gtrsim K}(1^{low}_{T,R}\check{u}_1)\|_{L^2_{x,t}}
\prod_{j=2}^4 \|\check{u}_j \|_{L^\infty_T L^2_x}  \\
& \lesssim T^{1/2} \frac{M N_{max}^{\frac{11}{10}}}{K} \|\check{u}_1\|_{X^{-\frac{11}{10},1}}\prod_{i=2}^{4}\ \|\check{u}_j \|_{L^\infty_t L^2_x} \\
& \lesssim T^{1/2} \frac{M N_{max}^{\frac{11}{10}}}{K} \prod_{i=1}^4 \|u_i\|_{Z^0_T} \; .
\end{split}
\end{displaymath}
The other terms can be controlled in exactly the same way. Note that  we use \eqref{QL.1} and not  \eqref{ihigh-lem.1} for these terms. \\
\end{proof}
\subsection{Uniform estimates on  solutions}
The preceding subsection enables us  to easily get an uniform $ H^s $-bound for solutions to \eqref{remKdV}. This is the aim of this subsection where we do not attempt to get the lowest propagated regularity since we will be forced to take $ s\ge 1/3 $ in the estimate on  the difference.

We first prove refined Strichartz estimates. The following linear estimate  in Bourgain's space  is established in \cite{Bourgain1993},
  $$
   \| u\|_{L^4(]0,1[\times \T)} \lesssim \| u\|_{X^{0,1/3}} \, \quad \forall  u\in X^{0,1/3} \; .
$$
 We will make use of a Strichartz estimate which follows directly from the above estimate (see for instance \cite{ginibre}),
 \begin{equation} \label{strichartz2}
 \| U(t) \varphi\|_{L^4(]0,T[\times \T)} \lesssim T^{\frac{1}{6}} \| \varphi \|_{L^2} \, , \quad\forall  T\in ]0,1[, \; \quad \forall \varphi\in L^2(\T)\; ,
 \end{equation}
 where the implicit constant does not depend on $ T$.
\begin{lemma} \label{refinedStrichartz}
Let $0<T<1 $ and  let $u\in L^\infty(0,T; H^s(\T)) $ be  a solution to \eqref{remKdV} emanating from $ u_0\in H^s(\T) $.
For $ s>\frac{11}{35} $ it holds
\begin{equation}\label{est1L6}
\|u\|_{L^4_T L^{20}_x } \lesssim \|u\|_{L^\infty_T H^{s}_x}(1+ \|u\|_{L^\infty_T H^s_x}^2)
\end{equation}
and for $ s>\frac{9}{28} $,
\begin{equation}\label{est2L6}
\|D_x^\frac{5}{24}u\|_{L^4_T L^4_x  } \lesssim \|u\|_{L^\infty_T H^{s}}(1+ \|u\|_{L^\infty_T H^s}^2) \; .
\end{equation}

\end{lemma}
\begin{proof} Let $u$ be a solution to \eqref{remKdV} defined on a time interval $[0,T]$. We use a nonhomogeneous Littlewood-Paley
decomposition, $u=\sum_Nu_N$ where $u_N=P_Nu$ and  $N$ is a dyadic integer.  Since, \eqref{est1L6} and \eqref{est2L6} are obvious for $ N \lesssim 1$, it
suffices to control $ \| u_N\|_{L^4_T L^{20}_x} $ for any $ N \gg 1 $.
For such $ N$,    Sobolev and Bernstein inequalities lead to
\begin{displaymath}
\| u_N\|_{L^4_T L^{20}_x}  \lesssim     N^{\frac{1}{5}} \|u_N\|_{L^4_{Tx}} \; .
\end{displaymath}
Let $\delta$ be a nonnegative number to be fixed later. we chop out the interval in small intervals of length $N^{-\delta}$. In other words, we have that $[0,T]=\underset{j \in J}{\bigcup}I_j$ where $I_j=[a_j, b_j]$, $|I_j|\thicksim N^{-\delta}$ and $\# J\sim N^{\delta}$.
Since $u_N$ is a solution to the integral equation
\begin{displaymath}
u_N(t) =e^{-(t-a_j)\partial_x^3}u_N(a_j)+\int_{a_j}^te^{-(t-t')\partial_x^3}P_N \partial_x(u^3-3P_0(u^2) u)(t')dt'
\end{displaymath}
for $t \in I_j$, we deduce from \eqref{strichartz2} that
\begin{equation}\label{2se}
\begin{split}
&\|u_N\|_{L^4_T L^{20}_x} \\ &\lesssim \Big(\sum_j \|D^{-\frac{\delta}{6}+\frac{1}{5}}_x u_N(a_j)\|_{L^2_x}^4 \Big)^{\frac14}+
\Big(\sum_j \big(\int_{I_j}\|D^{-\frac{\delta}{6}+\frac{6}{5}}_x P_N(u^3-3P_0(u^2) u)(t')\|_{L^2_x}dt'\big)^4 \Big)^{\frac14} \\
& \lesssim N^{\frac{\delta}4}\|D^{-\frac{\delta}{6}+\frac{1}{5}}_x  u_N\|_{L^{\infty}_T L^2_x}
+\Big(\sum_j |I_j|^3\int_{I_j}\|D^{-\frac{\delta}{6}+\frac{6}{5}}_x P_N(u^3-3P_0(u^2) u)(t')\|_{L^2_x}^4dt' \Big)^{\frac14} \\
& \lesssim N^{0-} \Bigl[\|D^{\frac \delta {12} +\frac{1}{5}+}_xu_N\|_{L^{\infty}_TL^2_x}+\|D^{-\frac{11}{12} \delta+\frac{6}{5}+}_x P_N(u^3-3P_0(u^2) u)\|_{L^4_TL^2_x}\Bigr] \, .
\end{split}
\end{equation}
Now, to prove \eqref{est1L6} we notice that
 Sobolev's inequalities and the fractional Leibniz rule lead to
 \begin{eqnarray}
 \|D^{-\frac{11}{12} \delta+\frac{6}{5}+}_x P_N(u^3-3P_0(u^2) u)\|_{L^4_TL^2_x}
  &\lesssim &  \|D^{-\frac{11\delta}{12}+1/p+\frac{7}{10}+}_x P_N(u^3-3P_0(u^2) u)\|_{L^4_TL^
 p_x} \nonumber \\
 & \lesssim & \|u\|_{L^{\infty}_TL^{q}_x}^2 \|D^{\kappa+}_xu\|_{L^{\infty}_TL^2_x} \, .\label{bb2}
 \end{eqnarray}
for all $1 \le p \le 2$   and $2 \le q \le \infty$ satisfying
 $\frac2{q}+\frac12=\frac1p$ and $0<\kappa=-\frac{11}{12}\delta+\frac{7}{10}+\frac1p<1$. Thus,  the Sobolev embedding yields
\begin{equation} \label{se2}
 \|D^{-\frac{11}{12} \delta+\frac{6}{5}+}_x P_N(u^3-3P_0(u^2) u)\|_{L^4_TL^2_x}  \lesssim \|u\|_{L^{\infty}_TH^{\kappa+}_x}^3 \, ,
\end{equation}
if we choose $\kappa$ satisfying $\kappa=\frac12-\frac1{q}=\frac34-\frac1{2p}$. This implies that
\begin{equation} \label{se3}
\kappa=-\frac{11}{12}\delta+\frac{7}{10}+\frac1p=-\frac{11}{12}\delta+\frac{11}{5}-2\kappa \quad \Rightarrow \quad \kappa=-\frac{11}{36}\delta
+\frac{11}{15} \, .
\end{equation}
Then, we choose $\delta$ such that $\frac{\delta}{12} + \frac{1}{5} =\kappa $ which leads to
$$
\delta=\frac{48}{35}, \quad \kappa=\frac{11}{35}, \quad p=\frac{70}{61} \quad \text{and} \quad q=\frac{70}{13}\, .
$$
Therefore, we conclude gathering \eqref{2se}--\eqref{se3} that if $u$ is a solution to \eqref{remKdV}-\eqref{ini} defined on the time interval $[0,T]$, then
 for $ N\gg 1$,
\begin{equation} \label{se4}
\|P_Nu\|_{L^4_T L^{20}_x}  \lesssim N^{0-} \Bigl[\|D_x^{\frac{11}{35}+}P_Nu_0\|_{L^2_x}+\|D_x^{\frac{11}{35}+}u\|_{L^\infty_TL^{2}_x}^3\Bigr]\; ,
\end{equation}
which proves \eqref{est1L6} with $s>\frac{11}{35}$, by summing over $ N $.

To prove \eqref{est2L6} we proceed in the same way. We   eventually obtain for $ N\gg 1$,
$$
\|D_x^\frac{5}{24}u_N\|_{L^4_T L^{4}_x}
 \lesssim \|D^{\frac \delta {12} +\frac{5}{24}+}_x u_N\|_{L^{\infty}_TL^2_x}+\|D^{-\frac{11}{12} \delta+\frac{29}{24}+}_x P_N(u^3-3P_0(u^2) u)\|_{L^4_TL^2_x}
$$
with
$$
 \|D^{-\frac{11}{12} \delta+\frac{29}{24}+}_x  P_N( u^3-3P_0(u^2)u)\|_{L^4_TL^2_x}\lesssim  \|u\|_{L^{q}_x}^2 \|D^{\kappa+}_xu\|_{L^{\infty}_TL^2_x} \, ,
 $$
for all $1 \le p \le 2$   and $2 \le q \le \infty$ satisfying
 $\frac2{q}+\frac12=\frac1p$ and $0<\kappa=-\frac{11}{12}\delta+\frac{17}{24}+\frac1p<1$. Thus,  the Sobolev embedding yields
$$
 \|D^{-\frac{11}{12} \delta+\frac{29}{24}+}_x P_N(u^3-3P_0(u^2) u)\|_{L^4_TL^2_x}  \lesssim \|u\|_{L^{\infty}_TH^{\kappa+}_x}^3 \, ,
$$
if we choose $\kappa$ satisfying $\kappa=\frac12-\frac1{q}=\frac34-\frac1{2p}$. This implies that
$$
\kappa=-\frac{11}{12}\delta+\frac{17}{24}+\frac1p=-\frac{11}{12}\delta+\frac{53}{24}-2\kappa \quad \Rightarrow \quad
 \kappa=-\frac{11}{36}\delta+\frac{53}{72} \, .
$$
Then,  choosing $\delta$ such that $\frac{\delta}{12} + \frac{5}{24} =\kappa $ which leads to
$$
\delta=\frac{19}{14}, \quad \kappa=\frac{9}{28}  \quad p=\frac{7}{6} \quad \text{and} \quad q=\frac{28}{5}\, .
$$
 we obtain  \eqref{est2L6} with $s=\frac{9}{28}+$.
\end{proof}
\begin{lemma} \label{trilin}
Assume that $0<T \le 1$, $ s>\frac{11}{35}$  and $u\in L^\infty(0,T; H^s(\T))$ is a  solution to \eqref{remKdV} emanating from $ u_0\in H^{s}(\T) $. Then,
\begin{equation} \label{trilin.1}
\|u\|_{Z^{s}_T} \lesssim \|u\|_{L^\infty_T H^s_x} + \|u\|_{L^\infty_T H^s_x}^3 (1+\|u\|_{L^\infty_T H^s_x})^4\,
\end{equation}
\end{lemma}

\begin{proof}
By using Lemma \ref{extension}, it is clear that we only have to estimate the $ X^{s-\frac{11}{10},1}_T $-norm of $ u  $ to prove \eqref{trilin.1}.
Now, using the Duhamel formula associated to \eqref{remKdV}, the standard linear estimates in Bourgain's spaces and the fractional Leibniz rule (\textit{c.f.} Theorem A.12 in  \cite{KPV2}), we have that
\begin{equation} \label{be}
\begin{split}
\|u\|_{X^{s-\frac{11}{10},1}_T} &\lesssim \|u_0\|_{H^{s-\frac{11}{10}}}+ \|\partial_x(u^3)\|_{X^{s-\frac{11}{10},0}_T}+ \| P_0(u^2) u_x\|_{X^{s-\frac{11}{10},0}_T} \\
&  \lesssim \|u_0\|_{H^{s-\frac{11}{10}}}+\|J_x^{s-\frac{1}{10}} (u^3)\|_{L^2_T L^2_x} +\|u\|_{L^\infty_T L^2_x}^2 \|u\|_{L^2_T H^{s-\frac{1}{10}}_x} \\
 & \lesssim \|u\|_{L^\infty_T H^{s}}+ \|J_x^{s} (u^3) \|_{L^2_T L^{\frac{5}{3}}_x}+\|u\|_{L^\infty_T L^2_x}^2 \|u\|_{L^2_T H^{s}_x} \\
 & \lesssim  \|u\|_{L^\infty_T H^{s}}+ \|u\|_{L^4_T L^{20}_x}^2 \|J_x^{s} u  \|_{L^\infty_T L^2_x}+\|u\|_{L^\infty_T L^2_x}^2 \|u\|_{L^2_T H^{s}_x}
 \, ,
\end{split}
\end{equation}
which leads to \eqref{trilin.1} by using \eqref{est1L6}. \end{proof}
To prove the main result of this section we need to define some subsets of $ \Z^3 $. In the sequel  we set
\begin{equation}\label{defD}
 D=\{(k_1,k_2,k_3)\in \Z^3\, : \, (k_1+k_2)(k_1+k_3)(k_2+k_3)\neq 0 \} \; ,
 \end{equation}
 and
 \begin{equation}\label{defA}
 \begin{split}
D^{1}&=\{(k_1,k_2,k_3)\in D \,  : \, \displaystyle \underset{1\le i \neq j\le 3}{\text{med}}(|k_i+k_j|)\lesssim  2^{-9} |k_1+k_2+k_3| \} \; ,\\
 D^{2}&=D\backslash D^{1} \; .
 \end{split}
\end{equation}

To  bound from below  the resonance function $ |\Omega_3| $ (see \eqref{res3}) on $ D^1  $ and $ D^2 $ we will make a frequent use of the following lemma.
\begin{lemma}\label{inter} On $ D^{1}$, it holds
$$ |k_1|\sim |k_2| \sim |k_3| \sim |k|\; \mbox{ and }\; \max_{1\le i \neq j\le 3}(|k_i+k_j|)\gtrsim |k| \, , $$
 where $|k|=|k_1+k_2+k_3|$, whereas  on $ D^{2}$ it holds
 \begin{displaymath}
\displaystyle  \underset{1\le i \neq j\le 3}{\emph{med}}(|k_i+k_j|)\gtrsim \max_{1\le i\le 3}|k_i| \, .
 \end{displaymath}
\end{lemma}
\begin{proof}
To prove the  first assertion, we assume without loss of generality that $|k_2+k_3| \ge |k_1+k_3|\ge |k_1+k_2| $. On  $ D^{1} $, this forces $|k_2|\sim |k_3|\sim |k_1+k_2+k_3| $.  On one hand $ |k_1|\ll | k_1+k_2+k_3| $ would imply $ |k_1+k_3| \sim  | k_1+k_2+k_3| $ which can not hold. On the other hand,
 $  |k_1|\gg | k_1+k_2+k_3| $ would imply $ \max(|k_2|,|k_3|) \sim |k_1|\gg |k_1+k_2+k_3|$ which is in contradiction with the preceding deduction. Therefore $ |k_1|\sim |k_1+k_2+k_3| $. Finally, either  $ k_2 k_3\ge 0 $ and then  $ |k_2+k_3| \gtrsim |k_1+k_2+k_3|  $ or $ k_2 k_3< 0 $ and then  $\max(|k_1+k_2|, |k_1+k_3|) \gtrsim |k_1+k_2+k_3| $.

 To prove the second assertion, we first notice that this assertion is trivial when $ \displaystyle \max_{1\le i\le 3}|k_i|\sim  |k|  $ where $k=k_1+k_2+k_3$. We thus can assume that
   $ \displaystyle \max_{1\le i\le 3}|k_i|\gg |k| $.  By symmetry, we can assume that $ |k_1|\ge |k_2|\ge |k_3|$.
    This forces  $ |k_2| \sim |k_1| \gg |k| $. Therefore $ |k_1+k_3|=|k-k_2|\sim |k_1|$ and $  |k_2+k_3| = |k-k_1| \sim |k_1| $.
\end{proof}
Let us  also set
\begin{equation} \label{m}
m_{min}=\min_{1 \le i \neq j \le 3} |k_i+k_j| \,
\end{equation}
and
\begin{equation} \label{m2}
\begin{split}
&A_1=\big\{(k_1,k_2,k_3) \in  \Z^3 \, : \, |k_2+k_3|=m_{min} \big\} \, , \\
&A_2=\big\{(k_1,k_2,k_3) \in  \Z^3/A_1\, : \, |k_1+k_3|=m_{min} \big\} \, , \\
&A_3=\big\{(k_1,k_2,k_3) \in \Z^3/(A_1\cup A_2) \, : \, |k_1+k_2|=m_{min} \big\}= \Z^3/(A_1\cup A_2) \, .
\end{split}
\end{equation}

Then, it is clear from the definition of those sets that
\begin{equation} \label{m3}
\sum_{j=1}^3 \indi_{A_j}(k_1,k_2,k_3)=1,\;\forall (k_1,k_2,k_3)\in\Z^3 \, .
\end{equation}
However, we will not work directly with this partition  because of a lack of symmetry.
Note that $ \chi_{A_1}(k_1,k_2,k_3)=\chi_{A_1}(k_1,k_3,k_2) $ and $ \chi_{A_3}(k_1,k_2,k_3)=\chi_{A_3}(k_2,k_1,k_3)$ but $ \chi_{A_2}(k_1,k_2,k_3)\neq \chi_{A_2}(k_3,k_2,k_1) $ on some set of $ \Z^3$. Therefore, in order to apply Lemma \ref{technical.pseudoproduct}, we  have to  symmetrize this partition\footnote{Note  that one does not need such symmetrization on the real line since the sets $ |k_i+k_j|=|k_{i'}+k_{j'}| $
 with $ (i,j)\neq (i',j') $ are  of measure zero}. For this we set
 $$
 \indi_{\tilde{A}_j}(k_1,k_2,k_3)=\indi_{A_j}(k_3,k_2,k_1) \; .
 $$
 Then, for all $  (k_1,k_2,k_3)\in\Z^3$,
  \begin{eqnarray}
1&=& \frac{1}{2} \sum_{j=1}^3\Bigl(  \indi_{A_j}+ \indi_{\tilde{A}_j}\Bigr)(k_1,k_2,k_3)=\sum_{j=1}^3  \Theta_j(k_1,k_2,k_3)   \label{m4}
\end{eqnarray}
 where
 \begin{equation}\label{defTheta}
 \Theta_1= \frac{1}{2}( \indi_{A_1}+ \indi_{\tilde{A}_3}),\quad \Theta_2= \frac{1}{2}( \indi_{A_2}+ \indi_{\tilde{A}_2})\; \text{ and }
  \Theta_3=  \frac{1}{2}(\indi_{A_3}+ \indi_{\tilde{A}_1} )\; .
 \end{equation}
 Note that this new partition satisfies the following symmetry property : for all $ j\in \{1,2,3\}$ and  $  (k_1,k_2,k_3)\in\Z^3$,
 \begin{equation}
 \Theta_j(k_1,k_2,k_3)=\Theta_j(k_{\sigma_j(1)},k_{\sigma_j(2)},k_{\sigma_j(3)})
 \end{equation}
 where $ \sigma_j\in {\mathcal S}_3 $ is defined by $ \sigma_j(j)=j $ and $ \sigma_j(i)\neq i $ for $ i\neq j $.\\

 We are now in position to prove the main result of this section.
\begin{proposition} \label{ee}
Assume that   $0<T \le 1 $, $ s>\frac{3}{10} $  and that $u\in Z^{s}_{T} $ is a solution to \eqref{remKdV}  with initial data $ u_0\in H^{s}(\T) $. Then,
\begin{equation} \label{ee.0}
\| u\|_{L^\infty_T H^s_x}^2  \lesssim \|u_0\|_{H^s}^2 + T^{\frac18}  \|u\|_{Z^{s}_T}^4\, .
\end{equation}
\end{proposition}

\begin{proof}

By using \eqref{remKdV}, we have
\begin{displaymath}
\frac12\frac{d}{dt}\|P_Nu(\cdot,t)\|_{L^2_x}^2 =-\Re \Bigl[  \int_{\mathbb T} \partial_x P_N\Bigl(u^3-3P_0(u^2)u \Bigr)P_Nu \, dx\Bigr] \, .
\end{displaymath}
which yields after integration in time between $0$ and $t$ and summation over $N$
\begin{equation} \label{ee.3}
 \sum_{N}\|P_N u(t)\|_{H^{s}_x}^2 \lesssim  \| u_0\|_{H^{s}}^2 + \sum_{N} N^{2s} \Bigl| \Re\Bigl[  \int_{\mathbb T \times [0,t]} \partial_x P_N\Bigl( u^3-3P_0(u^2)u\Bigr)P_Nudx ds\Bigr] \Bigr|\, .
\end{equation}

In the case where $N \lesssim 1$,  we easily get
\begin{align}
\Bigl|\sum_{N} N^{2s} \int_{\mathbb T \times [0,t]} \partial_x P_N\Bigl( u^3-3P_0(u^2)u\Bigr)P_Nu\, dx ds\Bigr| \,
& \lesssim \|u\|_{L^3_T L^3_x}^3 \|P_N u \|_{L^\infty_T L^\infty_x}+ \|u\|_{L^\infty_T L^2_x}^4  \nonumber \\
  & \lesssim   \|u\|_{L^\infty_T H^\frac{1}{6}_x}^4  \lesssim \|u\|_{Z^{\frac{1}{6}}_T}^4 \, . \label{ee.4}
\end{align}

In the following, we can then assume that $N \gg 1$ and  we use the classical decomposition of $ N(u):= \partial_x (u^3-3P_0(u^2) u) $ in a resonant and a non resonant part by writing :
\begin{eqnarray*}
{\mathcal F}_{x} [N(u)](k)
& = &   ik \Bigl[ \sum_{k_1+k_2+k_3=k \atop (k_1+k_2)(k_1+k_3)(k_2+k_3) \neq 0}   \hat{u}(k_1)  \hat{u}(k_2)  \hat{u}(k_3)
\nonumber \\
& & -3    \hat{u}(k)  \hat{u}(k)  \hat{u}(-k) \Bigr]\nonumber \\
& := &    \Bigl( {\mathcal F}_{x}\Bigl[ A(u,u,u)\Bigr](k) - {\mathcal F}_{x} \Bigr[B(u,u,u)\Bigr] (k) \Bigr)  \,
\end{eqnarray*}
i.e.
\begin{equation}\label{defAB}
\partial_x (u^3-3P_0(u^2)u)=\partial_x\Bigl(A(u,u,u)-B(u,u,u)\Bigr)  \; .
\end{equation}
Now, we notice that, since $ u $ is real-valued, we have
\begin{equation}\label{eee.3}
\int_{\T}\partial_x P_N B(u,u,u) P_N u = ik \sum_{k\in \Z} |\hat{u}(k)|^2 |\varphi_N(k) \hat{u}(k)|^2 \in i \R \; .
\end{equation}
Therefore \eqref{ee.3} and \eqref{defAB} lead to
\begin{eqnarray*}
\sum_{N\gg 1} \|P_N u(t)\|_{H^{s}_x}^2 &\lesssim & \|u_0\|_{H^{s}}^2 +\Bigl|\Re \Bigl[ \sum_{N\gg 1} N^{2s} \int_{\mathbb T \times [0,t]} \partial_x P_N(A(u,u,u))P_Nu\, dx ds
\Bigr] \Bigr|  \nonumber \\
&\lesssim &  \|u_0\|_{H^{s}}^2 +\sum_{N\gg 1}J_{t,N}(u)\, .
\end{eqnarray*}
By using the decomposition in \eqref{m4}, we get that $J_{t,N}(u)=\Bigl| \Re \Bigl( \sum_{l=1}^3J_{t,N}^l(u)\Bigr) \Bigr| $ with
\begin{equation}\label{qa}
J_{t,N}^l(u)=N^{2s}  \sum_{M\ge 1}\int_{ \T \times [0,t]} \partial_x P_N \Pi_{\Theta_l \indi_{D},M}^l(u,u,u) P_N u \, dxds \, ,
\end{equation}
where $ D $ is defined in \eqref{defD} and where  $ \Pi_{\eta,M}^l $ is defined in \eqref{def.pseudoproduct2}.
Thus, by symmetry, it is enough to estimate $J_{t,N}^3(u)$ that  will be still denoted $J_N^3(u)$ for  sake of simplicity.
We rewrite $ J_{N}^3 (u) $ as
\begin{equation} \label{ee.5}
\begin{split}
 N^{2s}\Bigl(\sum_{1 \le M\le N^{1/2}}& \int_{]0,t[}T_{\Theta_3 \indi_{D}, M,N}(u,u,u,u)ds +\sum_{M>N^{1/2}} \int_{]0,t[}T_{\Theta_3 \indi_{D}, M,N}(u,u,u,u)ds \Bigr)
\\&=:I_N^{low}(u)+I_N^{high}(u) \, ,
\end{split}
\end{equation}
where $T_{\eta,M,N}(u,u,u,u)$ is defined in \eqref{technical.pseudoproduct.1}.
 At this stage it worth noticing that $ \Theta_3 \indi_{D}$  satisfies the symmetry hypothesis  \eqref{symmetry} of Lemma
  \ref{technical.pseudoproduct}.\vspace{2mm} \\
\noindent {\bf $\bullet $ Estimate for $I_N^{low}(u)$.}
According to \eqref{technical.pseudoproduct.2} we have
\begin{eqnarray*}
I_N^{low,1}(u)  &= & \sum_{1\le M \le N^{\frac12}}M N^{2s}\int_{\mathbb T \times [0,t]} \Pi_{\eta_3\indi_{D},M}^3(u,u, u_{\sim N} ) P_N u \, dx ds \\
&=&  \sum_{N_1,N_2\ge 1}\;  \sum_{1\le M \le N^{\frac12}}M N^{2s}\int_{\mathbb T \times [0,t]} \Pi_{\eta_3\indi_{D},M}^3(P_{N_1}u,P_{N_2} u, u_{\sim N} ) P_N u \, dxds \, ,
\end{eqnarray*}
where $\eta_3$ is a function of $(k_1,k_2,k_3)$ whose $L^{\infty}-$norm is uniformly bounded in $N$ and $M$.
 We now separate the contributions of $ I_N^{low,1} $ of $ D^1 $ and $ I_N^{low,2} $ of $ D^2 $ to $ I_N^{low} $.\\
 On $ D^1$,  Lemma \ref{inter} and \eqref{res3} ensure that $ |k_1|\sim|k_2|\sim|k_3|\sim N $ and   $ |\Omega_3| \gtrsim M^2 N $. Therefore, \eqref{L2trilin.2} leads to
 \begin{eqnarray*}
 \big| I_N^{low,1}(u) \big|  &\lesssim &  \sum_{1\le M \le  N^{\frac12}} T^{\frac18} N^{2s}\frac{M^2 N^\frac{11}{10}}{M^2 N} N^{-4s}  \|P_{\sim N} u \|_{Z^s_T}^4 \nonumber\\
 &\lesssim &  T^{\frac18}N^{-2s+\frac{1}{10}+}  \|P_{\sim N} u \|_{Z^s_T}^4\; ,
\end{eqnarray*}
which is acceptable for $ s>\frac1{20} $.
 On $ D^2 $,   Lemma \ref{inter} and \eqref{res3} ensure that $ |\Omega_3| \gtrsim M N_{max}^2 $ and \eqref{L2trilin.2} leads to
\begin{eqnarray*}
\big|I_N^{low,2}(u)\big|  &\lesssim &  \sum_{N_1,N_2\ge 1}\;  \sum_{1\le M \le  N^{\frac12}} T^{\frac18} N^{2s}\frac{M^2 N_{max}^\frac{11}{10}}{M N^2_{max}} N^{-2s}  \|P_{\sim N} u \|_{Z^s_T}^2  \prod_{i=1}^2  N_i^{-s} \|P_{N_i} u \|_{Z^s_T}\nonumber \\
 &\lesssim &  T^{\frac18}N^{-\frac25}   \|u \|_{Z^s_T}^4\; ,
\end{eqnarray*}
which is acceptable. \\
\noindent {\bf $\bullet $ Estimate for $I_N^{high}(u)$.}
 We separate the contributions of $ I_N^{high,1} $ of $ D^1 $ and $ I_N^{high,2} $ of $ D^2 $ to $ I_N^{high} $.
 On $ D^1$,  \eqref{L2trilin.2} yields
 \begin{eqnarray*}
 \big| I_N^{high,1} (u) \big| &\lesssim &  \sum_{M >  N^{\frac12}} T^{\frac18} N^{2s}\frac{M N  N^\frac{11}{10}}{M^2 N} N^{-4s}  \|P_{\sim N} u \|_{Z^s_T}^4 \nonumber\\
 &\lesssim &  T^{\frac18}N^{-2s+\frac{6}{10}}  \|P_{\sim N} u \|_{Z^s_T}^4\; ,
\end{eqnarray*}
which is acceptable for $ s>\frac{3}{10} $.
 On $ D^2 $,   noticing that $ M>N^{1/2} $ forces $ N_1\vee N_2 \gtrsim N^{1/2}$,  \eqref{L2trilin.2} leads to
\begin{eqnarray*}
|I_N^{high,2}|  &\lesssim &  \sum_{N_1,N_2\ge 1\atop N_1\vee N_2\gtrsim N^{1/2}}\;  \sum_{ M >  N^{1/2}} T^{\frac18} N^{2s}\frac{M N  N_{max}^\frac{11}{10}}{M N^2_{max}} N^{-2s}  \|P_{\sim N} u \|_{Z^s_T}^2  \prod_{i=1}^2  N_i^{-s} \|P_{N_i} u \|_{Z^s_T}\nonumber \\
 &\lesssim &  T^{\frac18}N^{\frac{1}{10}-s/2}   \|u \|_{Z^s_T}^4\; ,
\end{eqnarray*}
which is acceptable for $ s>1/5 $.  This concludes the proof of the proposition.
\end{proof}
Combining Lemma \ref{trilin}  and Proposition \ref{ee} we can easily get  an a priori estimate on the $ H^{\frac{3}{10}}(\T)$-norm of smooth solutions to \eqref{remKdV}. This will be done in  Section  \ref{Secmaintheo}.

\section{The smoothing estimate}
The aim of this section is to prove the proposition  below that show a kind of smoothing effect first observed in \cite{TT}. This smoothing effect
 is the only way we know to treat some  resonant terms involving $ B $ (see \eqref{defAB}) when estimating the difference of two solutions. Note that, by symmetry,  the terms involving $ B $
 do cancel in the proof of the energy estimate \eqref{ee.0}.\begin{theorem}\label{theo56}
Let $ s\ge 1/3 $ be fixed. For any solution $ u \in Z^s_T $ of \eqref{remKdV}-\eqref{ini} and any $ k\ge2^9 $ it holds
\begin{equation}\label{estheo56}
\sup_{t\in ]0,T[} k\Bigl| |\widehat{u}(t,k)|^2-|\widehat{u_0}(k)|^2 \Bigr|\lesssim \sup_{N\gtrsim k}\Bigl[ \Bigl( \frac{k}{N}\Bigr)^{s-} \|P_{\le N} u\|_{Z^s_T}^4 (1+\|P_{\le   N} u \|_{Z^s_T}^4)\Bigr]
\end{equation}
where the implicit constant does not depend on $ k$.
\end{theorem}
\subsection{Notations}\label{notations}
In this section we will widely use the following notations :
$
\vec{k}_{(3)}=(k_1,k_2,k_3)\, .
$ Let $ D $, $ D^1 $ and $ D^2 $ be defined as in \eqref{defD}-\eqref{defA}. We set
$$
\Gamma_3(k)=\big\{ \vec{k}_{(3)} \in \Z^3 \, : \, k_1+k_2+k_3=k \big\} \, ,
$$
 $$
 D(k):=\Gamma_3(k)\cap D , \quad
D^1(k)= D(k)\cap D^1\mbox{ and } D^2(k)=D(k)\cap D^2 \;,
$$
$m_{1}=|k_{2}+k_{3}| $, $ m_{2}=|k_{1}+k_{3}|$, $m_{3}=|k_1+k_2|$, $m_{min}=\min(m_1,m_2,m_3) $ and
$ m_{med}=$med$(m_1,m_2,m_3) $.
$$
D_M(k)=\{\vec{k}_{(3)}\in D(k)\, : \,  m_{min} \sim M \} \text{ and } D_M^i(k)=D_M(k)\cap D^i , \quad i=1,2 \, .
$$
$M_1, \, M_2,\, M_3, \, M_{min} $ and $ M_{med} $ are the dyadic integers associated with respectively
$m_1,\, m_2,\,m_3,\,  m_{min}$ and $ m_{med} $.\\
For $ i\in \{1,2,3,4\}$, we set
$
\vec{k}_{i(3)}=(k_{i1},k_{i2},k_{i3})$,
$$
m_{i,min}=\min(|k_{i1}+k_{i2}|,|k_{i1}+k_{i3}|, |k_{i2}+k_{i3}|), \; m_{i,med}= \hbox{med} \{|k_{i1}+k_{i2}|,|k_{i1}+k_{i3}|, |k_{i2}+k_{i3}|\}
$$
$M_{i,min} $ and $ M_{i,med} $ are the dyadic numbers associated with respectively
$ m_{i,min}$ and $ m_{i,med} $.

\subsection{$L^2$-multilinear space-time estimates}
\subsubsection{$L^2$-trilinear estimates}
\begin{lemma}\label{prod4-est}
 Let $f_j\in l^2(\mathbb Z)$, $j=1,...,4 $. Then it holds
\begin{eqnarray}
\mathcal{J}^3_k &:= & \hspace*{-6mm}\sum_{\vec{k}_{(3)}\in \Gamma^3(k), \, k_4=-k} \hspace*{-6mm}\phi_M(k_1+k_2) \prod_{j=1}^4|f_j(k_j)| \nonumber \\
 & \lesssim &[(M^{\frac{1}{2}} \|f_3\|_{l^2}) \wedge
(M \|f_3\|_{l^\infty})]\|f_4\|_{l^\infty}\prod_{j=1}^2 \|f_j\|_{l^2} \, .\label{prod4-est.1}
\end{eqnarray}
\end{lemma}
\begin{proof}
 We can assume without loss of generality that $f_i \ge 0$ for $i=1,\cdots,4$. Then, we get by using H\"older and Young's inequalities that
\begin{eqnarray*}
\mathcal{J}^3_k &  \le&  |f_4(k)|  \| \phi_M (f_1\ast f_2) \|_{l^2}\|f_3 \|_{l^2} \\
& \le & M^{\frac12}\|f_1\ast f_2\|_{l^\infty} \|f_3 \|_{l^2(\Z)}\|f_4 \|_{l^\infty} \\
& \le & M^{\frac{1}{2}} \prod_{j=1}^3 \|f_j\|_{l^2} \|f_4 \|_{l^\infty} \;,
\end{eqnarray*}
and
\begin{eqnarray*}
\mathcal{J}^3_k &  \le&  |f_4(k)|  \| \phi_M (f_1\ast f_2) \|_{l^1(\Z)}\|f_3 \|_{l^\infty} \\
& \le & M \|f_1\ast f_2\|_{l^\infty} \|f_3 \|_{l^\infty}\|f_4 \|_{l^\infty} \\
& \le & M \prod_{j=1}^2  \|f_j\|_{l^2} \prod_{i=3}^4 \|f_i \|_{l^\infty} \;,
\end{eqnarray*}
which proves \eqref{prod4-est.1}.
\end{proof}
\begin{proposition} \label{L2tri}
Assume that $0<T \le 1$, $\eta:\mathbb Z^3 \to \mathbb C$ is a bounded  function and $u_i$ are functions in $Z^{0}_{T}$.  Assume also that $k\ge 2^9 $, $M\ge 1$ and  $j\in\{1,2,3\}$.
We define
\begin{equation} \label{L2tri1}
J_{\eta,M}^{3,k,T}(u_1,u_2,u_3,u_4):= \int_{[0,T]\times  \mathbb \T}  \Pi_{\eta,M}^j(u_1,u_2,u_3)   P_k u_4dxdt \, ,
\end{equation}
where $\Pi_{\eta,M}^j$ is defined in \eqref{def.pseudoproduct2}. Then for any $ K\gg 1 $  it holds
\begin{equation} \label{L2tri2}
\big| J_{\eta  \indi_{|\Omega_3|\ge K},M}^{3,k,T}(P_{\lesssim N_1} u_1,P_{\lesssim N_2}u_2, P_{\lesssim N_3}u_3,u_4) \big| \lesssim  \frac{M^{\frac12} N_{max}^{\frac{11}{10}}}{K} \prod_{i=1}^4\|u_i\|_{Z^{0}_T} \, ,
\end{equation}
where $N_{max}=\max(N_1,N_2,N_3) $.
Moreover, the implicit constant in estimate \eqref{L2tri2} only depends on the $L^{\infty}$-norm of the function $\eta$.
 \end{proposition}
\begin{proof}
Keeping in mind that $ l^2(\Z) \hookrightarrow l^\infty(\Z) $,
 we proceed exactly as in the proof of Proposition \ref{L2trilin} but with the help of \eqref{prod4-est.1} instead of \eqref{pseudoproduct.1}. \end{proof}
\subsubsection{$L^2$-quintic linear estimates}
We use the notations  $ \vec{k}_{(5)}=(k_1, k_2,.., k_6)\in \Z^6 $ and for any $ k\in \Z $,
$$
\Gamma_5(k)=\big\{ \vec{k}_{(5)} \in \Z^6 \, : \, \sum_{i=1}^6 k_i=k \big\} \, .
$$
 Before stating our quintic space-time estimates, let us   define the resonance  function of order $5$ for $\vec{k}_{(5)}=(k_1,\cdots,k_6) \in \Gamma^5(0)$ by
\begin{equation} \label{res5}
\Omega^5(\vec{k}_{(5)}) = k_1^3+k_2^3+k_3^3+k_4^3+k_5^3+k_6^3 \; .
\end{equation}
It is worth noticing that a direct calculus leads to
\begin{equation}\label{res55}
\Omega^5(\vec{k}_{(5)}) = \Omega^3(k_1,k_2,k_3) + \Omega^3(k_4,k_5,k_6) \, .
\end{equation}
In the sequel we set
$$
\vec{k}_{1(3)}=(k_{11},k_{12},k_{13}) \; .
$$

\begin{lemma}\label{prod5-est}
 Let $f_j\in l^2(\mathbb Z)$, $j=1,...,6 $. Then it holds that
\begin{equation} \label{prod4-est.2}
\mathcal{J}^{5,1}_k:=\hspace*{-1mm}\sum_{\vec{k}_{(5)}\in \Gamma^5(0) \atop k_6=-k} \hspace*{-1mm}\phi_{M}(k_1+k_2)\phi_{M'}(k_4+k_5) \prod_{j=1}^6|f_j(k_j)| \lesssim
 [(M^{\frac{1}{2}}M') \wedge (M {M'}^\frac12)]  \prod_{j=1}^6 \|f_j\|_{l^2} \,
\end{equation}
and
\begin{equation} \label{prod4-est.22}
\mathcal{J}^{5,1}_k \lesssim    \|f_6\|_{l^2} \min\Bigl( M'
\prod_{i=1}^2 \|\langle \cdot \rangle^{1/4} f_i\|_{l^2}  \prod_{j=3}^5 \|f_j\|_{l^2},  M \prod_{j=1}^3 \|f_j\|_{l^2}\prod_{i=4}^5 \|\langle \cdot \rangle^{1/4} f_i\|_{l^2} \Bigr) \, .
\end{equation}

\begin{equation} \label{prod4-est.3}
\mathcal{J}^{5,2}_k:=\hspace*{-2mm}\sum_{\vec{k}_{(5)}\in \Gamma^5(0) \atop k_6=-k} \hspace*{-2mm}\phi_{M}((k_1+k_2+k_3)+k_4)\phi_{M'}(k_1+k_2) \prod_{j=1}^6|f_j(k_j)| \lesssim
 M^{\frac{1}{2}}M'  \prod_{j=1}^6 \|f_j\|_{l^2} \, .
\end{equation}

\begin{equation} \label{prod4-est.4}
\mathcal{\tilde{J}}^5_k:=\hspace*{-4mm}\sum_{\vec{k}_{(3)}\in \Gamma^3(k)\atop  (k_4,k_5,k_6)\in \Gamma^3(-k)} \hspace*{-4mm}\phi_{M}(k_1+k_2)\phi_{M'}(k_{4}+k_{5}) \prod_{j=1}^6|f_j(k_j)|   \lesssim
 M^{\frac{1}{2}}M'^{\frac12}  \prod_{j=1}^6 \|f_j\|_{l^2} \, .
\end{equation}
\end{lemma}
\begin{proof}
 We can assume without loss of generality that $f_i \ge 0$ for $i=1,\cdots,6$.
Proceeding in the same way as in the proof of Lemma \ref{prod4-est}, we get
\begin{displaymath}
\begin{split}
\mathcal{J}^{5,1}_k &  \le  |f_6(k)| \|f_3 \|_{l^2} \\ & \ \times \min \Bigl(  \| \phi_M (f_1\ast f_2) \|_{l^2} \| \phi_{M'} (f_4\ast f_5) \|_{l^1},
   \| \phi_M (f_1\ast f_2) \|_{l^1} \| \phi_{M'} (f_4\ast f_5) \|_{l^2} \Bigr) \\
& \le  [(M^{\frac{1}{2}} M')\wedge (M {M'}^{\frac12})] \prod_{j=1}^6 \|f_j\|_{l^2} \; .
\end{split}
\end{displaymath}
In the same way,
\begin{eqnarray*}
\mathcal{J}^{5,1}_k &  \le&  |f_6(k)|\\
& & \times \min \Bigl(  \| f_1\ast f_2\ast f_3 \|_{l^\infty} \| \phi_{M'} (f_4\ast f_5) \|_{l^1},
   \| \phi_M (f_1\ast f_2) \|_{l^1} \| f_3\ast f_4\ast f_5 \|_{l^\infty } \Bigr) \\
   & \lesssim &  |f_6(k)| \min \Bigl(  \| f_1\ast f_2\|_{l^2} \| M' \prod_{i=3}^5 \|f_i\|_{l^2}\, ,\,
  \|  f_4\ast f_5 \|_{l^2} M' \prod_{i=1}^3 \|f_i\|_{l^2}\Bigr) \,
\end{eqnarray*}
which  leads to the desired result by using that
\begin{equation} \label{convol}
\| f_1\ast f_2\|_{l^2}\lesssim \prod_{i=1}^2 \|\langle \cdot\rangle^{1/4} f_i\|_{l^2} \;, .
\end{equation}
To derive \eqref{prod4-est.3}, we notice that
\begin{eqnarray*}
\mathcal{J}^{5,2}_k &  \le&  |f_6(k)| \|f_5 \|_{l^2}  \Bigl\| \Bigl[\phi_M \Bigl([\phi_{M'}(f_1\ast f_2)]\ast f_3\ast f_4\Bigr) \Bigr\|_{l^2}   \\
& \lesssim & M^{\frac12}  \|f_6\|_{l^{\infty}} \|f_5 \|_{l^2} \|f_4\|_{l^2(\Z)}\|[\phi_{M'}(f_1\ast f_2)]\ast f_3\|_{l^2}\\
& \lesssim &M^{\frac12}  \|f_6\|_{l^{\infty}} \|f_5 \|_{l^2} \|f_4\|_{l^2}\|f_3\|_{l^2}\|[\phi_{M'}(f_1\ast f_2)]\|_{l^1} \, ,
\end{eqnarray*}
which yields the desired estimate.
Finally,
\begin{eqnarray*}
\mathcal{\tilde{J}}^5_k &  \le&   \| \phi_M (f_1\ast f_2) \|_{l^2}\|f_3 \|_{l^2} \| [\phi_{M'} (f_4\ast f_5)]\ast f_6 \|_{l^\infty}\\
& \le & M^{\frac12}\|f_1\ast f_2\|_{l^\infty} \|f_3 \|_{l^2}\| \phi_{M'} (f_4\ast f_5)\|_{l^2} \|f_6\|_{l^2} \\
& \le & M^{\frac{1}{2}} M'^{\frac12} \prod_{j=1}^6 \|f_j\|_{l^2} \;.
\end{eqnarray*}
\end{proof}
\begin{proposition} \label{L2quin}
Assume that $0<T \le 1$, $\eta$ : $ \Z^5\to \C $  is a bounded function and $u_i$ are functions in $Z^{0}_{T} $.  Assume also that $k\ge 2^9 $, $M, M'\ge 1$ and  $ K\gg 1$.
We define
\begin{align*}
J_{\eta,M,M',K}^{5,k,T} & ( \vec{u}_{1(3)},u_{2},u_{3},u_{4}) \\
&:= \sum_{\vec{k}_{(3)}\in D_M(k) \atop |k_1|\gtrsim |k_2|\vee |k_3|} \sum_{\vec{k}_{1(3)}\in D_{M'}(k_1)
\atop |\Omega_5(k_{1(3)},k_2,k_3,k_4)|\gtrsim K}
\int_{[0,T]}\eta(\vec{k}_{1(3)},k_2,k_3) \prod_{j=1}^3 \widehat{u}_{1j}(k_{1j}) \prod_{i=2}^4\widehat{u_{i}}(k_{i})
\end{align*}
and
\begin{align*}
\widetilde{J}_{\eta,M,K}^{5,k,T} & (\vec{u}_{1(3)},u_{2},u_{3},u_{4}) \\
&:= \sum_{\vec{k}_{(3)}\in D^2_M(k)\atop |k_1|\gg |k_2|\vee |k_3|}  \sum_{\vec{k}_{1(3)}\in D(k_1)
\atop  |\Omega_5(\vec{k}_{1(3)},k_2,k_3,k_4)|\gtrsim K}
\int_{[0,T]}\eta(\vec{k}_{1(3)},k_2,k_3) \prod_{j=1}^3 \widehat{u}_{1j}(k_{1j}) \prod_{i=2}^4\widehat{u_{i}}(k_{i})
\end{align*}

with $ k_4=-k$ and where $ \vec{u}_{1(3)}:=(u_{11}, u_{12},u_{13})$.\\
Then
\begin{align}
\bigl| J_{\eta,M,M',K}^{5,k,T}  ( P_{N_{11}}u_{11},  & P_{N_{12}}u_{12},  P_{N_{13}}u_{13} ,u_{2},u_{3}, u_{4}) \bigr| \nonumber \\
& \lesssim  \frac{N_{1,max}^\frac{11}{10}}{K} M^{\frac12} M'
 \prod_{j=1}^3\|u_{1j}\|_{Z^{0}_T}
 \prod_{i=2}^4\|u_i\|_{Z^{0}_T} \,\label{L2quin1}
\end{align}
and
\begin{align}
\Bigl| \widetilde{J}_{\eta,M,K}^{5,k,T} & (P_{N_{11}}u_{11},  P_{N_{12}}u_{12}, P_{N_{13}}u_{13} ,u_{2},u_{3},u_{4})\Bigr|
 \nonumber \\
 &\lesssim  \frac{M}{K}  N_{1,max}^{\frac{11}{10}}N_{1,min}^{\frac{1}{2}}
\prod_{j=1}^3\|u_{1j}\|_{Z^{0}_T}
 \prod_{i=2}^4\|u_i\|_{Z^{0}_T} \label{L2quin2} \; .
\end{align}
Moreover, the implicit constant in estimate \eqref{L2trilin.2} only depends on the $L^{\infty}$-norm of the function $\eta$.
 \end{proposition}
 \begin{proof}
By symmetry, we may assume that $ N_{1,max}=N_{11}$ to prove \eqref{L2quin1}. We proceed as in the proof of Proposition \ref{L2trilin}. Setting
$ R= K/N_{11}^\frac{11}{10}\ll K  $ and using \eqref{prod4-est.2} and \eqref{prod4-est.3} we can easily estimate the contribution of $ 1_{T,R}^{low} P_{N_{11}} u_{11} $ by
\begin{equation*}
\begin{split}
\Bigl| J_{\eta,M,M',K}^{5,k,T} & (  1_{T,R}^{low} P_{N_{11}}u_{11},   P_{N_{12}}u_{12},    P_{N_{13}}u_{13} ,u_{2},u_{3}, u_{4})\Bigr| \\
&\lesssim  \sum_{\vec{k}_{(3)}\in  D_M(k)}    \sum_{\vec{k_1}_{(3)}\in D_{M'} (k_1)
\atop |\Omega_5(k_{1(3)},k_2,k_3,k_4)|\gtrsim K }    \| 1_{R,T}\|_{L^1_T} \\ & \quad \times\Bigl\| \eta(k_{1(3)},k_2,k_3) \widehat{P_{N_{11}}u}_{11}(k_{11}) \prod_{j=2}^3
\widehat{P_{N_{1j}}u}_{1j}(k_{1j}) \prod_{i=2}^4 |\widehat{u_{i}}(k_{i})  \Bigr\|_{L^\infty_T}\\
&\lesssim  \frac{N_{11}^\frac{11}{10}}{K}  M^{\frac12} M'
\prod_{j=1}^3\|u_{1j}\|_{L^\infty_T L^2_x}
 \prod_{i=2}^4\|u_i\|_{L^\infty_T L^2_x}  \; .
 \end{split}
\end{equation*}
Then we decompose the contribution of $ 1_{T,R}^{high}P_{N_{11}} u_{11}  $ in the same way as in \eqref{TG}.
The contribution of $ Q_{\gtrsim M'N_{11}^2} 1_{T,R}^{high}P_{N_{11}} u_{11} $ can be estimated by using  \eqref{prod4-est.2} and \eqref{prod4-est.3}
\begin{equation*}
\begin{split}
\Bigl| J_{\eta,M,M',K}^{5,k,T}  &(  Q_{\gtrsim M'N_{11}^2}   1_{T,R}^{high} P_{N_{11}}u_{11},   P_{N_{12}}u_{12},  P_{N_{13}}u_{13} ,u_{2},u_{3}, u_{4})\Bigr| \\
 &\lesssim \frac{N_{11}^\frac{11}{10}}{K}  \Bigl( (M^{\frac12} M')\wedge (M M'^{\frac12})\Bigr)
 \|u_{11}\|_{X^{-\frac{11}{10},1}_T} \prod_{j=1}^2\|u_{1j}\|_{L^\infty_T L^2_x}
 \prod_{i=2}^4\|u_i\|_{L^\infty_T L^2_x} \\
& \lesssim   \frac{N_{11}^\frac{11}{10}}{K} M^{\frac12} M'\prod_{j=1}^3\|u_{1j}\|_{Z^0_T}
 \prod_{i=2}^4\|u_i\|_{Z^0_T}\, ,
 \end{split}
\end{equation*}
and the other contributions can be estimated in the same way.

Now to prove \eqref{L2quin2} we use \eqref{prod4-est.1} instead of  \eqref{prod4-est.2}. Actually,  since $ |k_1|\gg |k_2|\vee |k_3| $ on the support of
$ \widetilde{J}_{\eta,M,M',K}^{5,k,T}$ , we know that $ m_{min}=|k_2+k_3|$ and $ |k_1|\sim k $.  Therefore \eqref{prod4-est.1} and Bernstein inequalities  lead to
\begin{align*}
\Bigl| \sum_{\vec{k}_{(3)}\in D^2_M(k)\atop |k_1|\gg |k_2|\vee |k_3|} &  \sum_{\vec{k}_{1(3)}\in D(k_1)
\atop  |\Omega_5( \vec{k}_{1(3)},k_2,k_3,k_4)|\gtrsim K}
\eta( \vec{k}_{1(3)},k_2,k_3) \prod_{j=1}^3 \widehat{P_{N_{1j}}u}_{1j}(k_{1j}) \prod_{i=2}^4\widehat{u_{i}}(k_{i}) \Bigr| \\
\lesssim &  \; M\Bigl\| \prod_{j=1}^3 P_{N_{1j}}u_{1j}\Bigr\|_{L^1_x}\prod_{i=2}^4 \|u_i\|_{L^2_x} \\
\lesssim &  \; M  N_{1,min}^{\frac{1}{2}}    \prod_{j=1}^3 \|P_{N_{1j}}u_{1j}\|_{L^2_x}  \prod_{i=2}^4 \|u_i\|_{L^2_x} \; .
\end{align*}
With this estimate in hand,  \eqref{L2quin2}  follows from the same considerations as \eqref{L2quin1} by taking $ R=K/N_{1,max}^\frac{11}{10} \ll K $.
 \end{proof}
 \begin{remark} Proceeding as in Proposition \ref{L2quin} but with the help of \eqref{prod4-est.4} we get in the same way
 \begin{align}
\Bigl|    \sum_{\vec{k}_{(3)}\in D_M(k)}  \sum_{\vec{k}_{4(3)}\in D_{M'}(-k)
\atop |\Omega_5(k_1,k_2,k_3,k_{4(3)})|\gtrsim K}&
\int_{[0,T]}\eta(k_2,k_3,k_{4(3)}) \prod_{j=1}^3 \widehat{P_{N_{4j}} u}_{4j}(k_{4j}) \prod_{i=1}^3\widehat{u_{i}}(k_{i})\Bigr| \nonumber\\
 & \lesssim \frac{N_{4,max}^\frac{11}{10}}{K}M^{\frac12} M'^{\frac12}\prod_{j=1}^3 \| P_{4j} u_{4j} \|_{Z^0_T} \prod_{i=1}^3 \|u_i\|_{Z^0_T}  \label{ref1}
\end{align}
and
\begin{align}
\Bigl| \sum_{\vec{k}_{(3)}\in D^2_M(k)\atop |k_1|\gg |k_2|\vee |k_3|}  \sum_{\vec{k}_{4(3)}\in D(-k)
\atop  |\Omega_5(k_{1},k_2,k_3,k_{4(3)})|\gtrsim K} &
\int_{[0,T]}\eta(k_2,k_3,k_{4(3)}) \prod_{j=1}^3 \widehat{u}_{4j}(k_{4j}) \prod_{i=1}^3\widehat{u_{i}}(k_{i})\Bigr| \nonumber \\
 & \lesssim \frac{N_{4,max}^\frac{11}{10}}{K}M^{\frac12} N_{4,min}^{\frac12}\prod_{j=1}^3 \| P_{4j} u_{4j} \|_{Z^0_T} \prod_{i=1}^3 \|u_i\|_{Z^0_T}  \label{ref2}
\end{align}
The above estimates will be actually also  needed in the proof of Proposition \ref{prop58}.
 \end{remark}
 \subsubsection{$L^2$-$7$-linear estimates}
 We use the notations : $ \vec{k}_{(7)}=(k_1, k_2,.., k_8)\in \Z^8 $ and for any $ k\in \Z $,
$$
\Gamma_7(k)=\big\{ \vec{k}_{(7)} \in \Z^8 \ : \ \sum_{i=1}^8 k_i=k \big\} \, .
$$
The proof of the following lemma follows from exactly the same considerations as the ones used in the proof of Lemma \ref{prod5-est}.
 \begin{lemma}\label{prod4-est7}
 Let $f_j\in l^2(\mathbb Z)$, $j=1,...,8 $. Setting
$$\mathcal{J}^{7,1}_k :=
 \sum_{\vec{k}_{(7)}\in \Gamma^7(0) \atop k_8=-k} \phi_{M}\Bigl(\sum_{q=1}^6 k_q\Bigr)\phi_{M_1}(k_{1}+k_{2})\phi_{M_2}(k_{4}+k_{5}) \prod_{j=1}^8|f_j(k_j)|
$$
and
$$
\mathcal{J}^{7,2}_k:=\sum_{\vec{k}_{(7)}\in \Gamma^7(0) \atop k_8=-k} \phi_{M}(\sum_{q=4}^7 k_q)\phi_{M_1}(k_{1}+k_{2})\phi_{M_2}(k_{4}+k_{5}) \prod_{j=1}^8|f_j(k_j)|
$$
it holds
\begin{equation}
\mathcal{J}^{7,1}_k+\mathcal{J}^{7,2}_k \lesssim \min\Bigl( M^{\frac{1}{2}}M_1 M_2  \prod_{j=1}^8\|f_j\|_{l^2}\; ,\;  M^{\frac{1}{2}}M_1 M_2^{\frac12}
  \prod_{j=4}^5 \|\langle \xi \rangle^{\frac14} f_j \|_{l^2} \prod_{j=1\atop j\not\in \{4,5\}}^8\|f_j\|_{l^2}\Bigr) \, .\label{prod4-est.5}
\end{equation}
Similarly, by setting
$$\tilde{\mathcal{J}}^{7}_{k,i} := \sum_{\vec{k}_{(7)}\in \Gamma^7(0) \atop k_8=-k}  \psi_i(\vec{k}_{(7)})  \prod_{j=1}^8|f_j(k_j)|
$$
with
\begin{displaymath}
\psi_1(\vec{k}_{(7)})=\phi_{M}\Bigl(\sum_{q=1}^6 k_q\Bigl)\phi_{M_1}\Bigr( \sum_{q=1}^4 k_q\Bigr)\phi_{M_2}(k_{1}+k_{2})  \, ,
\end{displaymath}
\begin{displaymath}
\psi_2(\vec{k}_{(7)}) = \phi_{M}(\sum_{q=1}^6 k_q)\phi_{M_1}(k_{4}+k_{5})\phi_{M_2}(k_{1}+k_{2}) ,
\end{displaymath}
$$
\psi_3(\vec{k}_{(7)})=\phi_{M}(k_6+k_7) \phi_{M_1}\Bigl(\sum_{q=1}^4 k_q\Bigl)\phi_{M_2}(k_{1}+k_{2})
$$
and
$$
\psi_4(\vec{k}_{(7)}) =\phi_{M}(k_6+k_7)\phi_{M_1}(k_{4}+k_{5})\phi_{M_2}(k_{1}+k_{2}) ,
$$
it holds
\begin{equation}
\sum_{i=1}^4 \tilde{\mathcal{J}}^{7,i}_k \lesssim \min\Bigl( M^{\frac{1}{2}}M_1 M_2  \prod_{j=1}^8\|f_j\|_{l^2}\; ,\;  M^{\frac{1}{2}}M_1 M_2^{\frac12}
  \prod_{j=1}^2 \|\langle \xi \rangle^{\frac14} f_j \|_{l^2} \prod_{j=3}^8\|f_j\|_{l^2}\Bigr) \, .\label{prod4-est.51}
\end{equation}

\end{lemma}
\begin{proposition} \label{L2seven}
Assume that $0<T \le 1$, $\eta$ : $ \Z^7\to \C $ is a bounded measurable function. Assume also that $u_{1i}$, $ u_{2i} $ with  $ i=1,2,3$,
 and $ u_3, \, u_4 $   are functions in $Z^{0}_{T}$  and  that $k\ge 2^9 $, $M\ge 1$ and  $M_1\ge 1 $.
For $ i=1,2 $, we set $ \vec{u}_{i(3)}:=(u_{i1}, u_{i2}, u_{i3}) $ and define
\begin{align*}
J_{\eta,M,M_1}^{7,k,T,i}  &(\vec{u}_{1(3)},\vec{u}_{2(3)},u_{3},u_{4})
:= \sum_{\vec{k}_{(3)}\in D^1_M(k)} \sum_{\vec{k}_{1(3)}\in D^1_{M_1}(k_1)} \sum_{\vec{k}_{2(3)}\in D^i(k_2) \atop \Omega_5(\vec{k}_{1(3)},k_2,k_3,k_4)\not \sim \Omega_3(\vec{k}_{2(3)})}\\
&
\int_{[0,T]}\eta(\vec{k}_{1(3)},\vec{k}_{2(3)},k_{3}) \prod_{j=1}^3 \widehat{u}_{1j}(k_{1j}) \prod_{m=1}^3\widehat{u}_{2m}(k_{2m}) \prod_{q=3}^4 \widehat{u}_q(k_q)
\end{align*}
with $ k_4=-k$ and where . \\
Then
\begin{equation} \label{L2seven1}
\big| J_{\eta,M,M_1}^{7,k,T,1}  (\vec{u}_{1(3)},\vec{u}_{2(3)},u_{3},u_{4}) \big| \lesssim  M^{\frac12} M_1 k^{\frac{1}{10}}  \prod_{j=1}^3\|u_{1j}\|_{Z^{0}_T}\|u_{2j}\|_{Z^{0}_T} \prod_{q=3}^4 \|u_q\|_{Z^0_T}
\end{equation}
and
\begin{equation} \label{L2seven2}
\big| J_{\eta,M,M_1}^{7,k,T,2}  (\vec{u}_{1(3)},\vec{u}_{2(3)},u_{3},u_{4}) \big| \lesssim  M^{\frac12} M_1 k^{-\frac{9}{10}+}    \prod_{j=1}^3\|u_{1j}\|_{Z^{0}_T}\|u_{2j}\|_{Z^{0}_T} \prod_{q=3}^4 \|u_q\|_{Z^0_T}\, .
\end{equation}
Moreover, the implicit constant in estimate \eqref{L2trilin.2} only depends on the $L^{\infty}$-norm of the function $\eta$.
 \end{proposition}
 \begin{proof}
We define  the resonance function of order $7$ for $\vec{k}_{(7)}=(k_1,\cdots,k_8) \in \Gamma^7(0)$ by
\begin{equation} \label{res7}
\Omega^7(\vec{k}_{(7)}) = \sum_{i=1}^8 k_i^3 \; .
\end{equation}
  Again a direct calculation shows  that
 \begin{equation}\label{si7}
 \Omega_7(\vec{k}_{1(3)},\vec{k}_{2(3)},k_3,k_4)=\Omega_5(\vec{k}_{1(3)},k_2,k_3,k_4)+\Omega_3(\vec{k}_{2(3)})
 \end{equation}
  and thus  $ | \Omega_7(\vec{k}_{1(3)},\vec{k}_{2(3)},k_3,k_4)\ \gtrsim |\Omega_3(\vec{k}_{2(3)})|$ on the support of $J_{\eta,M,M_1}^{7,k,T,i}$ with $ i\in \{1,2\} $. Moreover, we deduce from Lemma \ref{inter} that $|\Omega_3(\vec{k}_{2(3)})| \gtrsim M_2^2k$ on the support of $J_{\eta,M,M_1}^{7,k,T,1}$ and that $|\Omega_3(\vec{k}_{2(3)})| \gtrsim M_2k^2$ on the support of  $J_{\eta,M,M_1}^{7,k,T,2}$.

 Estimates \eqref{L2seven1}-\eqref{L2seven2} follow then from the same considerations as  in the proof of Propositions \ref{L2tri}-\ref{L2quin}  by making use of  \eqref{prod4-est.5}. Note that we loose a factor $ k^{0+} $ in \eqref{L2seven2} by resuming over $M_{2,min}$.
 \end{proof}
\begin{remark} \label{remark42}Setting
\begin{align*}
\tilde{J}_{\eta,M,M_1}^{7,k,T,i}  &(\vec{u}_{11(3)},u_{12}, u_{13},u_2, u_{3},u_{4})
:= \sum_{\vec{k}_{(3)}\in D^1_M(k)} \sum_{\vec{k}_{1(3)}\in D^1_{M_1}(k_1)} \sum_{\vec{k}_{11(3)}\in D^i(k_2) \atop \Omega_5(\vec{k}_{1(3)},k_2,k_3,k_4)\not \sim \Omega_3(\vec{k}_{11(3)})}\\
&
\int_{[0,T]}\eta(\vec{k}_{11(3)},k_{12}, k_{13},k_2, k_{3}) \prod_{j=1}^3 \widehat{u}_{11j}(k_{11j}) \prod_{m=2}^3\widehat{u}_{1m}(k_{1m}) \prod_{q=2}^4 \widehat{u}_q(k_q)
\end{align*}
and proceeding as in the proof of the preceding proposition but with the help of \eqref{prod4-est.51} we obtain
\begin{equation} \label{L2seven11}
\big| \tilde{J}_{\eta,M,M_1}^{7,k,T,1} (\vec{u}_{11(3)},u_{12}, u_{13},u_2, u_{3},u_{4}) \big| \lesssim  M^{\frac12} M_1 k^{\frac{1}{10}}  \prod_{j=1}^3\|u_{1j}\|_{Z^{0}_T}\|u_{2j}\|_{Z^{0}_T} \prod_{q=3}^4 \|u_q\|_{Z^0_T}
\end{equation}
and
\begin{equation} \label{L2seven22}
\big| \tilde{J}_{\eta,M,M_1}^{7,k,T,2}  (\vec{u}_{11(3)},u_{12}, u_{13},u_2, u_{3},u_{4}) \big| \lesssim  M^{\frac12} M_1 k^{-\frac{9}{10}+}    \prod_{j=1}^3\|u_{1j}\|_{Z^{0}_T}\|u_{2j}\|_{Z^{0}_T} \prod_{q=3}^4 \|u_q\|_{Z^0_T}\, .
\end{equation}
These estimates are   actually also needed in the proof of Proposition \ref{prop58}.
\end{remark}
 \subsection{Proof of Theorem \ref{theo56}} To prove Theorem \ref{theo56} we construct a modified energy in the same way as  in \cite{MoPiVe}. Note that this way of construction of modified energies has much in common with the I-method \cite{CKSTT}.

Theorem \ref{theo56}  will  be  a direct consequence of Lemma \ref{lem57} and Proposition \ref{prop58} below.\vspace*{2mm}\\
{\it Definition of the modified energy : }
 For $t \ge 0$, we define the modified energy at  the mode $ k\ge 2^9 $ by
\begin{equation} \label{defEk}
  \mathcal{E}_k(t)= \mathcal{E}_k(u(t)) =
 \frac{k}{2} |\widehat{u}(t,k)|^2 + \alpha \mathcal{E}_k^{3,1}(t) + \beta  \mathcal{E}_k^{3,2}(t) +\gamma \mathcal{E}_k^5(t)
 \end{equation}
where $\alpha$, $\gamma$ and $\beta$ are real constants to be determined later. \\
In the sequel of this subsection, to simplify the formula, we set $ k_4=-k $. \\
$ \mathcal{E}_k^{3,1}$, $\mathcal{E}_k^{3,2}$,  $\mathcal{E}_k^5 $  are then defined as follows :
\begin{displaymath}
\mathcal{E}_k^{3,1}(u) = k^2\Re\Bigl[ \sum_{M<k^{\frac{7}{12}} } \sum_{\vec{k}_{(3)}\in D^1_M(k)}  \frac{1}{\Omega^3(\vec{k}_{(3)})} \prod_{j=1}^4 \widehat{u}(k_j) \Bigr]\, ,
\end{displaymath}

\begin{equation*}
\begin{split}
\mathcal{E}_k^{3,2}(u) &= k^2\Re\Bigl[  \sum_{\vec{k}_{(3)}\in D^2(k)\atop |k_{med}| \ll k^{\frac{2}{3}}}  \frac{1}{\Omega^3(\vec{k}_{(3)})} \prod_{j=1}^4 \widehat{u}(k_j)\Bigr]
\end{split}
\end{equation*}
where $\vec{k}_{(3)}=(k_1,k_2,k_3)$ and the dyadic decompositions in $N_j$ are nonhomogeneous,
\begin{align}
\mathcal{E}_k^{5}(u) =& k^2 \Re\Bigl[ \sum_{i=1}^4 \sum_{1\le M<k^\frac{7}{12}}  \sum_{\vec{k}_{(3)}\in D^1_M(k)} \sum_{\vec{k}_{i(3)}\in D^1(k_i)\atop  \Omega^3(\vec{k}_{(3)}) \ll \Omega^3(\vec{k_i}_{(3)})}
 \frac{k_i}{\Omega^3(\vec{k}_{(3)})\Omega^5(\vec{k}_{i(5)})}
 \prod_{j=1 \atop{j \neq i}}^4 \widehat{u}(k_j)\prod_{q=1}^3  \widehat{u}(k_{i,q})\Bigr]
  \, . \label{prop-ee.6b1}
\end{align}with the  notation
 \begin{equation}\label{nota1}\vec{k}_{j(5)}=(\vec{\widetilde{k}}_{j(3)},\vec{k}_{j(3)}) \in \Gamma^5(k) \, ,
 \end{equation}
  where $\vec{\widetilde{k}}_{j(3)}$ is defined by
\begin{displaymath}
\vec{\widetilde{k}}_{1(3)}=(k_2,k_3,k_4), \ \vec{\widetilde{k_2}}_{(3)}=(k_1,k_3,k_4), \, \vec{\widetilde{k}}_{3(3)}=(k_1,k_2,k_4),
 \ \vec{\widetilde{k}}_{4(3)}=(k_1,k_2,k_3) \, .
\end{displaymath}

Next, we show that if $s >\frac14$, then the non quadratic part of  $\mathcal{E}_k(u) $ is controlled by the  $ H^s $-norm of $ u$.
\begin{lemma}\label{lem57}
Let $s > 1/4$ For any $ u\in H^s(\T) $ it holds \begin{equation} \label{lem-Est.1}
|\mathcal{E}_k^{3,1}(u)|+ |\mathcal{E}_k^{3,2}(u)| \lesssim
\|P_{\lesssim k} u\|_{H^s}^4 \end{equation}
and
\begin{equation} \label{lem-Est.2}
|\mathcal{E}_k^{5}(u)| \lesssim
\|P_{\lesssim k} u\|_{H^s}^6 \end{equation}

\end{lemma}

\begin{proof}
Since $  |\Omega_3(\vec{k}_{(3)})|\ge M_{min} M_{med} k $ on $ D^1(k) $, \eqref{prod4-est.1}  leads to
$$
|\mathcal{E}_k^{3,1}(u)|\lesssim \sum_{1\le M_{min}\le M_{med} }\frac{k^2 M_{min}^{\frac12} k^{-4s}}{M_{min} M_{med} k} \|P_{\sim k} u \|_{L^\infty_t H^s_x}^4
\lesssim k^{1-4s}  \|P_{\lesssim k}u\|^4_{H^s_x}
$$
which is acceptable. In the same way, on $ D^2(k) $, it holds  $ |\Omega_3(\vec{k}_{(3)})|\ge M_{min} k^2 $ and thus \eqref{prod4-est.1} leads this time  to
$$
|\mathcal{E}_k^{3,2}(u)|\lesssim \sum_{M\ge 1 }\frac{k^2 M^{\frac12} k^{-2s}}{M k^2} \|P_{\sim k} u \|_{L^\infty_t H^s_x}^2 \|P_{\lesssim k}u\|_{L^\infty_t L^2_x}^2
\lesssim k^{-2s}  \|P_{\lesssim k} u\|^4_{H^s_x}\; .
$$
Now, fix $i\in \{1,2,3,4\}$. For $ \vec{k}_{(3)}\in D^1(k)$ and $ \vec{k}_{i(3)}\in D^1(k_i)$  we have
$$| k_1|\sim |k_2|\sim |k_3|\sim |k_4|\sim |k_{i,1}|\sim |k_{i,2}|\sim |k_{i,3}| \sim k \; . $$
Moreover, for $ \vec{k}_{i(5)} $ (see \eqref{nota1}) such that   $\Omega^3(\vec{k}_{(3)}) \ll \Omega^3(\vec{k}_{i(3)})$ we have
 $$ |\Omega^5( \vec{k}_{i(5)})|=|\Omega^3(\vec{k}_{(3)})+ \Omega^3(\vec{k}_{i(3)})|\sim | \Omega^3(\vec{k}_{i(3)})|\gg |\Omega^3(\vec{k}_{(3)})|>0 \, .
 $$
Therefore $\mathcal{E}_k^{5}(u)$ is well defined and according to   \eqref{prod4-est.2}-\eqref{prod4-est.4},
\begin{eqnarray*}
|\mathcal{E}_k^{5}(u)| &  \lesssim  & \sum_{i=1}^4  \sum_{1\le M_{min}\le M_{med} }  \sum_{1\le M_{i,min}\le M_{i,med} }\frac{k^3 M_{min}^{1/2} M_{i,min}  k^{-6s}}{M_{min} M_{med} M_{i,min} M_{i,med}k^2} \|P_{\sim k} u \|_{H^s_x}^6 \\
& \lesssim & k^{1-6s}  \|P_{\lesssim k}u\|^6_{H^s_x}
\end{eqnarray*}
which is acceptable.
\end{proof}
\begin{proposition} \label{prop58}Let $ s\ge 1/3 $ and $ u\in Z^s_T $ be a solution to \eqref{remKdV}. Then for $ k\ge 2^9 $,
\begin{equation}
|\mathcal{E}_k(t)- \mathcal{E}_k(0)|\lesssim
\sup_{N\gtrsim k}\Bigl[ \Bigl( \frac{k}{N}\Bigr)^{s-} \|P_{\lesssim N} u\|_{Z^s_T}^4 (1+\|P_{\lesssim  N} u \|_{Z^s_T}^4)\Bigr]\; .
\end{equation}
\end{proposition}
\begin{proof}
 Since $ u $ is real valued we can restrict ourself to positive $ k$. As above, we differentiate $\mathcal{E}_k$ with respect to time and then integrate between 0 and $t$ to get
\begin{align}
\mathcal{E}_k(t) &= \mathcal{E}_k(0) - k \Re\Bigl[ \int_0^t P_k\partial_x(A(u,u,u)-B(u,u,u))P_k u \Bigr]
+ \alpha \int_0^t\frac{d}{dt}\mathcal{E}_k^{3,1}(t')dt' \nonumber \\
&\quad + \beta \int_0^t \frac{d}{dt} \mathcal{E}_k^{3,2}(t')dt'
+\gamma  \int_0^t \frac{d}{dt} \mathcal{E}_k^5(t')dt'  \nonumber \\
&=: \mathcal{E}_k(0) + I_k + \alpha J_k^1 + \beta J_k^2+ \gamma  K_k \, . \label{prop-ee.3}
\end{align}
As in the preceding section, since $ u $ is real-valued, the contribution of $ \partial_x B(u,u,u) $ is purely imaginary and thus vanishes. Recalling that we set $ k_4=-k $, we can thus  rewrite $I_k$ in Fourier variables as
\begin{align*}
I_k &= k^2 \Im \Bigl[\int_0^t  \sum_{\vec{k}_{(3)}\in D(k)}  \widehat{u}(k_1) \widehat{u}(k_2) \widehat{u}(k_3) \widehat{u}_4(k_4)\Bigr] \\
\end{align*}
with  $  D(k) $ defined as in the beginning of this section. We
 denote by  $ I^1_k $ and $ I^2_k $ the contributions to $ I_k$ of  respectively $D^1(k) $ and $ D^2(k) $. Finally, we decompose $ I^1_k $ and
$ I^2_k $ in the following way :
\begin{align*}
I^1_k &= (\sum_{M\ge k^{\frac{7}{12}}} + \sum_{M< k^{\frac{7}{12}}}) k^2  \Im \Bigl[\int_0^t  \sum_{\vec{k}_{(3)}\in D^1_M(k)}   \widehat{u}(k_1) \widehat{u}(k_2) \widehat{u}(k_3) \widehat{u}(k_4) \Bigr]\\
& = I_k^{1,high}+ I_k^{1,low}
\end{align*}
and
\begin{align*}
I^2_k &= (\sum_{\vec{k}_{(3)}\in D^{2}(k)\atop k_{med}\ge k^{2/3}} + \sum_{\vec{k}_{(3)}\in D^{2}(k) \atop k_{med}< k^{2/3}})
k^2  \Im \Bigl[\int_0^t  \widehat{u}(k_1) \widehat{u}(k_2) \widehat{u}(k_3) \widehat{u}(k_4) \Bigr] \\
& = I_k^{2,high}+ I_k^{2,low} \, .
\end{align*}
{\bf $\bullet $  Estimate for $I^{1,high}_k$.} In view of Lemma \ref{inter}, on $D^1 $ it holds $ |k_1|\sim |k_2|\sim |k_3|\sim |k| $ and \eqref{res3} ensures that $|\Omega_3|\gtrsim M^2  k$.
Therefore \eqref{L2tri2} leads to
\begin{eqnarray*}
|I^{1,high}_k|& \lesssim & k^2 \sum_{M\ge k^{\frac{7}{12}}}\frac{M^{\frac12} k^\frac{11}{10}}{M^2 k} k^{-4s} \prod_{i=1}^4 \| P_{\sim k} u \|_{Z^s}    \\
& \lesssim & k^{-4s+\frac{49}{40}}  \|P_{\lesssim k} u\|_{Z^s}^4 \, ,  \\
\end{eqnarray*}
which is acceptable for $ s> \frac{49}{160}$.\\

\noindent {\bf  $\bullet $ Estimate for  $I^{1,low}_k+\alpha J_k$.}   By \eqref{remKdV}, we can rewrite $\frac{d}{dt}\mathcal{E}_k^{3,1}$ as the sum of the
``linear'' contribution\footnote{By ``linear'' contribution, we mean the contribution of the linear part when substituting $ u_t $ by using the equation.}
\begin{equation*}
\sum_{M< k^\frac{7}{12}} k^2 \Re \Bigl[ \sum_{\vec{k}_{(3)}\in D^1_M(k)}  \frac{ i(k_1^3+k_2^3+k_3^3+k_4^3)}{\Omega^3(\vec{k}_{(3)})}
 \prod_{j=1}^4 \widehat{u}(k_j) \Bigr]
\end{equation*}
and the
``nonlinear" contribution
\begin{align*}
& \sum_{i=1}^4 \sum_{M< k^{\frac{7}{12}}}k^2 \Re \Bigl[  \sum_{\vec{k}_{(3)}\in D^1_M(k)}  \frac{-ik_i}{\Omega^3(\vec{k}_{(3)})}
 \prod_{j=1 \atop{j \neq i}}^4 \widehat{u}(k_j)\Bigl(3 |\widehat{u}(k_i)|^2 \widehat{u}(k_i)+\sum_{\vec{k}_{i(3)}\in D(k_i)} \prod_{q=1}^3 \widehat{u}(k_{i,q})\Bigr)\Bigr]  \, .
\end{align*}
Using the resonance relation \eqref{res3}, we see by choosing $\alpha= 1$ that $I_k^{1,low}$ is canceled out by the linear contribution of $\int_0^t \frac{d}{dt}\mathcal{E}_N^{3}$.
Hence,
$$
I_k^{1,low}+ J_k=c \sum_{j=0}^1   \tilde{A}_k^j   \, ,
$$
where, by symmetry,
\begin{displaymath}
\begin{split}
\tilde{A}_k^0&=\sum_{1\le M< k^\frac{7}{12}} k^3 \Im \Bigl[ \int_0^t  \sum_{\vec{k}_{(3)}\in D^1_M(k)}  \frac{1}{\Omega^3(\vec{k}_{(3)})}
 \prod_{j=1}^3 \widehat{u}(k_j) \\
  &  \hspace*{30mm} \Bigl(3|\widehat{u}(-k)|^2 \widehat{u}(-k)+\sum_{\vec{k_{4}}_{(3)}\in D(-k)} \prod_{q=1}^3 \widehat{u}(k_{4,q})\Bigr)\Big]\\
   &= A^{0,0}_k+ A^{0}_k
  \end{split}
\end{displaymath}
and
\begin{displaymath}
\begin{split}
\tilde{A}_k^1&=3\sum_{1\le M< k^\frac{7}{12}} k^2 \Im \Bigl[\int_0^t  \sum_{\vec{k}_{(3)}\in D^1_M(k)}  \frac{k_1}{\Omega^3(\vec{k}_{(3)})}
 \prod_{j=2}^4 \widehat{u}(k_j) \\
  &  \hspace*{30mm}\Bigl(3|\widehat{u}(k_1)|^2 \widehat{u}(k_1)+\sum_{\vec{k_{1}}_{(3)}\in D(k_1)} \prod_{q=1}^3 \widehat{u}(k_{1,q})
 \Bigr) \Bigr]\\
 &= A^{1,0}_k+ A^{1}_k \, .\end{split}
\end{displaymath}

It  thus remains  to treat the terms $\tilde{A}_k^j$ corresponding to the nonlinear contribution of $\frac{d}{dt}\mathcal{E}_N^{3}$. Since $ |k_i|\sim k $,
 $ A_k^0 $ and $ A_k^1 $ can be treated almost in the same way. Actually, some estimates on $ A_k^0 $ are easier thanks to \eqref{prod4-est.4}
  and \eqref{ref1}. We thus only consider $ \tilde{A}_k^1$.  First, by using \eqref{prod4-est.1}, $ A^{1,0}_k $ can be easily estimated by
 \begin{eqnarray*}
|A^{1,0}_{k}| & \lesssim & \sum_{1\le M_{min}\le M_{med}}\frac{k^3 M_{min}^{\frac12} }{M_{min} M_{med} k} k^{-6 s}\|P_{\sim k} u\|_{L^\infty_t H^s_x}^6  \\
& \lesssim &k^{2-6s}\|P_{\lesssim k} u\|_{L^\infty_t H^s_x}^6 \\
\end{eqnarray*}
which is acceptable for $ s\ge 1/3 $. \\
\noindent {\bf  $\bullet $ Estimate for $A^1_k$.} 
By symmetry we can assume that
 $ M_{11} \le M_{12} \le M_{13} $. We set    $ N_{1,max}=\max(N_{11}, N_{12}, N_{13}) $. \\
  {\bf  Case 1:}  $ |\Omega_3(\vec{k}_{(3)})| \gtrsim |\Omega_3(\vec{k}_{1(3)})|$. Then we must have
 $$
 M_{11} \lesssim \frac{M_{min} M_{med} k}{M_{12} M_{13}} \; .
 $$
 {\bf Case 1-1:} $ \vec{k}_{1(3)} \in D^2(k_1)$.  In this region it holds $ M_{12} \gtrsim k $  and thus $ M_{11} \lesssim \frac{M_{min} M_{med}}{k} $. On account of \eqref{prod4-est.2}-\eqref{prod4-est.3} we get
 \begin{eqnarray*}
|A^1_{k}| & \lesssim & \sum_{1\le M_{min} \le \min(M_{med},k^\frac{7}{12})} \sum_{1\le M_{11} \lesssim  \frac{M_{min} M_{med}}{k}}
\sum_{N_{1,max}\gtrsim k}
\frac{k^3 M_{min}^{\frac12} M_{11}}{M_{min} M_{med} k } k^{-3 s} N_{1,max}^{-s} \\
 & & \|P_{\sim k} u\|_{L^\infty_t H^s_x}^3 \|P_{ N_{1,max}} u\|_{L^\infty_t H^s_x}
\|P_{\lesssim N_{1,max}} u\|^2_{L^\infty_t L^2_x}  \\
& \lesssim & \sum_{N_{1,max}\gtrsim k} \sum_{ 1\le M_{min}<k^{\frac{7}{12}}} M_{min}^{\frac12} k^{1-4 s} (\frac{k}{N_{1,max}})^s
  \|P_{\lesssim N_{1,max}}u\|^6_{L^\infty_t H^s_x}  \\
& \lesssim &  \sum_{N_{1,max}\gtrsim k} \sum_{ 1\le M_{min}<k^{\frac{7}{12}}}  (\frac{M_{min}}{k^\frac{7}{12}})^{\frac12} k^{\frac{31}{24}-4 s} (\frac{k}{N_{1,max}})^s
 \|P_{\lesssim N_{1,max}}u\|^6_{L^\infty_t H^s_x}  \\
 & \lesssim &   \sum_{N_{1,max}\gtrsim k}  k^{-4 s+\frac{31}{24}} (\frac{k}{N_{1,max}})^s      \|P_{\lesssim N_{1,max}}u\|^6_{L^\infty_t H^s_x} \end{eqnarray*}
which is acceptable for $ s\ge \frac{31}{96} $. \\
{\bf Case 1.2:}  $ \vec{k}_{1(3)} \in D^1(k_1)$. Then it holds
$$
|k_{11}|\sim |k_{12}|\sim |k_{13}|\sim k \; .
$$
Since  $ |\Omega_3(\vec{k}_{(3)})| \gtrsim |\Omega_3(\vec{k}_{1(3)})|$,  we must have $ M_{min}M_{med} \gtrsim M_{11}^2 $.
Therefore \eqref{prod4-est.2}-\eqref{prod4-est.3} lead to
\begin{eqnarray*}
|A^1_{k}| & \lesssim & \sum_{1\le M_{min} \le M_{med}}\frac{k^3 M_{min}^{\frac12} (M_{min}M_{med})^{\frac12}\,  k^{-6 s}}{M_{min} M_{med} k}  \|P_{\sim k} u\|_{L^\infty_t H^s_x}^6  \\
& \lesssim &  k^{2-6s} \|P_{\sim k} u\|_{L^\infty_t H^s_x}^6
\end{eqnarray*}
which is acceptable for $ s\ge 1/3$. \\
{\bf Case 2:}  $ |\Omega_3(\vec{k}_{(3)})| \ll |\Omega_3(\vec{k}_{1(3)})|$.
Then, by \eqref{res55}, $| \Omega_5|\sim  |\Omega_3(\vec{k}_{1(3)})|  \gg |\Omega_3(\vec{k}_{(3)})| $.\\
{\bf Case 2-1:}  $ \vec{k}_{1(3)} \in D^2(k_1)$. By symmetry we may assume $ k_{11}\ge k_{12} \ge k_{13} $.  Then  $  |\Omega_3(\vec{k}_{1(3)})|\sim  M_{1,min} N_{11}^2 $ and according to \eqref{L2quin1} we  obtain
\begin{align}
|A^1_{k}|  \lesssim     \sum_{N_{11}\gtrsim k \atop 1\le N_{13}\le N_{12} \le N_{11}} \sum_{1\le M_{min} \le M_{med}\atop 1\le M_{1,min} \le N_{11}} &
\frac{k^3 M_{min}^{1/2}M_{1,min} }{M_{min} M_{med} k \, (M_{1,min} N_{11}^2)}  N_{11}^\frac{11}{10}\nonumber \\
& k^{-3s}
N_{11}^{-s}   \|P_{\sim k} u\|_{Z^s}^3 \|P_{N_{11}} u\|_{Z^s} \|P_{\lesssim N_{12}}u\|_{Z^0}  \|P_{\lesssim N_{13}}u\|_{Z^0}
\nonumber \\
 \lesssim      \sum_{N_{11}\gtrsim k} k^{-4s+\frac{11}{10}+}  & (\frac{k}{N_{11}})^s \|P_{\lesssim N_{11}}u\|_{Z^s}^6
\end{align}
which is acceptable for $ s>\frac{11}{40}$. \\
{\bf Case 2-2:}  $ \vec{k}_{1(3)} \in D^1(k_1)$. Then we must have
\begin{equation}\label{uu7}
|k_{11}|\sim |k_{12}|\sim |k_{13}|\sim |k_1|\sim |k_2| \sim |k_3| \sim k \, .
\end{equation}
and
$$
 \Omega^3(\vec{k_1}_{(3)})\sim M_{1,min} M_{1,med} k \, .
 $$
We call $ A^{1,low}_k $  this contribution to $A^1_k $ and $ A^{0,low}_k $ the same contribution to $ A^0_k $.
Using the equation and the resonance relation \eqref{res5}, we can rewrite $K_k:=\int_0^t \frac{d}{dt}\mathcal{E}_k^5 $ as
\begin{align*}
& \Re\Bigl[ \int_0^t \sum_{i=1}^4 \sum_{1\le M<k^\frac{7}{12}} k^2 \sum_{\vec{k}_{(3)}\in D^1_M(k)} \sum_{\vec{k}_{i(3)}\in D^1(k_i)\atop  \Omega^3(\vec{k}_{(3)}) \ll \Omega^3(\vec{k}_{i(3)})}
 \frac{ik_i}{\Omega^3(\vec{k}_{(3)})}
 \int_0^t \prod_{j=1 \atop{j \neq i}}^4 \widehat{u}(k_j)\prod_{q=1}^3  \widehat{u}(k_{iq})\\
 &+ \int_0^t \sum_{i=1}^4 \sum_{1\le M<k^\frac{7}{12}} k^2 \sum_{\vec{k}_{(3)}\in D^1_M(k)} \sum_{\vec{k}_{i(3)}\in D^1(k_i)\atop  \Omega^3(\vec{k}_{(3)}) \ll \Omega^3(\vec{k}_{i(3)})}
 \frac{k_i}{\Omega^3(\vec{k}_{(3)})\Omega^5(\vec{k}_{i(5)})} \prod_{q=1}^3  \widehat{u}(k_{iq}) \\
 &\sum_{m=1\atop m\neq i }^4  \prod_{j=1 \atop{j \neq i, m}}^4 \widehat{u}(k_j)  (-ik_{m})
 \Bigl( 3 |\widehat{u}(k_{m})|^2\widehat{u}(k_{m})+\sum_{\vec{k}_{m(3)}\in D(k_m)} \prod_{p=1}^3 \widehat{u}({k_{mp}})\Bigr) \\
&+  \int_0^t \sum_{i=1}^4 \sum_{1\le M<k^\frac{7}{12}} k^2 \sum_{\vec{k}_{(3)}\in D^1_M(k)} \sum_{\vec{k}_{i(3)}\in D^1(k_i)\atop  \Omega^3(\vec{k}_{(3)}) \ll \Omega^3(\vec{k}_{i(3)})}
 \frac{k_i}{\Omega^3(\vec{k}_{(3)})\Omega^5(\vec{k}_{i(5)})}
 \prod_{j=1 \atop{j \neq i}}^4 \widehat{u}(k_j) \\
 &\sum_{m=1}^3  \prod_{q=1\atop q\neq m}^3  \widehat{u}(k_{i,q}) (-ik_{im})
 \Bigl( 3 |\widehat{u}(k_{im})|^2\widehat{u}(k_{im})+\sum_{\vec{k}_{i,m(3)}\in D(k_{i,m})} \prod_{p=1}^3 \widehat{u}({k_{i,m,p}})\Bigr)\Bigr] \\
  &:= H_k^1+\widetilde{H}_k^2+\widetilde{H}_k^3.
\end{align*}
{\bf $\bullet $ Estimate for $  A^{1,low}_k+A^{0,low}_k+\gamma K_k$.} By choosing $\gamma=c$, the above calculations lead to
\begin{equation} \label{prop-ee.11}
  A^{1,low}_k+A^{0,low}_k+\gamma K_k =c\Bigl(  \widetilde{H}_k^2+\widetilde{H}_k^3 \Bigr) \, .
\end{equation}
Because of \eqref{uu7}, $  \widetilde{H}_k^2$ and $ \widetilde{H}_k^3 $ can be estimated in the same way\footnote{
Actually one needs to use  \eqref{L2seven11} and \eqref{L2seven22}  instead of  \eqref{L2seven1}  and  \eqref{L2seven2} to treat $ \widetilde{H}_k^3 $ }.  It thus suffices to  consider $\widetilde{H}_k^2$ for  any fixed couple $(i,m)\in \{1,2,3,4\}^2$ with $ i\neq m$. By symmetry, we can restrict ourselves
  to $(i,m)\in \{(1,2), (1,4)\} $ . Since the case $ m=4 $ is easier (see  \eqref{prod4-est.4}), we only consider the case $ (i,m)=(1,2) $. We thus have to bound
 \begin{align*}
 &  k^2 \Im\Bigl[  \int_0^t  \sum_{1\le M< k^\frac{7}{12}}  \sum_{\vec{k}_{(3)}\in D^1_M(k)} \sum_{\vec{k}_{1(3)}\in D^1(k_1)\atop  \Omega^3(\vec{k}_{(3)}) \ll \Omega^3(\vec{k}_{1(3)})}
 \frac{k_1 k_{2}}{\Omega^3(\vec{k}_{(3)})\Omega^5(\vec{k}_{1(3)},k_2,k_3)}
 \prod_{j=3}^4 \widehat{u}(k_j) \\
 & \quad \times  \prod_{q=1}^3  \widehat{u}(k_{1,q})
 \Bigl(  3|\widehat{u}(k_{2})|^2\widehat{u}(k_{2})+\sum_{\vec{k}_{2(3)}\in D(k_{2})} \prod_{p=1}^3 \widehat{u}({k_{2p}})\Bigr)\Bigr] \\
 & = : H^{3,0}_k+ H^3_k \, .
 \end{align*}
 First by  \eqref{prod4-est.2}-\eqref{prod4-est.3}, we easily get
 \begin{align*}
 |H_k^{3,0}| & \lesssim \sum_{M \ge 1}\sum_{M_{1,min} \ge 1}\frac{k^4}{M^2 k M_{1,min}^2 k} M^{\frac12}M_{1,min}k^{-8s} \|P_{\sim k} u \|_{H^s}^8 \\
  & \lesssim  k^{2-8s}  \|P_{\sim k} u \|_{H^s}^8 \, ,
 \end{align*}
 which is acceptable for $ s\ge 1/4 $. \\
{\bf $\bullet $ Estimate for $H^3_k $.} Now, to bound $ H^3_k $ we separate different contributions. \\
{\bf Case 1:} $|\Omega_5(\vec{k}_{1(5)})|\gtrsim |\Omega_3(\vec{k}_{2(3)})|$. Then we must have $  |\Omega_3(\vec{k}_{2(3)})|\lesssim |\Omega_3(\vec{k}_{1(3)})|$ since
 $ |\Omega_5(\vec{k}_{1(5)})|\sim |\Omega_3(\vec{k}_{1(3)})|$.    This forces $ M_{2,min}\lesssim M_{1,med} $.\\
{\bf Case 1-1:} $ N_{2,med}\sim N_{2,max}. $ Then \eqref{prod4-est.5}  leads to
 \begin{eqnarray*}
| H^3_k|& \lesssim &\sum_{1\le M \le M_{med}\atop 1\le M_{1,min} \le M_{1,med}\lesssim k} \sum_{N_{2,max}\gtrsim k}\frac{k^4 M^{\frac12} M_{1,min} M_{1,med}}{M M_{med} k M_{1,min} M_{1,med} k}  \\
 & &k^{-5 s} N_{2,max}^{-2s}   \|P_{\sim k} u\|_{L^\infty_t H^s_x}^5
 \|P_{\sim N_{2,max}} u\|_{L^\infty_t H^s_x}^2  \|P_{\lesssim N_{2,max}} u\|_{L^\infty_t L^2_x}  \\
 & \lesssim & \sum_{N_{2,max}\gtrsim k}  k^{-7s+2+}  (\frac{k}{N_{2,max}})^{2s}  \|P_{\lesssim N_{2,max}} u\|_{L^\infty_t H^s_x}^8
\end{eqnarray*}
which is acceptable for $ s>2/7 $. \\
{\bf Case 1-2:} $ N_{2,med}\ll N_{2,max}. $ On account of \eqref{uu7}, it holds
 $ N_{2,max}\sim k$ and  $  |\Omega_3(\vec{k}_{2(3)})|\sim M_{2,min} k^2 $.  The inequality  $  |\Omega_3(\vec{k}_{2(3)})|\lesssim |\Omega_3(\vec{k}_{1(3)})|$  then  ensures that
 $$
 M_{2,min} \lesssim \frac{M_{1,min} M_{1,med}}{k}\, .
 $$
 Estimate \eqref{prod4-est.5} thus  leads  for $ s\ge 1/4 $ to
 \begin{eqnarray*}
 |H^3_k|& \lesssim &\sum_{1\le M\le M_{med}\atop 1\le M_{1,min} \le M_{1,med}\lesssim k}
  \frac{k^4 M^{\frac12} M_{1,min}  (M_{1,min} M_{1,med} k^{-1} )^{\frac12} }{M M_{med} k M_{1,min} M_{1,med} k} k^{-6 s} \\
 & & \quad \quad \quad \quad \quad \quad \quad \quad \times \|P_{\sim k} u\|_{L^\infty_t H^s_x}^6  \|P_{\lesssim  k }u\|_{L^\infty_t H^{\frac14}_x}^2  \\
 & \lesssim &\sum_{1\le M \le M_{med}\atop 1\le M_{1,min} \le M_{1,med}\lesssim k}
  \frac{k^{3/2} k^{-6s}}{M^{\frac12} M_{med}} \Bigl( \frac{M_{1,min}}{M_{1,med}}\Bigr)^{\frac12}  \|P_{\lesssim  k} u\|_{L^\infty_t H^s_x}^8  \\
 & \lesssim & k^{-6s+\frac{3}{2}+} \|P_{\lesssim k}u\|_{L^\infty_t H^s_x}^8 \, ,
\end{eqnarray*}
which is acceptable for $ s>1/4 $. \\
{\bf Case 2:} $|\Omega_5(\vec{k}_{1(5)})|\ll |\Omega_3(\vec{k}_{2(3)})|$. Then, $ |\Omega_7|\sim |\Omega_3(\vec{k}_{2(3)}|$ on account of \eqref{si7}. \\
{\bf Case 2-1:} $ \vec{k}_{2(3)} \in D^2(k_2)$. Then \eqref{L2seven2} gives
\begin{eqnarray*}
 |H^3_k| & \lesssim & \sum_{ M\ge 1 \atop M_{1,min}\ge 1}\sum_{N_{2,max} \gtrsim k}\frac{k^4 M^{\frac12} M_{1,min} }{M^2 k M_{1,min}^2  k } k^{-5 s} N_{2,max}^{-s} k^{-\frac{9}{10}+} \\ && \quad \quad \quad \quad \quad \quad \quad \quad  \times
\|P_{\sim k} u\|_{Z^s}^5
 \|P_{N_{2,max}}u\|_{Z^s}  \|P_{\lesssim N_{2,max}} u\|^2_{Z^0}  \\
 & \lesssim & \sum_{N_{2,max} \gtrsim k} k^{-6s+\frac{11}{10}+} (\frac{k}{N_{2,max}})^{s}  \|P_{\lesssim N_{2,max}}u\|_{Z^s}^6 \, ,
\end{eqnarray*}
which is acceptable for $ s> \frac{11}{60} $. \\
{\bf Case 2-2:} $ \vec{k}_{2(3)} \in D^1(k_2)$.  Then we have
$$
|k_{21}|\sim |k_{22}|\sim |k_{23}| \sim k \; .
$$
Therefore \eqref{L2seven1} leads to
\begin{eqnarray*}
 |H^3_k| & \lesssim & \sum_{ M\ge 1 \atop M_{1,min}\ge 1}\frac{k^4 M^{\frac12} M_{1,min} }{M^2 k M_{1,min}^2 k } k^{-8s}  k^\frac{1}{10}
 \|P_{\sim k} u\|_{Z^s}^8  \\
 & \lesssim & k^{\frac{21}{10}-8s}  \|P_{\sim k} u\|_{Z^s}^8 \, ,
\end{eqnarray*}
which is acceptable for $ s\ge \frac{21}{88} $. \\

\noindent {\bf $\bullet $ Estimate for $I^{2,high}_k$}.
By symmetry, in the sequel we assume $ |k_1|\ge |k_2| \ge |k_3| $.  We note from Lemma \ref{inter} that on $ D^2(k) $ it holds
$ |\Omega_3(\vec{k}_{(3)})| \gtrsim M_{min} k_1^2$.   \\
 {\bf Case 1:} $ |k_1|\sim |k_2| \ge |k_3|$. \\
{\bf Case 1-1:} $ |k_3| \gtrsim k $. Then by \eqref{L2tri2} it holds
\begin{eqnarray*}
  \big| I^{2,high}_k \big| & \lesssim & \sum_{N_1\gtrsim k}\sum_{M\ge 1}\frac{k^2  M^{\frac12} N_1^{\frac{11}{10}}}{M N_1^2} k^{-2 s}  N_1^{-2s} \|P_k u\|_{Z^s}
   \|P_{\gtrsim k} P_{\lesssim N_1}u\|_{Z^s} \|P_{\sim N_1}u\|_{Z^s}^2 \\
 & \lesssim &  \sum_{N_1\gtrsim k}\ k^{\frac{11}{10}-4s} (\frac{k}{N_1})^{2s} \|P_k  u \|_{Z^s} P_{\lesssim N_1} u \|_{Z^s}^3 \; ,
 \end{eqnarray*}
 which is acceptable for $ s\ge \frac{11}{40}$.\\
 {\bf Case  1-2:} $ |k_3|\ll k $. It forces  $ m_{min}=|k_1+k_2|\sim k  $ and thus \eqref{L2tri2} leads to
 \begin{eqnarray*}
 \big|  I^{2,high}_k \big| & \lesssim & \sum_{N_1\gtrsim k} \sum_{M\sim k}\frac{k^2  M^{\frac12} N_1^{\frac{11}{10}}}{M N_1^2} k^{- s}  N_1^{-2s} \|P_k u\|_{Z^s}
    \|P_{\sim N_1}u\|_{Z^s}^2 \|P_{\lesssim N_1}u\|_{Z^0}  \\
 & \lesssim &   \sum_{N_1\gtrsim k}  k^{\frac{3}{5}-3s} (\frac{k}{N_1})^{2s} \| P_{\lesssim N_1} u \|_{Z^s}^5 \, ,
  \end{eqnarray*}
  which is acceptable for $ s\ge \frac{1}{5}$.\\
 {\bf Case 2:} $ |k_1| \gg |k_2| \ge |k_3|$. Then $ |k_1|\sim k $ and $ m_{min}=|k_2+k_3| $. In this region we will make use of the fact that $|k_2|\gtrsim k^{\frac{2}{3}} $.\\
 {\bf Case 2-1:} $|k_3|\ll |k_2| $. Then it holds $ m_{min}=|k_2+k_3|\sim |k_2|\gtrsim k^{\frac{2}{3}} $ and thus  by \eqref{L2tri2},
  \begin{eqnarray*}
   \big| I^{2,high}_k \big| & \lesssim & \sum_{M\ge k^{\frac{2}{3}} }\frac{k^2  M^{\frac12}  k^{\frac{11}{10}}}{M k^2} k^{-2 s}  k^{-\frac{2s}{3}}\|P_{\sim  k} u\|_{Z^s}^2
     \|P_{\gtrsim  k^{\frac{2}{3}}}  P_{\lesssim k} u\|_{Z^s} \|  P_{\lesssim k} u\|_{Z^0}  \\
 & \lesssim &  k^{(\frac{13}{30}-\frac{8s}{3})} \| P_{\lesssim k} u \|_{Z^s}^4\; ,
  \end{eqnarray*}
  which is acceptable for $ s> \frac{13}{80} $.\\
 {\bf Case 2.2:} $|k_2|\sim |k_3|$. Then  \eqref{L2tri2}   leads to
 \begin{eqnarray*}
  \big| I^{2,high}_k \big| & \lesssim & \sum_{M\ge 1 }\frac{k^2  M^{\frac12} k^{\frac{11}{10}}}{M k^2} k^{-2 s}  k^{-\frac{4s}{3}}\|P_{\sim  k} u\|_{Z^s}^2
     \|P_{\gtrsim  k^{\frac{2}{3}}}  P_{\lesssim k}  u\|_{Z^s} ^2 \\
 & \lesssim &  k^{\frac{11}{10}-\frac{10s}{3}} \| P_{\lesssim k} u \|_{Z^s}^4\; ,
  \end{eqnarray*}
 which is acceptable for $ s>\frac{33}{100} $.  
 
\noindent  {$\bullet $ \bf Estimate for $I^{2,low}_k+\beta J^2_k$}.
By \eqref{remKdV}, we can rewrite $J_2^k=\int_0^t \frac{d}{dt}\mathcal{E}_k^{3,2}$ as the sum of the
``linear'' contribution
\begin{equation*}
k^2 \Re \int_0^t \Bigl[ \sum_{\vec{k}_{(3)}\in D^2(k) \atop |k_{med}|\ll  k^{\frac{2}{3}}}  \frac{i (k_1^3+k_2^3+k_3^3+k_4^3)}{\Omega^3(\vec{k}_{(3)})}
 \prod_{j=1}^4 \widehat{u}(k_j)\Bigr]
\end{equation*}
and the
``nonlinear" contribution
\begin{align*}
 &k^2  \Re  \int_0^t  \Bigl[ \sum_{i=1}^4\sum_{\vec{k}_{(3)}\in D^2(k) \atop |k_{med}|\ll k^{\frac{2}{3}}}   \frac{-ik_i}{\Omega^3(\vec{k}_{(3)})}
 \prod_{j=1 \atop{j \neq i}}^4 \widehat{u}(k_j)
   \Bigl(3|\widehat{u}(k_i)|^2 \widehat{u}(k_i)+\sum_{\vec{k}_{i(3)}\in D(k_i)} \prod_{q=1}^3 \widehat{u}(k_{i,q})\Bigr) \Bigr] \nonumber \\
  &= k^2  \Im  \int_0^t \Bigl[\sum_{i=1}^4\sum_{\vec{k}_{(3)}\in D^2(k) \atop |k_{med}|\ll k^{\frac{2}{3}}}   \frac{k_i}{\Omega^3(\vec{k}_{(3)})}
 \prod_{j=1 \atop{j \neq i}}^4 \widehat{u}(k_j)
   \Bigl(3|\widehat{u}(k_i)|^2 \widehat{u}(k_i)+\sum_{\vec{k}_{i(3)}\in D(k_i)} \prod_{q=1}^3 \widehat{u}(k_{i,q})\Bigr) \Bigr] \nonumber \\
   &=\sum_{i=1}^4 (B_k^{i,0}+B_k^i )\, . \label{prop-ee.6b1}
\end{align*}

Using the resonance relation \eqref{res3}, we see by choosing $\beta= 1$ that $I_k^{2,low}$ is canceled out by the linear contribution of $\int_0^t \frac{d}{dt}\mathcal{E}_k^{3,2}$.
Hence,
 $$
I^{2,low}_k+ J^2_k=  \sum_{i=1}^4 (B_k^{i,0}+B_k^i )  \;
$$
Note that since $ |k_{med}|\ll k^{\frac{2}{3}} $, we must have $ |k_{max}| \sim k $.
In the sequel, by symmetry we  assume  that $ k_{max}=k_1$. This forces $ |k_1|\sim k $ and $ M_{min}=M_1\lesssim k^{\frac{2}{3}} $. \\

\noindent  {$\bullet $ \bf Estimate for $ B^{i,0}_k$, $ i=1,2,3,4$}.
 Let $ N_i $ be the dyadic variable associated to the dyadic decomposition with respect to $ k_{i} $. Note that $ N_i\lesssim k$.  According to
 \eqref{prod4-est.1} it holds
 \begin{eqnarray*}
 |B_k^{i,0}|&  \lesssim  & \sum_{M\ge 1} \sum_{1\le N_{i}\le  k }
 \frac{k^2  M^{\frac12} N_i  }{M k^2} k^{-  s}  N_{i}^{-3s}
 \|P_{ \sim k} u\|_{L^\infty_t H^s_x}  \|P_{N_{i}} u\|_{L^\infty_t H^s_x}^3   \| P_{ \lesssim k}  u\|_{L^\infty_T L^2_x}^2 \\
 & \lesssim &  k^{-s} k^{\max(0,1-3s)}\|P_{ \lesssim k}  u\|_{L^\infty_t H^s_x}^6
 \end{eqnarray*}
 which is acceptable for $ s\ge 1/4$. \\
 
\noindent  {$\bullet $ \bf Estimate for $ B^2_k$ and $B^3_k$}.
 By symmetry, these two contributions can be treated in exactly the same way.  So we only consider $ B^2_k $.
 By symmetry we can also assume that $ |k_{21}|\ge |k_{22}|\ge |k_{23}|$.
For $ i=1,2,3$,  let $ N_{2i}$  be the dyadic variable associated to the dyadic decomposition with respect to $ k_{2i} $.
Recall also that $ |k_2|\ll  k^{\frac{2}{3}} $ on the contribution of $B^2_k$.\\
{\bf Case 1:} $|k_{21}|\ge k $.  Then we must have $ N_{22}\sim N_{21} $.
 On account  of  \eqref{prod4-est.1} and H\"older  and Sobolev inequalities, we have for $ s\ge 1/4 $,
    \begin{eqnarray*}
 \big| B^2_k \big| &\lesssim &  \sum_{1\le M\lesssim k^{2/3}} \sum_{N_{21}\ge k}
 \frac{k^2 k^{\frac{2}{3}}  M^{\frac12}  }{M k^2} k^{-2 s} N_{21}^{-2s} \\
 & &\quad \times \|P_{ \sim k} u\|_{L^\infty_T H^s_x}^2  \|P_{N_{21}} u\|_{ L^\infty_T H^s_x}^2   \| P_{\lesssim k^{\frac{2}{3}}} u\|_{ L^\infty_T L^2_x}
  \| P_{\le N_{21}} u\|_{ L^\infty_T L^\infty_x} \\
 &\lesssim &  \sum_{1\le M\lesssim k^{2/3}} \sum_{N_{21}\ge k}
  k^{\frac{2}{3}}  k^{-2 s} N_{21}^{-2s+\frac{1}{4}+}  \\
 & &\quad \times \|P_{ \sim k} u\|_{L^\infty_T H^s_x}^2  \|P_{N_{21}} u\|_{ L^\infty_T H^s_x}^2   \| P_{\lesssim k^{\frac{2}{3}}} u\|_{ L^\infty_T H^{\frac14}_x}
    \| P_{\le N_{21}} u\|_{ L^\infty_T H^{1/4}_x}\\
 & \lesssim &  \sum_{N_{21}\ge k}   k^{-4s+\frac{11}{12}+}(\frac{k}{N_{21}})^{2s-(\frac{1}{4}+)} \|P_{\lesssim N_{21}} u\|_{ L^\infty_T H^s_x}^6   \,
 \end{eqnarray*}
 which is acceptable for $ s\ge 1/4$.\\
{\bf  Case 2:} $|k_{21}|<k$.
Using that $  N_{2,max} \gtrsim N_2 $,  similar considerations as in   \eqref{prod4-est.1} together with \eqref{convol}, H\"older  and Sobolev inequalities, we get for $ s\ge 1/4 $,
  \begin{eqnarray*}
\big| B^2_k \big| &\lesssim & \sum_{ M\ll k^{\frac{2}{3}}}\sum_{N_2\ll k^{\frac{2}{3}} }
 \frac{k^2 N_2  M^{\frac12}}{M k^2} k^{-2 s}  N_2^{\frac{1}{4}-s} \|P_{ \sim k} u\|_{L^\infty_T H^s_x}^2   \|P_{\lesssim k} u\|_{L^\infty_T H^{\frac14}_x}^4\\
  & \lesssim & k^{-2s} k^{\max(0, \frac{5}{6}-\frac{2}{3}s)} \|P_{\lesssim k} u\|_{L^\infty_T H^s}^6 \; ,
 \end{eqnarray*}
  which is acceptable for $ s> \frac{5}{16}$.\\
  
 \noindent
 {\bf $\bullet $  Estimate for $ B^1_k$}.
 By symmetry we can assume that
 $$ |k_{11}|\ge |k_{12}|\ge |k_{13}| \, . $$
 For $ i=1,2,3$,  let $ N_{1,i}$  be the dyadic variable associated to the dyadic decomposition with respect to $ k_{1,i} $. \\
\noindent
 {\bf Case 1:} $|\Omega_3(\vec{k}_{(3)})| \not \sim | \Omega_3(\vec{k}_{1(3)})|$.  \\
 {\bf  Case 1-1: } $M_{1,med} \ge 2^{-6} |k_{11}|$. Then $ |\Omega_5| \gtrsim  | \Omega_3(\vec{k}_{1(3)})|\gtrsim k_{11}^2$ and
 \eqref{L2quin2} leads, for $ s>1/4 $  to
     \begin{eqnarray*}
 \big| B^1_k \big| &\lesssim & \sum_{M\lesssim k} \sum_{N_{11}\gtrsim k} \sum_{N_{13}\le N_{12}\le N_{11}}  \frac{k^2 k M N_{13}^{\frac12-s} }{M k^2 N_{11}^2} k^{-  s}  N_{11}^{-s} N_{12}^{-s}  N_{11}^{\frac{11}{10}}\\
 & &\quad \quad \times \|P_{ k} u\|_{Z^s}  \|P_{ N_{11}} u\|_{Z^s}   \|P_{N_{12}} u\|_{Z^s}  \|P_{ N_{13}} u\|_{Z^s} \|P_{\lesssim k} u \|_{Z^0}^2 \\
 & \lesssim &  \sum_{N_{11}\gtrsim k}  k^{-2s+\frac{1}{10}+} (\frac{k}{N_{11}})^{s+\frac{9}{10}}  \|P_{\lesssim N_{11}}u\|_{Z^s}^6
 \end{eqnarray*}
 which acceptable for $ s> \frac{1}4$.\\
 \noindent
 {\bf Case 1-2:}   $M_{1,med} \le 2^{-6} |k_{11}|$. Then $|k_{13}|\sim |k_{12}| \sim |k_{11}|\sim k$ and \eqref{L2quin1} leads to
  \begin{eqnarray*}
 \big| B^1_k \big| &\lesssim & \sum_{M\ge 1} \sum_{M_{1,min}\ge 1}  \frac{k^2 k M^{\frac12} M_{1,min}  }{M k^2 M_{1,min}^2 k} k^{- 4 s} k^{\frac{11}{10}}
  \|P_{\sim  k} u\|_{Z^s}^4  \|u \|_{Z^0}^2\\
 & \lesssim &  k^{\frac{11}{10}-4s}  \|P_{\sim  k} u\|_{Z^s}^4  \|u \|_{Z^0}^2 \, ,
 \end{eqnarray*}
 which acceptable for $ s>\frac{11}{40}$.\\
 {\bf  Case 2:} $|\Omega_3(\vec{k}_{(3)})|  \sim | \Omega_3(\vec{k}_{1(3)})|$.  \\
 {\bf  Case 2-1:} $ |k_{11}|\sim |k_{12}|$.
 Then we claim that $ |k_{13}|\gtrsim k $. Indeed, recalling that $ |k_{11}+k_{12}+k_{13}|\sim k $,  $ |k_{13}| \ll k $ would imply that $ |k_{11}+k_{12}|\sim k $ and thus
  $  | \Omega_3(\vec{k}_{1(3)})|\sim k k_{11}^2 \gtrsim k^3 \gg |\Omega_3(\vec{k}_{(3)})| $ since $   |\Omega_3(\vec{k}_{(3)})| \le k^{\frac{8}{3}}$.
  Therefore,  $|\Omega_3(\vec{k}_{1(3)})| \gtrsim  M_{1,min}^2 k_{11}$  and $ |\Omega_3(\vec{k}_{(3)})| \sim M_1 k^2 $  force $ M_{1,min} \lesssim (M_1 k)^{1/2} \lesssim (N_2 k)^{1/2} $ and  \eqref{prod4-est.2} leads,  for $ s\ge 1/4 $, to
    \begin{eqnarray*}
  \big| B^1_k \big| &\lesssim & \sum_{1\lesssim M_{1}\lesssim N_2 \lesssim   k^{\frac23 }
   \atop 1\le M_{1,min} \lesssim (N_2 k)^{1/2} } \sum_{N_{11}\ge N_{13}\gtrsim k}  \frac{k^2 k M_{1,min}^{\frac12} M_{1}  }{M_{1}  k^2 } k^{- s} N_{13}^{-s}
 N_{11}^{-2s}  N_2^{-\frac14} \\
  & & \quad \times \|P_{ k} u\|_{L^\infty_t H^s_x}  \|P_{\sim N_{11}} u\|_{L^\infty_t H^s_x}^2  \|P_{N_{13}} u\|_{L^\infty_t H^s_x}
  \|P_{N_2} u\|_{L^\infty_T H^{1/4}_x}  \|P_{\lesssim k^{\frac23 }} u\|_{L^\infty_T L^2_x}\\
 &\lesssim  &   k^{-4s+\frac{5}{4}+}  (\frac{k}{N_{11}})^{2s}  \|P_{\lesssim N_{11}} u\|_{L^\infty_t H^s_x}^6
 \end{eqnarray*}
 which is acceptable for $ s\ge  \frac{5}{16}$.\\
{\bf  Case 2-2: $|k_{11}|\gg |k_{12}| \ge |k_{13}|$}. Then $\frac{k}{2} \le  |k_{11}| \le 2  k $, $ M_{1,min}=M_{11}$  and $ | \Omega_3(\vec{k}_{1(3)})| \sim M_{11} k^2 $.  Therefore,  $|\Omega_3(\vec{k}_{(3)})|  \sim | \Omega_3(\vec{k}_{1(3)})|$ forces $ M_1\sim M_{11}$.\\
{\bf  Case 2-2-1: $ k_{11} \neq k$}.
Let $ k'\in \{k_{12},k_{13},k_2,k_3\}$ such that $|k'|=\max(|k_{12}|,|k_{13}|,|k_2|,|k_3|) $. Since $ |k'|\ll |k| $, it holds
\begin{equation}\label{fe}
 |\Omega_3(k_{11},-k,k')|\ge \frac{k^2}{2} \, .
 \end{equation}
{\bf  Case 2-2-1-1:} $ |\Omega_5| \gtrsim k^2$. Then, for $ 1/4 \le s\le 1/2 $, \eqref{L2quin2} leads to
  \begin{eqnarray*}
 \big| B^1_k\big| &\lesssim & \sum_{ M_{1}\lesssim   k^{\frac23 } } \sum_{1\le N_{13}\le N_{12}\lesssim k} \frac{k^3  M_{1} N_{13}^{\frac{1}{2}-s} }{M_{1}  k^4 } k^{-  2 s}
 N_{12}^{-s}  k^{\frac{11}{10}} \\
 &  &  \quad \times  \|P_{\sim k} u \|_{Z^s}^2  \|P_{N_{12}}u\|_{Z^s}  \|P_{N_{13}}u\|_{Z^s} \|P_{\lesssim k} u\|_{Z^0}^2\\
 &\lesssim  &   k^{-2s+\frac{1}{10}+} \|P_{\lesssim k}u\|_{Z^s}^6\, ,
 \end{eqnarray*}
 which is acceptable for $s\ge \frac{1}{20}$. \\
 {\bf  Case 2-2-1-2:} $ |\Omega_5| \ll k^2$. Recall that we have $ M_{11} \sim M_1 $ in this region. Let $ (z_1,z_2,z_3)  $ be such that  $ \{z_1,z_2,z_3\}:=\{k_{12},k_{13},k_2,k_3\} \setminus \{k'\} $. It then follows from \eqref{fe} and \eqref{res55}  that
 $ |\Omega_3(z_1,z_2,z_3)|\gtrsim k^2 $ in this region. Since $ |k_2|\vee |k_3| \ll k^{\frac{2}{3}} $ this forces
    $ |k_{12}|\wedge |k_{13}|\gtrsim k^{\frac23 } $. For $ s\ge 1/4$, \eqref{prod4-est.22} thus yields
   \begin{eqnarray*}
 \big| B^1_k\big| &\lesssim & \sum_{ M_{1,1}\sim M_{1} \lesssim   k^{\frac23 } } \frac{k^3  M_{1,1}  }{M_{1}  k^2 } k^{-  2 s}
 k^{-\frac{4s}{3}} \\
 &  & \quad \times \|P_{ \sim k} u\|_{L^\infty_t H^s_x}^2   \|P_{\gtrsim k^{\frac23}} P_{\lesssim k} u \|_{L^\infty_T H^s_x}^2  \|P_{\lesssim k^{\frac23 }}u\|_{L^\infty_T H^{\frac14}_x}^2 \\
 &\lesssim  &   k^{-\frac{10}{3}s +1+}   \|P_{\lesssim k}  u\|_{L^\infty_T H^s_x}^6 \, ,
 \end{eqnarray*}
which is acceptable for $ s> \frac{3}{10}$. \\
{\bf  Case 2-2-2:} $ k_{11} = k$. This is the more complicated case. Following \cite{TT} we first notice that
\begin{equation}\label{rr}
\Bigl| \frac{k^2}{(k_1+k_2)(k_1+k_3)}-1\Bigr|= \Bigl| \frac{(k_2+k_3)k-k_2k_3}{(k_1+k_2)(k_1+k_3)}\Bigr|\lesssim \frac{|k_2|\vee |k_3|}{|k|}  \, .
\end{equation}
We decompose the contribution of this region to $ B_k^1 $ as
\begin{eqnarray*}
B^1_k & = &  \Im \int_0^t \Bigl[\sum_{|k_2|\vee|k_3|\le  k^{\frac{2}{3}}}    \Bigl[ \Bigl(\frac{k^2}{(k_1+k_2)(k_1+k_3)}-1\Bigr)+1\Bigr]  \frac{k_1}{k_2+k_3}  \\ & & \quad \quad \quad \times \widehat{u}(k)  \widehat{u}(k_{12}) \widehat{u}(k_{13})\widehat{u}(k_2) \widehat{u}(k_3) \widehat{u}(-k)\Bigr] \\
& = &C^1_k+C^2_k \, .
\end{eqnarray*}
It is also worth noticing that, since $  k_{12}+k_{13}+k_2+k_3=0 $ in this region,  we must have
\begin{equation}\label{nn} 
 (k_{12}+k_{13})=-(k_2+k_3)  \Longrightarrow M_{1,1}= M_{1} \, .
\end{equation}
$\bullet $ {\bf Estimate for $ C^1_k.$}  \\
{\bf  Case 1:} $|k_2|\vee |k_3| \gg |k_2|\wedge |k_3|$. By symmetry we can assume that $ |k_2| \gg |k_3|$,  which forces $ M_1 \sim |k_2|$.
According to \eqref{rr} and  \eqref{prod4-est.22}, $ C^1_k $ can be easily estimated for $ s\ge 1/4 $ by
\begin{eqnarray*}
 |C^1_k| &\lesssim & \sum_{N_2\sim M_1\lesssim   k^{\frac23 } } \frac{k}{M_1}  \frac{N_2}{k}  M_{1}  k^{- 2 s} N_2^{-s} \\ && \quad \times
    \|P_{k} u\|_{L^\infty_t H^s_x}^2   \|P_{N_2} u\|_{L^\infty_T H^s_x} \|P_{\lesssim k} u \|_{L^\infty_T H^{\frac14}_x}^2 \|P_{\lesssim k} u\|_{L^\infty_T L^2_x}\\
 & \lesssim & k^{-\frac{8}{3}s+\frac{2}{3}}   \|P_{\lesssim k} u\|_{L^\infty_t H^s_x}^6 \, ,
 \end{eqnarray*}
 which is acceptable for $s\ge \frac{1}{4} $.\\
 {\bf  Case 2:} $|k_2|\sim |k_3|$. Then $|k_2|\wedge |k_3| \gtrsim M_1 $ and since $ M_1= M_{1,1}$ we also have $|k_{12}|\gtrsim M_1$. Therefore,
 According to \eqref{rr}-\eqref{nn} and  \eqref{prod4-est.22}, $ C^1_k $ can be easily estimated for $ s<1/2 $ by
\begin{eqnarray*}
 |C^1_k| &\lesssim & \sum_{M_1\lesssim N_2\lesssim   k^{\frac23 } \atop  N_{1,2} \gtrsim M_1} \frac{k}{M_1}  \frac{N_2}{k}  M_{1}^{\frac12} M_{1}   k^{- 2 s} N_2^{-2s}N_{12}^{-s}\\
  & & \quad \times
    \|P_{k} u\|_{L^\infty_t H^s_x}^2   \|P_{N_2} u\|_{L^\infty_T H^s_x}^2  \|P_{N_{12}} u\|_{L^\infty_T H^s_x} \|P_{\lesssim k} u \|_{L^\infty_T L^2_x}\\
     & \lesssim &    \sum_{M_1\lesssim N_2\lesssim   k^{\frac23 } } M_1^{\frac{1}{2}-s} N_2^{1-2s} k^{-2s}\|P_{\lesssim k} u\|_{L^\infty_t H^s_x}^6 \\
 & \lesssim & k^{1-4s}   \|P_{\lesssim k} u\|_{L^\infty_t H^s_x}^6
 \end{eqnarray*}
 which is acceptable for $ s>1/4 $. \\
 
\noindent {$\bullet $ \bf Estimate for $ C^2_k$}.
Rewriting $ k_1 $ as $ k_1=k_{11}+k_{12}+k_{13} =k+k_{12}+k_{13}$ we decompose $ C^2_k $ as the sum of three terms
 $ C^{21}_k +C^{22}_k+C^{23}_k $.  \\
 
\noindent {$\bullet $  \bf Estimate on $ C^{22}_k$ and $ C^{23}_k$}.
 We only  consider $ C^{22}_k $ which is the contribution of $ k_{12}$ since  $C^{23}_k $ can be treated in exactly the same way.
  We proceed as for $ C^1_k $. \\
  {\bf  Case 1:} $|k_{12}|\gg |k_{13}|.$ This  forces $ M_{11} \sim |k_{12}|$.
According to \eqref{nn} and  \eqref{prod4-est.22}, $ C^{22}_k $ can be easily estimated for $ s\ge 1/4 $ by
\begin{eqnarray*}
 |C^{22}_k| &\lesssim & \sum_{N_{12}\sim M_1 \lesssim   k^{\frac23} } \frac{N_{12}}{M_{1}} M_{1}^{\frac12}  k^{- 2 s} N_{12}^{-s}\\
 && \quad \times
    \|P_{k} u\|_{L^\infty_t H^s_x}^2   \|P_{N_{12}} u\|_{L^\infty_T H^s_x} \|P_{\lesssim k} u \|_{L^\infty_T H^{1/4}_x}^2 \|P_{\lesssim k} u\|_{L^\infty_T L^2_x}\\
 & \lesssim & k^{-\frac{8}{3}s+\frac{2}{3}}   \|P_{\lesssim k} u\|_{L^\infty_t H^s_x}^6
 \end{eqnarray*}
 which is acceptable for $s> \frac{1}{4} $.\\
 {\bf  Case 2:} $|k_{12}|\sim |k_{13}|$. Then $|k_{12}|\wedge |k_{13}| \gtrsim M_1 $ and since $ M_1=M_{1,1}$ we also have $|k_{2}|\vee |k_3|\gtrsim M_1$. Therefore,
 according to  \eqref{prod4-est.2}, $ C^{22}_k $ can be easily estimated for $ s<1/2 $ by
\begin{eqnarray*}
 |C^{22}_k| &\lesssim & \sum_{M_1\lesssim N_2 \lesssim   k^{\frac23}  \atop  M_1 \lesssim N_{12}\ll k } \frac{N_{12}}{M_1}  M_{1}^{\frac12} M_{1}   k^{- 2 s} N_2^{-s}N_{12}^{-2s}\\
  & & \quad \times
    \|P_{k} u\|_{L^\infty_t H^s_x}^2   \|P_{N_2} u\|_{L^\infty_T H^s_x}  \|P_{N_{12}} u\|_{L^\infty_T H^s_x}^2 \|P_{\lesssim k} u \|_{L^\infty_T L^2_x}\\
     & \lesssim &    \sum_{M_1\lesssim \lesssim   k^{\frac23} \atop M_1\lesssim N_{12} \ll k } M_1^{\frac{1}{2}-s} N_{12}^{1-2s} k^{-2s}\|P_{\lesssim k} u\|_{L^\infty_t H^s_x}^6 \\
 & \lesssim & k^{\frac{4}{3}-\frac{14}{3}s}   \|P_{\lesssim k} u\|_{L^\infty_t H^s_x}^6
 \end{eqnarray*}
 which is acceptable for $ s>2/7 $. \\
 
\noindent  {\bf $\bullet $  Estimate for $ C^{21}_k$.}
 We first notice that since $ |k_1|\gg  |k_2|\vee |k_3|$ and $ |k_{11}|\gg |k_{12}|\vee |k_{13}| $, $(k_1+k_2)(k_1+k_3)(k_2+k_3) \neq 0 $ if and only if $ k_2+k_3\neq 0 $ and similarly, $(k_{11}+k_{12})(k_{11}+k_{13})(k_{12}+k_{13}) \neq 0 $ if and only if $ k_{12}+k_{13}=-(k_2+k_3)\neq 0 $.  We can thus rewrite
  $ C^{21}_k $ as
 \begin{eqnarray*}
  C^{21}_k =  k |u(k)|^2 \Im \int_0^t \Bigl[ \sum_{ |k_2|\vee |k_3|\le k^{\frac{2}{3}}, \, k_{2}+k_{3}\neq 0  \atop k_{12}+k_{13}=-k_2-k_3,
  \,  |k_{12}|\vee |k_{13}|\ll k} \frac{1}{k_2+k_3} \widehat{u}(k_{12}) \widehat{u}(k_{13}) \widehat{u}(k_{2}) \widehat{u}(k_{3})\Bigr]
  \end{eqnarray*}
  We now separate the contributions    $ C^{21,low}_k $ and $  C^{21,high}_k $ of the regions $ |k_{12}|\vee |k_{13}|\le k^{\frac{2}{3}}$
   and $ |k_{12}|\vee |k_{13}|
  > k^{\frac{2}{3}}$. Let us start by bounding $  C^{21,high}_k $. In the region $ |k_{13}|\sim |k_{12}| $, it can be bounded for  $ s\ge 1/4 $, thanks to Sobolev inequalities, by
    \begin{eqnarray*}
| C^{21,high}_k| &\lesssim & \sum_{M_{1,1}=M_{1}\le k^{\frac{2}{3}}}\frac{k   M_{1,1}^{\frac12} }{M_{1} } k^{- 2 s}  k^{-\frac{4s}{3}} \\ && \quad \times
   \|P_{k} u\|_{L^\infty_t H^s_x}^2   \|P_{\lesssim k}P_{\gtrsim   k^{\frac{2}{3}} } u \|_{L^\infty_T H^s_x}^2 \|P_{\le  k^{\frac{2}{3}} } u\|_{L^\infty_T H^{1/4}}^2\\
 & \lesssim & k^{1-\frac{10s}{3}}  \|P_{\lesssim k} u\|_{L^\infty_t H^s_x}^6
 \end{eqnarray*}
 which is acceptable for $ s> \frac{3}{10} $. \\
 On the other hand, in the region $ |k_{12}|\not\sim |k_{13}| $, we must have $ |k_{12}|\gg |k_{13}| $ and thus $ M_{1,1}\sim |k_{12}| $. Moreover, \eqref{nn} forces $ |k_{2}|\vee  |k_{3}| \sim  |k_{12}|$.
 Therefore, Sobolev inequalities lead for $ s< 1/2 $ to
   \begin{eqnarray*}
| C^{21,high}_k| &\lesssim & \sum_{ k^{\frac{2}{3}}\lesssim N_{12}\lesssim k }\frac{k  N_{12}^{1-2s} }{N_{12} }
 k^{- 2 s}  N_{12}^{-2s}
   \|P_{k} u\|_{L^\infty_t H^s_x}^2   \|P_{N_{12}} u \|_{L^\infty_T H^s_x}^2 \|P_{\lesssim N_{12}} u \|_{L^\infty_T H^s_x}^2\\
 & \lesssim & k^{1-\frac{14s}{3}}  \|P_{\lesssim k} u\|_{L^\infty_t H^s_x}^6 \, ,
 \end{eqnarray*}
  which is acceptable for $ s\ge \frac{3}{14}$. \\
 
  Finally, we claim that $ C^{21,low}_k =0$. Indeed,
 performing the change of variables
 $$
  (k_2,k_3,k_{12},k_{13})\mapsto(-k_{12}, -k_{13},-k_2,-k_3) = (\hat{k}_2,\hat{k}_3,\hat{k}_{12},\hat{k}_{13})
  $$
    and using that $ u $ is real-valued we get
   \begin{eqnarray*}
  C^{21,low}_k &=  &  k |\widehat{u}(k)|^2 \Im \int_0^t  \Bigl[
  \sum_{ |\hat{k}_2|\vee |\hat{k}_3|\vee |\hat{k}_{12}|\vee |\hat{k}_{13}|\le k^{\frac{2}{3}} \atop \hat{k}_{12}+\hat{k}_{13}=-\hat{k}_2-\hat{k}_3\neq 0 }
     \frac{1}{\hat{k}_2+\hat{k}_3} \widehat{u}(-\hat{k}_{12}) \widehat{u}(-\hat{k}_{13}) \widehat{u}(-\hat{k}_2) \widehat{u}(-\hat{k}_3)\Bigr] \\
     & =& k |\widehat{u}(k)|^2 \Im  \int_0^t \Bigl[
  \sum_{ |\hat{k}_2|\vee |\hat{k}_3|\vee |\hat{k}_{12}|\vee |\hat{k}_{13}|\le k^{\frac{2}{3}} \atop \hat{k}_{12}+\hat{k}_{13}=-\hat{k}_2-\hat{k}_3\neq 0 }
     \frac{1}{\hat{k}_2+\hat{k}_3} \overline{\widehat{u}(\hat{k}_{12}) \widehat{u}(\hat{k}_{13}) \widehat{u}(\hat{k}_2) \widehat{u}(\hat{k}_3)}\Bigr] \\
     & =& - C^{21,low}_k \, ,
  \end{eqnarray*}
 which ensures that $C^{21,low}_k=0 $.\\
   
\noindent   {\bf $\bullet $  Estimate for $ B^4_k$}.
  By symmetry we can assume $ |k_{41}| \ge |k_{42}|\ge |k_{43} |$.
 $ B^4_k $ can be controlled exactly as $ B^1_k $ and is even easier (see \eqref{prod4-est.4} and \eqref{ref2}) except for the treatment of the  region ($ k_{41}=-k_1 $ and
 $  |k_{42}|\vee |k_{43}|\le  k^{\frac{2}{3}} $)  which is slightly different to the treatment of the region ($ k_1=-k_4=k $ and
     $|k_{12}|\vee |k_{13}|\le  k^{\frac{2}{3}})$ for $ B^1_k$. We thus only consider the region $k_{1}+ k_{41} = 0$ and $ |k_{42}|\vee |k_{43}|\le  k^{\frac{2}{3}} $. In this region,  according to \eqref{rr}, we can decompose $ B^4_k $ as
 \begin{eqnarray*}
B^4_k & = & \Im \int_0^t  \Bigl(  \sum_{\Lambda(k)}
 \Bigl[ \Bigl(\frac{k^2}{(k_1+k_2)(k_1+k_3)}-1\Bigr)+1\Bigr]  \\
 & & \hspace*{2cm} \frac{k}{(k_2+k_3)}  \widehat{u}(k_1) \widehat{u}(k_2) \widehat{u}(k_3) \widehat{u}(-k_1)  \widehat{u}(k_{42}) \widehat{u}(k_{43}) \Bigr)\\
& = &B^{41}_k+B^{42}_k \, ,
\end{eqnarray*}
with
\begin{displaymath}
\begin{split}
\Lambda(k)&=\big\{(k_1,k_2,k_3,k_{42},k_{43})\in \Z^5 \; : \; k_1+k_2+k_3=k,\\ & \quad \quad \quad \quad k_{42}+k_{43}=-k_2-k_3\neq 0, \,  |k_2|\vee|k_3|\vee|k_{42}|\vee|k_{43}|\le  k^{\frac{2}{3}}\big\}\; .
\end{split}
\end{displaymath}
$ B^{41}_k $ can be easily estimated as $ C^{1}_k$ (actually it is is even easier) by using \eqref{prod4-est.4} and the fact that   $ |k_{42}+k_{43}|=|k_2+k_3|$. \\
Finally, we claim that $ B^{42}_k =0 $. Indeed, performing the change of variables $\varphi \; :\; (k_1,k_2,k_3,k_{42},k_{43})\mapsto (k_1, -k_{42}, -k_{43}, -k_2,-k_3)
=(\hat{k}_1,\hat{k}_2,\hat{k}_3,\hat{k}_{42},\hat{k}_{43}) $ and noticing that $ \hat{k}_1+\hat{k}_2+\hat{k}_3=k_1+(-k_{42}-k_{43})=k_1+k_2+k_3=k $  ensures that
 $ \varphi $  is a bijection on $ \Lambda(k) $, we get
 \begin{eqnarray*}
B^{42}_k & = & \Im \int_0^t  \Bigl(  \sum_{\Lambda(k)}   \frac{k}{\hat{k}_2+\hat{k}_3}  |\widehat{u}(\hat{k}_1)|^2 \widehat{u}(-\hat{k}_{42}) \widehat{u}(-\hat{k}_{43})
 \widehat{u}(-\hat{k}_2)
 \widehat{u}(-\hat{k}_3)  \Bigr) \\
  & = & \Im \int_0^t  \Bigl(  \sum_{\Lambda(k)}   \frac{k}{\hat{k}_2+\hat{k}_3}  |\widehat{u}(\hat{k}_1)|^2 \overline{\widehat{u}(\hat{k}_{42}) \widehat{u}(\hat{k}_{43})
 \widehat{u}(\hat{k}_2)
 \widehat{u}(\hat{k}_3)}  \Bigr)\\
 &=& - B^{42}_k\ \; .
\end{eqnarray*}
which ensures that $ B^{42}_k=0 $ and  completes the proof of Proposition \ref{prop58}.
\end{proof}
\begin{remark}\label{opti}
For the same reasons explain in (\cite{NTT}, Remark 3.2) our method of proof of the smoothing effect seems to break down for $s<1/3$. The reason is that  the term $ A^1_k $ can neither  be controlled  for $s<1/3 $ nor  be canceled by adding a term of order 7 in the modified energy. Indeed, it is shown in \cite{NTT} that for any $ k $ large enough one can find many couples of triplets
$ (\vec{k}_{(3)}, \vec{k}_{1(3)}) $ such that $ \vec{k}_{(3)}\in D^1(k) $, $ \vec{k}_{1(3)}\in D^1(k_1) $ and $ |\Omega_5(\vec{k}_{1(3)},k_2,k_3,-k)| \lesssim 1 $. Therefore, a supplementary term in the modified energy will not be useful to treat this term since we would not be able to control this term for $ s<1/3 $ and the \lq\lq nonlinear contribution\rq\rq \, of the time derivative of this term
would be even worst.

On the other hand, note that even if we only give an estimate of $ A^{1,0} $ for $ s\ge 1/3 $, we could  lower the Sobolev index here  by adding a supplementary term in the modified energy. This is due to the fact that on the support of  $ A^{1,0} $ we have $|\Omega_5(k_1,-k_1,k_1,k_2,k_3,k_4)|=|\Omega_3(k_1,k_2,k_3)|\gtrsim |k| $.
\end{remark}
The following corollary of Theorem \ref{theo56} will be crucial for the local well-posedness result.
\begin{corollary}
Assume that $ s\ge 1/3$, $0<T \le 1$ and $u,v\in Z^s_T $ are two solutions to \eqref{remKdV} defined in the time interval $[0,T]$. Then, for all integer number $ k\ge 2^9  $  such that $ |\widehat{u}(0,k)|=|\widehat{v}(0,k)| $ and all $ 0	<s'<s $, it holds
\begin{equation}\label{eqcoro51}
\sup_{t\in ]0,T[} k^{1+s'-s}\Bigl| |\widehat{u}(t,k)|^2-|\widehat{v}(t,k)|^2 \Bigr|\lesssim  \|u-v\|_{Z^{s'}_T}  (\| u\|_{Z^s_T}+\| v\|_{Z^s_T})^3 (1+\| u\|_{Z^s_T}+\| v\|_{Z^s_T})^4 \, ,
\end{equation}
where the implicit constant is independent of $ k$.
\end{corollary}
\begin{proof} Proceeding as in the proof of Proposition \ref{prop58}  we obtain  \eqref{prop-ee.3} for $ u $ and for $v$. Taking the difference of these two identities and
 estimating the right-hand side member as in Proposition \ref{prop58} and estimating the non quadratic terms of the modified energy as in Lemma \ref{lem57}, the triangular inequality leads for any $ k\ge 1  $ to
\begin{eqnarray*}
\sup_{t\in ]0,T[} k \Bigl| |\widehat{u}(t,k)|^2-|\widehat{v}(t,k)|^2  \Bigr| & \lesssim & \sup_{N\ge k} \Bigl( \frac{k}{N}\Bigr)^{s-}
 \|P_{\le N} (u-v)\|_{Z^{s}_T} \\
  & &\hspace*{-20mm} (\|P_{\le N} u\|_{Z^s_T}+\| P_{\le N}v\|_{Z^s_T})^3 (1+\| P_{\le N} u\|_{Z^s_T}+\|P_{\le N}  v\|_{Z^s_T})^4 \, .
  \end{eqnarray*}
This last inequality clearly yields \eqref{eqcoro51}
\end{proof}

\section{Estimates for the difference}
We will need the following multilinear estimates of order three and five.
\begin{proposition} \label{L2fivelin}
Assume that   $0<T \le 1$, $\eta_1$, $ \eta_2$ are  bounded functions and $u_i$ are functions in $Z^{0}_{}:=X^{-\frac{11}{10},1}_{} \cap L^\infty_T L^2_x$.  Assume also that $N \gg 1$, $M\ge 1$ and  $j'\in\{1,2,3\}$.
Then
\begin{equation} \label{L2fivelin.2}
\big| G_{\eta_1 ,M}^T(\Pi_{\eta_2,M'}^{j'}(u_1,u_2,u_3)  ,u_4,u_5, u_6) \big| \lesssim T  M M' \prod_{i=1}^6 \|u_i\|_{L^\infty_TL^2_x} \, ,
\end{equation}
where $ G_{\eta,M}^T $ is defined in \eqref{L2trilin.1}.
Let also $ N_1,N_2,N_3\ge 1  $ be dyadic integers and $ (K_1, K_2)\in ]0,+\infty[^2 $ such that $ K_2\gg K_1 $. Then it holds
\begin{align}
\big| G_{\eta_1 \indi_{D^1} \indi_{|\Omega_3|\sim K_1},M}^T
& \Bigl(\Pi_{\eta_2  \indi_{|\Omega_3|\ge K_2} ,M'}^{j'}(P_{N_1} u_1,P_{N_2}u_2,P_{N_3} u_3),u_4,u_5,P_N u_6\Bigr) \big| \nonumber \\
&  \lesssim  \frac{T^{\frac18} }{K_2} M M' \max(N_1,N_2,N_3)^{\frac{11}{10}}\prod_{i=1}^6\|u_i\|_{Z^{0}_T} \, ,\label{L2fivelin.3}
\end{align}
where $ D^1 $ and $ \Pi_{\eta,M}^{j'} $  are  defined in respectively \eqref{defA} and \eqref{def.pseudoproduct2}.
Moreover, the implicit constant in estimates \eqref{L2fivelin.2} and \eqref{L2fivelin.3} only depends on the $L^{\infty}$-norm of the functions $\eta_1$ and $\eta_2$.
 \end{proposition}
 \begin{proof}
 \eqref{L2fivelin.2} follows by using twice \eqref{pseudoproduct.1}. To prove  \eqref{L2fivelin.3}, we first notice that  $K_2\gg K_1$  and \eqref{res55} ensure that
  $|\Omega_5(\vec{k}_{(5)})| \sim |\Omega_3(k_1,k_2,k_3) |\ge K_2 $. Then the result follows by proceeding as in the proof of
   \eqref{L2trilin.2}  with $ R= [K_2/\max(N_1,N_2,N_3)^\frac{11}{10}]^\frac{8}{7}$ .  \end{proof}
\subsection{Definition of the modified energy for the difference}
Let $N_0 \ge 2^9$,  $N$ be a nonhomogeneous dyadic number and $ (u,v)\in H^s(\T)^2 $ with $s\in \R$. We define the modified energy of the difference at the dyadic frequency $N$ by
\begin{equation} \label{defEN}
  \mathcal{E}_N[u,v,N_0]= \left\{ \begin{array}{ll} \frac 12 \|P_N (u-v)\|_{L^2_x}^2 & \text{for} \ N \le N_0 \, \\
\frac 12 \|P_N (u-v)\|_{L^2_x}^2 + \mathcal{E}_N^{3}[u,v]  & \text{for} \ N> N_0 \, , \end{array}\right.
\end{equation}
where
 \begin{eqnarray*}
\mathcal{E}_N^{3}[u,v] &= &   \sum_{k\in \Z}\sum_{\vec{k}_{(3)}\in {D^{1}(k)}\atop m_{min}  \le N^{\frac12}}
\frac{k}{\Omega(\vec{k}_{(3)})}  \varphi_N^2(k) \\
 & &\Re \Bigl[ \Bigl( \hat{u}(k_1) \hat{u}(k_2)
 + \hat{u}(k_1) \hat{v}(k_2)+ \hat{v}(k_1) \hat{v}(k_2)\Bigr) (\hat{u}-\hat{v})(k_3)(\hat{u}-\hat{v})(-k) \Bigr]
\end{eqnarray*}
where $\vec{k}_{(3)}=(k_1,k_2,k_3)$, $ m_{min} = min_{1\le i \neq j\le 3}(|k_i+k_j|)$
 and where the set $ A_3$ is defined in \eqref{m2}.  The modified energy $ E^{s'}[u,v,N_0] $ of the difference $u-v$ is defined by
$$
E^{s'}[u,v,N_0]=\sum_{N\ge 1}   N^{2s'} \mathcal{E}_N[u,v,N_0] \, .
$$	
The following lemma ensures that $E^{s'}[u,v,N_0] $ is well-defined as soon as $ (u,v)\in H^s(\T)^2 $ with $s>0$. Moreover,  for $N_0>2^9 $ large enough  we have $ E^{s'}[u,v,N_0] \sim \|u-v\|_{H^{s'}}^2 $.

\begin{lemma} Let $ (u,v)\in H^s(\T)^2 $ with $ s>0 $. Then, for any $ s'\in \R $ and any $ N_0 \gg ( \|u\|_{ H^s}+ \|v\|_{H^s} )^{1/s} $, it holds
\begin{equation}\label{coercivew}
 C^{-1}  \|u-v\|_{H^{s'}}^2 \le E^{s'}[u,v,N_0]\le C \|u-v\|_{H^{s'}}^2
\end{equation}
for some constant $ C>1 $.
\end{lemma}
\begin{proof}
Let us recall that  Lemma \ref{inter} ensures that $ |k_1|\sim |k_2|\sim |k_3|\sim |k| $ on  $ D^{1}  $. Therefore, a direct application of \eqref{pseudoproduct.1} leads to
\begin{displaymath}
\begin{split}
N^{2s'} |\mathcal{E}_N^{3}[u,v]|\lesssim & \sum_{ M_{min}\ge 1}
\frac{ N^{2s' } N M_{min}}{M_{min}^2 N} N^{-2s'} N^{-2s} \\ & \quad \times (\|P_{\sim N} u\|_{H^s}^2+\|P_{\sim N} v\|_{H^s}^2) \|P_{\sim N} (u-v)\|_{H^{s'}}^2
\, .
\end{split}
\end{displaymath}
Summing over $ M_{min} $ and $ N\ge N_0 $, we obtain
$$
\sum_{N\ge N_0} N^{2s'} |\mathcal{E}_N^{3}[u,v]|\lesssim N_0^{-2s} (\|u\|_{ H^s}^2+ \|v\|_{H^s}^2) \|u-v\|_{H^{s'}}^2 \;.
$$
that clearly implies \eqref{coercivew} for $ N_0 \gg ( \|u\|_{ H^s}+ \|v\|_{H^s} )^{1/s} $.
\end{proof}
Let now $(u,v) $ be a couple of solutions to the renormalized mKdV equation on $ ]0,T[ $.
The following proposition enables  to control   $E^{s'}[u,v,N_0]$ on $ ]0,T[ $.
\begin{proposition} Let $ 0<T<1 $. Let $ u $ and $ v$ be two solutions of the renormalized mKdV \eqref{remKdV}  belonging to $ L^\infty(0,T : H^s(\mathbb T))$ with $ 1/3\le s <1/2$ and associated with the same initial data $ u_0\in H^s(\T) $. Then, for $ s/2<s'<s-\frac{1}{10} $ and any $ N_0\ge 2^{10} $,  it holds
\begin{equation}\label{estmodifiedw}
\sup_{t\in [0,T]} E^{s'}[u(t),v(t),N_0] \lesssim T^{\frac18} N_0^{\frac32} (1+\|u\|_{Z^s_{T}}+\|v\|_{Z^s_{T}})^8 \|w\|^2_{Z^{s'}_{T}} \, ,
\end{equation}
where we set $ w=u-v$.
\end{proposition}
\begin{proof}
To simplify the notation, we denote $  \mathcal{E}_N[u(t),v(t),N_0]$ simply by $ \mathcal{E}_N(t)$. Note that $ u(t) $ and $v(t) $ are well defined for any
 $ t\in [0,T] $ since, by the equation, $(u,v)\in (C([0,T];H^{s-3}))^2 $ and that, for any $ N\in 2^{\N} $, $ E_N(0)=0 $ since $ u(0)=v(0)=u_0$.
For $ N\le N_0 $, the definition of $  \mathcal{E}_N(t)  $ easily leads to
$$
\frac{d}{dt}  \mathcal{E}_N(t)=\int_{ \T}  P_N \Bigl(w(u^2+uv+v^2)- 3 P_0(w(u+v))u -3P_0(v^2) w\Bigr) \partial_xP_N w
$$
which yields after applying Bernstein inequalities,  integrating on $ ]0,t [ $ and summing over $ N\le N_0 $,
\begin{align*}
\sum_{N\le N_0} \mathcal{E}_N(t)   \lesssim  \; & N_0^{1+2s'} \Bigl( \|w\|_{L^\infty_T L^3_x}^2  (\|u\|_{L^\infty_T L^6_x}^2+\|v\|_{L^\infty_T L^6_x}^2)\\
&+ \|w\|^2_{L^\infty_T L^2_x} (\|u\|_{L^\infty_T L^2_x}^2+\|v\|_{L^\infty_T L^2_x}^2)\Bigr) \\
  \lesssim \;   & T \,  N_0^{1+2s'} \|w\|_{L^\infty_T H^{\frac16}_x}^2  (\|u\|_{L^\infty_T H^{\frac13}_x}^2+\|v\|_{L^\infty_T H^{\frac13}_x}^2)\\
    \lesssim \;   &T \, N_0^{\frac32} \|w\|_{L^\infty_T H^{s'}_x}^2  (\|u\|_{L^\infty_T H^{s}_x}^2+\|v\|_{L^\infty_T H^s_x}^2),
\end{align*}
since, by hypotheses, $ 1/6\le s'<1/4 $ and $s\ge 1/3$.

Now for $ N>N_0$, we first notice that the difference $ w=u-v $ satisfies
 \begin{eqnarray}
 w_t + \partial_x^3 w  &=&- \partial_x A(u,u,w)- \partial_x A(u,v,w)- \partial_x A(v,v,w)\nonumber\\
  && +\partial_x \Bigl( B(u,u,w)+(B(u,u,v)-B(v,v,v)) \Bigr)  \, ,\label{y}
 \end{eqnarray}
 where $ A$ and $ B $ are defined in \eqref{defAB}.
 Therefore, differentiating $ \mathcal{E}_N $ with respect to time and integrating between $ 0 $ and $ t $  yields
  \begin{eqnarray*}
  N^{2s'}   \mathcal{E}_N(t)&=&  -N^{2s'} \int_0^t \Re \Bigl( \int_{\T} \partial_x P_N [ A(u,u,w)+ A(u,v,w)+A(v,v,w)] P_N w \, d\tau \Bigr)\nonumber\\
  && +N^{2s'} \int_0^t \Re \Bigl( \int_{\T} \partial_x P_N  B(u,u,w) P_N w \, d\tau\Bigr) \nonumber \\
  & & +N^{2s'} \int_0^t \Re \Bigl( \int_{\T} \partial_x P_N(B(u,u,v)-B(v,v,v))P_N w \, d\tau \Bigr)  \nonumber\\
   & &+N^{2s'}\int_0^t  \Re \frac{d}{dt} \mathcal{E}_N^{3}(\tau)d\tau  \nonumber\\
  & = &  C_N(t)+D_N(t)+F_N(t) +G_N(t)\, .
 \end{eqnarray*}

 As in \eqref{eee.3} we notice that, since $ u $ and $ v$ are real-valued,
 $$
 \int_{\T} \partial_x  P_N B(u,u,w) P_N w =ik \sum_{k\in \Z} |\hat{u}(k)|^2 |\varphi_N(k) \hat{w}(k)|^2 \in i \R \; .
 $$
 and thus $ D_N(t)=0 $.
 On the other hand, the smoothing effect \eqref{eqcoro51} leads to
 \begin{eqnarray*}
 |F_N(t)| & \lesssim  &
 N^{2s'}\Bigl| \int_0^t  k \sum_{k\in \Z}( |\hat{u}(\tau,k)|^2-|\hat{v}(\tau,k)|^2) \varphi_N(k)^2 \hat{v}(\tau,k) \hat{w}(\tau,-k) d\tau\Bigr| \\
 & \lesssim & \sup_{\tau\in [0,T]}\sup_{k\in \Z \atop \frac{N}{2}\le |k|\le 2 N}\Bigl( |k|^{1+(s'-s)}\Bigl| |\hat{u}(\tau,k)|^2-|\hat{v}(\tau,k)|^2\Bigr|\Bigr)
\\ && \quad \times \int_0^t \|v_N(\tau)\|_{ H^s}
  \|w_N(\tau)\|_{H^{s'}} \, d\tau \\
  & \lesssim & \delta_N  \, T (1+\|u\|_{Z^{s}_T}^7+\|v\|_{Z^{s}_T}^7)  \|w\|_{Z^{s'}_T}  \|v\|_{L^\infty_T H^{s}}\|w\|_{L^\infty_T H^{s'}}
 \end{eqnarray*}
 with $ \|(\delta_{2^j})_j\|_{l^1(\N)} \lesssim 1$.  
 
 It thus remains to control $ C_N(t)+G_N(t) $. We notice that $ C_N(t) $ can be decomposed as
 \begin{eqnarray*}
 C_N(t) & = & \sum_{j=1}^3 N^{2s'}   \Re
\int_{ [0,t]\times  \T} \partial_x P_N \Bigl(  \Pi_{\indi_{D^1}}+\Pi_{\indi_{D^2}}\Bigr) (z_1^j,z_2^j,w) P_N w  \\
&= & C_N^1(t)+C_N^2(t) \; ,
\end{eqnarray*}
with $ (z_1^1,z_2^1)=(u,u)$, $ (z_1^2,z_2^2)=(u,v) $ and $(z_1^3,z_2^3)=(v,v) $. We then decompose further 
 $ C_N^1(t) $ and $ C_N^2(t) $ as  
  \begin{eqnarray*}
 C_N^1(t) &= &  \sum_{j=1}^3 N^{2s'}   \Re
\int_{ [0,t]\times  \T} \partial_x P_N \Bigl( \Pi_{\indi_{D^1\cap \{m_{min}>N^{1/2}\}}}+ \Pi_{\indi_{D^1\cap \{m_{min}\le N^{1/2}\}}}\Bigr) (z_1^j,z_2^j,w) P_N w \\
& =& C_N^{1,low}(t)+C_N^{1,high}(t) 
\end{eqnarray*}
and
  \begin{eqnarray*}
 C_N^2(t) &= &  \sum_{j=1}^3 N^{2s'}  \Re
\int_{ [0,t]\times  \T} \partial_x P_N \Bigl( 
\Pi_{\indi_{D^2\cap \{|k_1|\vee |k_2|\ll N\}}}+ \Pi_{\indi_{D^2\cap \{|k_1|\vee |k_2|\gtrsim N\}}}\Bigr) (z_1^j,z_2^j,w) P_N w \\
& =& C_N^{2,low}(t)+C_N^{2,high}(t) \; .
\end{eqnarray*}

 \noindent {\bf $ \bullet $ Estimate for $ C_N^{2,low} $}.  
From Lemma \ref{inter}, we  infer that for any $ N\ge 1 $, 
$$ D^2\cap \{|k_1|\vee |k_2| \ll N\} \cap \{\sum_{i=1}^3 k_i\in \supp \Phi_N \}= D\cap \{|k_1|\vee |k_2| \ll N\}  \cap \{\sum_{i=1}^3 k_i\in \supp \Phi_N \}\; .
$$
  Since $|k_1|\vee |k_2|\ll N\sim |k_3| $ on the support of $ C_N^{2,low}$ , it can thus be rewritten as 
$$
 C_N^{2,low}(t)=   \sum_{j=1}^3 N^{2s'} \sum_{1\le M \ll N}  \Re
\int_{ [0,t]\times  \T} \partial_x P_N  
\Pi^3_{\indi_{D\cap \{|k_1|\vee |k_2|\ll N\}},M}(z_1^j,z_2^j,w) P_N w\; .
$$
Since $ \indi_{D\cap \{|k_1|\vee |k_2|\ll N\} } $ satisfies the symmetry hypotheses of Lemma \ref{technical.pseudoproduct}, it can further be rewritten as
$$
 C_N^{2,low}(t)=   \sum_{j=1}^3 N^{2s'} \sum_{1\le M \ll N}   M \Re
\int_{ [0,t]\times  \T} P_N  
\Pi^3_{\eta \indi_{D\cap \{|k_1|\vee |k_2|\ll N\}},M}(z_1^j,z_2^j,w) P_N w
$$
with $ \|\eta\|_{\infty} \lesssim 1$.  Using  Lemma \ref{inter}, Proposition \ref{L2trilin} and \eqref{res3} it thus can be estimated by 
\begin{eqnarray*}
|C_N^{2,low}(t)| & \lesssim & T^{1/8}\sum_{1\le M\ll N} \sum_{N_1\vee N_2\gtrsim M} N^{2s'}  \frac{M^2 N_{max}^\frac{11}{10}}{M N_{max}^2} N^{-2s'}  N_1^{-s} N_2^{-s}  \\ && \quad \times \|P_{N_1}z_1\|_{Z^s_T}
\|P_{N_2} z_2\|_{Z^s_T} \|P_{\sim N} w\|_{Z^{s'}_T}^2\\
 & \lesssim&  T^{1/8} N^{\frac{1}{10}-s}  (\|u\|_{Z^s_T}^2+\|v\|_{Z^s_T}^2) \|w\|_{Z^{s'}_T}^2 \, ,
\end{eqnarray*}
which is acceptable for $s>\frac{1}{10}$. \\
 \noindent {\bf $ \bullet $ Estimate for $ C_N^{2,high} $}.   On account of  Proposition \ref{L2trilin}  we have
 \begin{eqnarray*}
 | C_N^{2,high}(t)|& \lesssim &  T^{1/8} \sum_{N_3\ge 1} \sum_{N_1\vee N_2\gtrsim N} \sum_{1\le M\le N_{max}}
 N^{2s'} \frac{M N N_{max}^\frac{11}{10}}{M N^2_{max}} N^{-s'} N_1^{-s} N_2^{-s} N_3^{-s'} \\
  & & \hspace*{20mm} \times  \|P_{N_1}z_1\|_{Z^s_T}
\|P_{N_2} z_2\|_{Z^s_T} \| P_{N_3}  w\|_{Z^{s'}_T}\ \| P_N w\|_{Z^{s'}_T}\\
& \lesssim & T^{1/8}  N^{-s+s'+\frac{1}{10}+}(\|u\|_{Z^s_T}^2+\|v\|_{Z^s_T}^2) \|w\|_{Z^{s'}_T}^2\;,
 \end{eqnarray*}
 which is acceptable since $ s'<s-\frac{1}{10} $. \\
 \noindent {\bf $\bullet $  Estimate for $ C_N^{1,high} $}.  
  Performing a dyadic decomposition in $ m_{min}\sim M  $ ,  Lemma \ref{inter}, Proposition \ref{L2trilin} and \eqref{res3} lead to
  \begin{eqnarray*}
 |C_N^{1,high}| & \lesssim &  T^{1/8}\, \sum_{M\gtrsim N^{\frac12}}
 N^{2s'} \frac{M N N^\frac{11}{10}}{M^2 N } N^{-2s'} N^{-2s}  \\ && \quad \times (\|P_{\sim N}u\|_{Z^s_T}^2+\|P_{\sim N} v\|_{Z^s_T}^2) \|P_{\sim N} w\|_{Z^{s'}_T}^2\\
& \lesssim &  T^{1/8}\, N^{-2s+\frac{3}{5}+} (\|u\|_{Z^s_T}^2+\| v\|_{Z^s_T}^2) \| w\|_{Z^{s'}_T}^2 \, ,
 \end{eqnarray*}
which is acceptable for $ s>\frac{3}{10} $.\\
\noindent {\bf $\bullet $  Estimate for $ C_N^{1,low}+G_N $}.  
We have
\begin{displaymath}
G_N(t)  =  N^{2s'}\Re \int_0^t  \frac{d}{dt}  \mathcal{E}_N^{3}(t)=: - C_N^{1,low}+ A_{N,1}+A_{N,2}+A_{N,3}  \, ,
\end{displaymath}
where
\begin{eqnarray*}
A_{N,1} & = & N^{2s'} \int_0^t \sum_{k\in \Z} \sum_{\vec{k}_{(3)} \in D^1(k)\atop m_{min} \le N^{\frac12}}
 \sum_{1\le q\neq q'\le 2} \sum_{i=1}^3
  \frac{k \varphi_N^2(k) k_q}{\Omega^3(\vec{k}_{(3)})}   \\ && \times \Im\Bigl[ \hat{z}_{q',i}(k_{q'})\hat{w}(k_3)\hat{w}(-k)
     \Bigl(3|\widehat{z}_{q,i}(k_q)|^2 \widehat{z}_{q,i}(k_q)+\sum_{\vec{k}_{q(3)}\in D(k_q)}
 \prod_{j=1}^3 \widehat{z}_{q,i}(k_{q,j}) \Bigr) \Bigr] \, ,
 \end{eqnarray*}
 
 \begin{eqnarray*}
A_{N,2} & = & N^{2s'} \int_0^t \sum_{k\in \Z} \sum_{\vec{k}_{(3)}\in  D^1(k)\atop m_{min} \le N^{\frac12}} \sum_{i=1}^3
  \frac{k \varphi_N^2(k) k_3}{\Omega^3(\vec{k}_{(3)})}   \\ && \times \Im \Bigl[   \hat{w}(-k) \widehat{z}_{1,i}(k_1) \widehat{z}_{2,i}(k_2)  \Bigl((|\widehat{u}(k_3)|^2+|\widehat{v}(k_3)|^2) \widehat{w}(k_3)+\widehat{w}(-k_3)\widehat{u}(k_3)\widehat{v}(k_3) \\
 & & \quad \quad +
 \sum_{\vec{k}_{3(3)}\in D(k_3)}  [\widehat{u}(k_{3,1})\widehat{u}(k_{3,2})+\widehat{u}(k_{3,1})\widehat{v}(k_{3,2})+\widehat{v}(k_{3,1})\widehat{v}(k_{3,2})] \widehat{w}(k_{3,3})\Bigl)
 \Bigr]\\
 \end{eqnarray*}
 and
\begin{eqnarray*}
A_{N,3} & =&N^{2s'} \int_0^t \sum_{k\in \Z} \sum_{\vec{k}_{(3)}\in  D^1(k)\atop m_{min} \le N^{\frac12}}\sum_{i=1}^3
  \frac{k^2 \varphi_N^2(k)}{\Omega^3(\vec{k}_{(3)})}   \\ && \times  \Im\Bigl[     \widehat{z}_{1,i}(k_1) \widehat{z}_{2,i}(k_2)  \hat{w}(k_3) \Bigl((|\widehat{u}(k)|^2+|\widehat{v}(k)|^2) \widehat{w}(-k)+\widehat{w}(k)\widehat{u}(-k)\widehat{v}(-k) \\
 & & \quad \quad +
 \sum_{\vec{k}_{4(3)}\in D(-k)}   [\widehat{u}(k_{4,1})\widehat{u}(k_{4,2})+\widehat{u}(k_{4,1})\widehat{v}(k_{4,2})+\widehat{v}(k_{4,1})\widehat{v}(k_{4,2})] \widehat{w}(k_{4,3})\Bigl)
 \Bigr] \, ,
\end{eqnarray*}
where we set $ (z_{1,1},z_{2,1})=(u,u) $, $ (z_{1,2},z_{2,2})=(u,v) $ and $(z_{1,3},z_{2,3})=(v,v) $.
Hence,
$$
C_N^{1,ow}+G_N=\sum_{j=1}^3  A_{N,j}   \, .
$$
For any sextuplet $\vec{N}=(N_{1,1},N_{1,2},N_{1,3},N_2,N_3,N)  \in (2^{\N})^6 $,  any $\vec{z}=(z_{1,1},z_{1,2},z_{1,3},z_2,z_3,z_4) \in (Z^s_T)^6 $ and any function
 $ \eta \in L^\infty(\R) $, 
 we set
\begin{eqnarray*}
R_{\vec{N},\eta}(\vec{z})&=& N^{2s'+2} \Bigl| \int_0^t  \sum_{k\in \Z} \sum_{\vec{k}_{(3)}\in D^1(k)\atop m_{min} \le N^{\frac12}}  \frac{\eta(k_1,k_2,k_3)}{\Omega^3(\vec{k}_{(3)})}\widehat{P_{N_2}z_2}(k_2)
\widehat{P_{N_3}z_3}(k_3)\widehat{P_{N}z_4}(-k) \\
 & & \hspace*{30mm} \times \widehat{P_{N_{1,1}}z_{1,1}}(k_{1})\widehat{P_{N_{1,2}}z_{1,2}}(-k_{1})\widehat{P_{N_{1,3}}z_{1,3}}(k_{1}) \Bigr|
\end{eqnarray*}
and
\begin{eqnarray*}
S_{\vec{N},\eta}(\vec{z})&=& N^{2s'+2}\Bigl| \int_0^t  \sum_{k\in \Z} \sum_{\vec{k}_{(3)}\in  D^1(k)\atop m_{min} \le N^{\frac12}}  \frac{\eta(k_1,k_2,k_3)}{\Omega^3(\vec{k}_{(3)})} \widehat{P_{N_2}z_2}(k_2)
\widehat{P_{N_3}z_3}(k_3)\widehat{P_{N}z_4}(-k) \\
& & \hspace*{30mm} \sum_{\vec{k_1}_{(3)}\in D(k_1)}
 \widehat{P_{N_{1,1}}z_{1,1}}(k_{{1,1}})\widehat{P_{N_{1,2}}z_{1,2}}(k_{1,2})\widehat{P_{N_{1,3}}z_{1,3}}(k_{1,3}) \Bigr| \, .
\end{eqnarray*}
We observe that to get the desired estimates for the $A_{N,j}$, it suffices to prove  that for any sextuplet $\vec{N}=(N_{1,1},N_{1,2},N_{1,3},N_2,N_3,N)  \in (2^{\N})^6 $,  any $\vec{z}=(z_{1,1},z_{1,2},z_{1,3},z_2,z_3,z_4) \in (Z^s{})^6 $ and any $ \eta \in L^\infty $ with $ \|\eta\|_{\infty}\lesssim 1$,
\begin{equation}\label{eej}
R_{\vec{N},\eta}(\vec{z})+S_{\vec{N},\eta}(\vec{z})\lesssim  T^{\frac18}
 \widetilde{N}_{max}^{-2(s-s')+}\|P_{N} z_4 \|_{Z^s_T} \prod_{i=2}^3 \|P_{N_i} z_i \|_{Z^s_T}  \prod_{j=1}^3 \|P_{N_{1,j}} z_{1,j} \|_{Z^s_T}
 \end{equation}
 where $ \widetilde{N}_{max}=\max(N_{1,1},N_{1,2},N_{1,3},N_2,N_3,N) $. 
 
 Indeed, the modulus of the $A_{N,j} $ are controlled by sums of terms of this form  with 
 $$
  \eta(k_1,k_2,k_3)= \frac{k_4 k_j}{N^2} \indi_{D^1}  \varphi_N(k_4) \,, \quad j=1,2,3,4, \quad -k_4=k_1+k_2+k_3,
  $$
  and  where $ w$
  appears two times in  the components of $ \vec{z} $ and all the other components are $ u $ or $v$. Therefore \eqref{eej} leads to
  \begin{equation}\label{eek}
\sum_{j=1}^3|A_{j,N}|\lesssim T^{\frac18}
 \widetilde{N}_{max}^{0-}
\|P_{ \lesssim \widetilde{N}_{max}} w \|_{Z^{s'}_T}^2 \Bigr( \|P_{ \lesssim \widetilde{N}_{max}} u \|_{Z^s_T} +
 \|P_{ \lesssim \widetilde{N}_{max}}v \|_{Z^s_T} \Bigl)^4\;
 \end{equation}
and \eqref{estmodifiedw} then follows by summing over $ (N_{1,1},N_{1,2},N_{1,3},N_2,N_3,N) $ thanks to the factor $  \widetilde{N}_{max}^{0-}$.

 To simplify the notation, we denote
  $ P_{N_i}z_i $, $ P_{N_1,j} z_{1,j} $ and  $P_N z_4 $  by respectively $ z_i $ and $ z_{1,j} $ and $ z_4$ in the sequel.\\
  
\noindent  {\bf $\bullet $ Estimate for $R_{\vec{N},\eta}(\vec{z})$}.
 We recall that  on $ D^1 $ we must have $ |k_1|\sim |k_2|\sim |k_3|\sim |k| $ and  thus $ R_{\vec{N},\eta}(\vec{z}) $ vanishes except if
  $N_{1,1}\sim N_{1,2}\sim N_{1,3}\sim  N_2\sim N_3\sim N $. In particular, $ \widetilde{N}_{max}\sim N $. Then, estimate \eqref{pseudoproduct.1}  leads to
  \begin{eqnarray*}
R_{\vec{N},\eta}(\vec{z}) & \lesssim  & \sum_{M\ge 1 } T
 \frac{N^{2s'+2} M}{M^2  N} N^{-6s}
\prod_{i=2}^4 \|z_i \|_{L^\infty_T H^s_x}  \prod_{j=1}^3\|z_{1,j} \|_{L^\infty_T H^s_x}  \\
& \lesssim & T \widetilde{N}_{max}^{-2(s-s')}  \widetilde{N}_{max}^{1-4s} \prod_{i=2}^4  \| z_i \|_{L^\infty_T H^s_x}  \prod_{j=1}^3\|z_{1,j} \|_{L^\infty_T H^s_x} \, ,
\end{eqnarray*}
which is acceptable since for $ s > 1/4$.\\

\noindent  {\bf $\bullet $ Estimate for $ S_{\vec{N},\eta}(\vec{z})$}.  We set $\vec{k}_{1(3)}=(k_{1,1},k_{1,2},k_{1,3})$. By symmetry, we may assume that $ N_{1,1}\ge N_{1,2}\ge N_{1,3} $ so that  $ N_{1,1}\sim  \widetilde{N}_{max}$.  We separate different contributions.\\
{\bf Case 1: $M_{1,med}\ge 2^{-9} N$.} \\
{\bf  Case 1-1: $ |\Omega_3(\vec{k}_{1(3)})|\gg  |\Omega_3(\vec{k}_{(3)})|$.} Noticing that  $ |\Omega_3(\vec{k}_{(3)})|\sim M_{min} M_{med} N $ and
 $  |\Omega_3(\vec{k}_{1(3)})|\gtrsim 2^{-9}  M_{1,min} N N_{1,1}$ in this region,   \eqref{L2fivelin.3} leads to
\begin{eqnarray*}
S_{\vec{N},\eta}(\vec{z}) & \lesssim  &\sum_{1\le  M_{1,min}\lesssim N_{1,2}} \sum_{M_{med}\gtrsim M_{min}\ge 1 } T^{\frac18}
 \frac{N^{2s'+2} M_{min}}{M_{min}M_{med}  N^2}
 \widetilde{N}_{max}^{\frac{1}{10}}  \widetilde{N}_{max}^{-2s} N^{-2s} \\ && \quad \times \prod_{i=2}^4 \|z_i \|_{Z^s_T}  \prod_{j=1}^3 \|z_{1,j} \|_{Z^s_T} \\
& \lesssim & T^{\frac18}  N_{max}^{-2(s-s')}  \widetilde{N}_{max}^{-2s'+\frac{1}{10}+}N^{2(s'-s)}   \prod_{i=2}^4 \|z_i \|_{Z^s_T}  \prod_{j=1}^3 \| z_{1,j} \|_{Z^s_T} \, ,
\end{eqnarray*}
which is acceptable since for $ s'>s/2 $ and $ s\ge 1/3$ force   $ s'>1/6$ . \\

\noindent {\bf  Case 1-2: $ |\Omega_3(\vec{k}_{1(3)})|\lesssim |\Omega_3(\vec{k}_{(3)})|$.}
Then $m_{min} \le N^{1/2} $ yields
$$ M_{1,min} \lesssim \frac{M_{min} M_{med} N }{ M_{1,med} N_{1,1} }\lesssim  \frac{M_{med}N^\frac{3}{2}}{ N N_{1,1}} =\frac{M_{med} N^\frac{1}{2}}{N_{1,1}}\; .
$$
Therefore \eqref{L2fivelin.2} leads to
\begin{eqnarray*}
S_{\vec{N},\eta}(\vec{z})& \lesssim  & T \sum_{ 1 \le M_{min}\le M_{med}\lesssim N } \frac{N^{2s'+2} M_{min} M_{med} N^{\frac{1}{2}}}{M_{min} M_{med} N  N_{1,1}} N_{1,1}^{-s}N^{-3s}
  \prod_{i=2}^4  \| z_i \|_{L^\infty_T H^s_x}  \prod_{j=1}^3\|z_{1,j} \|_{L^\infty_T H^s_x}\\
 & \lesssim & T \widetilde{N}_{max}^{-2(s-s')}   \widetilde{N}_{max}^{s-1-2s'} N^{2s'-3s+\frac{3}{2}+}  \prod_{i=2}^4  \| z_i \|_{L^\infty_T H^s_x}  \prod_{j=1}^3\|z_{1,j} \|_{L^\infty_T H^s_x}
\end{eqnarray*}
which is acceptable since $ s'>s/2$ ensures that $ s-1-2s'<0 $ and since
$$
(s-1-2s')+(2s'-3s+\frac{3}{2}+)=-2s+\frac 1 2 + <0 \quad \Leftrightarrow \quad s>\frac14\; .
$$
{\bf  Case 2: $ M_{1,med}< 2^{-9} N$}. Then $ N_{1,1}\sim N_{1,2}\sim N_{1,3}\sim  N_2\sim N_3\sim N $. \\
{\bf  Case 2-1: $ M_{1,min}\le 2^9 M_{med} $}.  Then \eqref{L2fivelin.2} leads to
\begin{eqnarray*}
S_{\vec{N},\eta}(\vec{z})& \lesssim  & T \sum_{ 1\le M_{min} \le M_{med}\lesssim N} \frac{N^{2s'+2} M_{min}M_{med}}{M_{min}M_{med} N}  N^{-6s}
\prod_{i=2}^4  \| z_i \|_{L^\infty_T H^s_x}  \prod_{j=1}^3\|z_{1,j} \|_{L^\infty_T H^s_x}
\\
 & \lesssim & T  \widetilde{N}_{max}^{-2(s-s')} \widetilde{N}_{max}^{-4s+1+} \prod_{i=2}^4  \| z_i \|_{L^\infty_T H^s_x}  \prod_{j=1}^3\|z_{1,j} \|_{L^\infty_T H^s_x} \; ,
\end{eqnarray*}
which acceptable for $ s>1/4 $.\\

\noindent {\bf  Case 2-2: $ M_{1,min}> 2^9 M_{med} $}. Then,
$$
  |\Omega_3(\vec{k}_{1(3)})| \sim M_{1,min} M_{1,med} M_{1,max}\ge 2^{18}M_{min} M_{med} N\gg     |\Omega_3(\vec{k}_{(3)})|\, ,
 $$
and thus  \eqref{L2fivelin.3} leads  to
\begin{eqnarray*}
S_{\vec{N},\eta}(\vec{z}) & \lesssim  & T^{\frac18} \sum_{  M_{min} \ge 1 }
 \sum_{ 1 \le M_{1,med} \lesssim N}
\frac{N^{2s'+2} M_{min} }{M_{min}^2 N \, M_{1,med} } N^{-6s}  N^\frac{1}{10} \\ & &
\ \prod_{i=2}^4 \|z_i \|_{Z^s_T}  \prod_{j=1}^3 \| z_{1,j} \|_{Z^s_T}\\
& \lesssim & T^{\frac18}   \widetilde{N}_{max}^{-2(s-s')}  \widetilde{N}_{max}^{-4s+\frac{11}{10}}
\ \prod_{i=2}^4 \|z_i \|_{Z^s_T}  \prod_{j=1}^3 \| z_{1,j} \|_{Z^s_T} \, ,
 \end{eqnarray*}
 which is acceptable for $ s> \frac{11}{40} $.
\end{proof}

\begin{lemma} \label{trilili}
Assume that $0<T \le 1$, $ s\ge 1/3 $  and $(u,v)\in L^\infty(0,T : H^s(\T))^2$ are two  solution to \eqref{remKdV} associated to the initial data
$(u_0,v_0)\in H^s(\T)^2$. Then, setting $ s'=\frac13-\frac18=\frac{5}{24}$, it holds
\begin{equation} \label{triw}
\|u-v\|_{Z^{s'}_T} \lesssim (1+\|u\|_{L^\infty_T H^s_x}^2+\|v\|_{L^\infty_T H^s_x}^2)^3  \|u-v\|_{L^\infty_T H^{s'}_x}\,
\end{equation}
\end{lemma}
\begin{proof} We proceed as in Lemma \ref{trilin} so that we are reduced to estimate $ \|u-v\|_{X^{s'-\frac{11}{10},1}_T} $.
 Setting $ w=u-v $,  the Duhamel formula associated to \eqref{mKdV}, the standard linear estimates in Bourgain's spaces and the fractional Leibniz rule lead to
\begin{equation} \label{bed}
\begin{split}
\|w\|_{X^{s'-\frac{11}{10},1}_T} &\lesssim \|u_0-v_0\|_{H^{s'}}+ \|\partial_x(w(u^2+uv+v^2))\|_{X^{s'-\frac{11}{10},0}_T}+ \| P_0(u^2) w_x\|_{X^{s'-\frac{11}{10},0}_T} \\
& \quad + \|P_0(u^2-v^2)v_x\|_{X^{s'-\frac{11}{10},0}_T}\\
&  \lesssim \|u-v\|_{L^\infty_T H^{s'}_x}+\|J_x^{s'-\frac{1}{10}} (w(u^2+uv+v^2))\|_{L^2_T L^2_x} +\|u\|_{L^\infty_T L^2_x}^2 \|w\|_{L^2_T H^{s'-\frac{1}{10}}_x}
 \\
 &\quad + (\|u\|_{L^\infty_T L^2_x} +\|v\|_{L^\infty_T L^2_x} )  \|v\|_{L^2_T H^{s'-\frac{1}{10}}_x}    \|w\|_{L^\infty_T L^2_x} \; .
 \end{split}
\end{equation}
Then, we notice that
\begin{align*}
\|J_x^{s'-\frac{1}{10}}  & (w(u^2+uv+v^2))\|_{L^2_T L^2_x}
 \\ & \lesssim  \|J_x^{s'} (w(u^2+uv+v^2)) \|_{L^2_T L^{\frac{5}{3}}_x} \\
 & \lesssim    ( \|u\|_{L^4_T L^{20}_x}^2 + \|v\|_{L^4_T L^{20}_x}^2)  \|J_x^{s'} w  \|_{L^\infty_T L^2_x} \\
& \hspace*{2mm} +  ( \|u\|_{L^4_T L^{20}_x} + \|v\|_{L^4_T L^{20}_x}) \|w\|_{L^\infty_T L^\frac{24}{7}_x}
(\|J^{\frac{5}{24}}_x u \|_{L^4_T L^\frac{120}{31}_x}+\|J^{\frac{5}{24}}_x v \|_{L^4_T L^\frac{120}{31}_x})
\end{align*}
which leads to \eqref{triw} thanks to \eqref{est1L6}-\eqref{est2L6} and Sobolev inequalities since $ H^\frac{5}{24}(\T) \hookrightarrow L^\frac{24}{7}(\T) $ and for $ s\ge 1/3 $, it holds $ s'=s-1/8\ge \frac{5}{24} $ and $ \frac{120}{31}<4 $. \end{proof}
\section{Proof of Theorem \ref{main}} \label{Secmaintheo}
\subsection{Unconditional uniqueness for the renormalized mKdV equation}
Let us start by proving the unconditional uniqueness of \eqref{remKdV}.
 Let $ T>0 $ and  $(v_1,v_2)\in L^\infty(0,T;H^{1/3})^2$ be a couple of functions that satisfies \eqref{remKdV} in the distributional sense with $ v_1(0)=v_2(0)=u_0\in H^s(\T) $ . We first notice that    Lemma \ref{trilin} ensures that
   $(u,v)\in Z^{1/3}_{\tilde T}$ with $ \tilde{T}=\min(1,T) $ and, from Proposition \ref{ee},  we infer that
  $$
  \|v_1\|_{Z^{\frac13}_{\tilde T}}+  \|v_2\|_{Z^{\frac13}_{\tilde T}} \lesssim  \|u_0\|_{H^{\frac13}}+{\tilde T}^{\frac18}  \Bigl(1+\| v_1\|_{L^\infty_{T} H^{\frac13}_x}^3+\| v_2\|_{L^\infty_{ T} H^{\frac13}_x}^3\Bigr)^4\, .
  $$
  Hence, taking  $ \tilde{T}\le \min(1,T, (1+\|u_0\|_{H^{\frac13}})^{-16} )$, we get
  $$
    \|v_1\|_{Z^{\frac13}_{\tilde T}}+  \|v_2\|_{Z^{\frac13}_{\tilde T}} \lesssim  \|u_0\|_{H^{\frac13}}\, .
    $$
  Then, noticing that $ \frac16 <\frac{5}{24} <\frac13-\frac19 $,  \eqref{coercivew}-\eqref{estmodifiedw} and  \eqref{triw} lead  to
  $$
\|v_{1}-v_{2} \|_{L^\infty_T H^{\frac{5}{24}}_x}^2\lesssim  T^{\frac18}N_0^{\frac32} (1+\|u_0\|_{H^{\frac12}})^{25} \|v_1-v_2\|^2_{L^\infty_T H^{\frac{5}{24}}_x}
$$
with $ N_0\gg \|u_0\|_{H^{\frac13}}^{-3} $.
This forces
  $$
  \|v_1-v_2 \|_{L^\infty_{T'} H^{\frac{5}{24}}_x} =0
  $$
  with $T' \sim \min({\tilde T}, (1+\|u_0\|_{H^{\frac13}})^{-300}) $. Hence $ v_1=v_2 $ $ a.e. $ on $ [0,T'] $. Therefore there exists $ t_1\in [T'/2,T']$ such that
   $ v_1(t_1)=v_2(t_1) $ and  $ \|v_1(t_1)\|_{H^{\frac13}} \le \|v_1\|_{L^\infty_T H^{\frac13}_x} $.  Using this bound we can repeat this argument a finite number of times to extend the uniqueness  result on $ [0,T]$.

\subsection{Local well-posedness of the renormalized mKdV equation}
It is known from the classical well-posedness theory that an initial data $ u_0\in H^\infty(\T) $ gives rise to a global solution $u \in C(\mathbb R; H^{\infty}(\mathbb T))$ to the Cauchy problem \eqref{mKdV}. Then combining Lemma \ref{trilin} and Proposition \ref{ee} we infer that $ u $ verifies
\begin{equation}
\| u\|_{L^\infty_T H^s_x}^2  \lesssim \|u_0\|_{H^s}^2 + T^{\frac18}  \Bigl(\| u\|_{L^\infty_T H^s_x}+\| u\|_{L^\infty_T H^s_x}^3\Bigr)^4\,
\end{equation}
for any $ 0<T<1$. Taking $ T=T(\|u_0\|_{H^s})\sim \min (1,(1+\|u_0\|_{H^s})^{-10}) $, the continuity of $ : T\mapsto \|u\|_{L^\infty_T H^s_x} $ ensures that
$$
\| u\|_{L^\infty_T H^s}\lesssim  \|u_0\|_{H^s} \;
$$
and Lemma \ref{trilin} then leads to
\begin{equation}\label{mm}
\|u\|_{Z^s_T} \lesssim \|u_0\|_{H^s}(1+\|u_0\|_{H^s}^2) \;.
\end{equation}
 Moreover, we infer from Theorem \ref{theo56}  that for any $ K\in 2^{\N} $  it holds
\begin{align}
\|P_{\ge K} u(t)\|_{L^\infty_T H^s_x}^2 & \le \sum_{k\ge K} \sup_{t\in [0,T]} |k|^{2s} |\widehat{u}(t,k)|^2\nonumber \\
 &  \lesssim \sum_{k\ge K} \Bigl( |k|^{2s} |\widehat{u_0}(k)|^2+ k^{2s-1}  \|u\|_{Z^s}^4 (1+\|u\|_{Z^s})^4\Bigr) \nonumber\\
 & \le \|P_{\ge K} u_0\|_{H^s}^2 + K^{2s-1}   (1+\|u_0\|_{H^s})^{24}\, . \label{bb}
\end{align}

Now let us fix $ 1/3\le s<1/2$. For  $ u_0\in H^s(\T) $ we set $ u_{0,n}=P_{\le n} u_0 $ and we denote by $ u_n \in C(\R;H^\infty(\T)) $ the
 solutions to \eqref{remKdV} emanating from $u_{0,n} $. In view of  \eqref{mm} we infer that for any $ n\in \N $,
$$
\|u_n\|_{Z^s_T} \lesssim \|u_0\|_{H^s}(1+\|u_0\|_{H^s}^2) \; ,
$$
with $ T=T(\|u_0\|_{H^s}) $,
and \eqref{bb} ensures that
\begin{equation} \label{bbb}
\lim_{K\to +\infty} \sup_{n\in \N} \|P_{\ge K} u_n(t) \|_{L^\infty_T H^s_x} =0 \; .
\end{equation}
This proves that the sequence $\{u_n\} $ is   bounded in  $L^\infty(]0,T[ :H^{s}(\T))$ and thus $ u_n^3 $ is bounded in $ L^\infty(]0,T[:L^2(\T)) $. Moreover, in view of the equation \eqref{remKdV},
 the sequence $ \{\partial_t u_{n}\}$ is bounded in $ L^\infty(]0,T[:H^{-3}(\T))$. By Aubin-Lions compactness theorem we infer that, for any $
 0<T\le T(\|u_0\|_{H^s})$,  $\{u_n\} $ is relatively compact in $ L^2(]0,T[\times\T) $. Therefore, using a diagonal extraction argument, we obtain the existence  of an increasing sequence $\{n_k\}\subset \N $ and $ u\in L^\infty(]0,T[ : H^{s}(\T)) $
  such that
  \begin{align}
   u_{n_k} \rightharpoonup u \mbox{ weak star in } L^\infty(]0,T[:H^{s}(\T)) \label{conv1} \\
   u_{n_k} \to u \mbox{ in }L^2(]0,T[: L^2(\T)) \cap L^3(]0,T[:L^3(\T))   \label{conv2}\\
   u_{n_k} \to u \mbox{ a.e. in } ]0,T[\times \T  \label{conv3} \\
      u_{n_k}^3 \to   u^3 \mbox{ in } L^1(]0,T[:L^1(\T)) \label{conv4}
   \end{align}
 These convergences results enable us to pass to the limit on the equation and to obtain that the limit function $ u$ satisfies
 \eqref{weakmKdV} with $ F(u)=u^3-3P_0(u^2)$. Therefore the unconditional uniqueness result ensures that $ u $ is the only accumulation point of $ \{u_n\} $ and thus
   $\{u_{n}\} $ converges to $ u $ in the sense \eqref{conv1}-\eqref{conv4}.
   Now, using the  bounds on $ \{u_n\} $ and $  \{\partial_t u_n\} $, it is clear that for any $ \phi\in C^\infty(\T) $ and any $ T>0 $, the sequence
   $\{t\mapsto (u_n, \phi)_{H^{s}}\} $ is uniformly equi-continuous on $ [0,T] $. By Ascoli's theorem it follows  that
   $$\ (u_{n}, \phi)_{H^{s}}\to (u,\phi) \mbox{  in } C([0,T]) \; .
   $$
   In particular, for any fixed $ N\ge 1$, it holds
   $$
  \lim_{n\to \infty}  \sup_{t\in [0,T]} \|P_{\le N} (u_n-u)(t) \|_{H^s} =0 \; .
   $$
   This last limit combined with \eqref{bbb} ensures that
   $$
   u_n \to u\mbox{ in }  C([0,T]:H^s(\T)) \; .
   $$
   and thus $ u\in C([0,T]:H^s(\T)) $.
   
  Finally,  to prove the continuity with respect to initial data, we take a sequence $ \{u_0^m\} \subset B_{H^s}(0,2\|u_0\|_{H^s}) $ that converges to $u_0 $ in $ H^s(\T) $.
  Denoting by $ u^m $ the associated solutions to  \eqref{remKdV} that we have constructed above, we obtain in exactly the same way as above that for
   $T\sim \min(1, (1+\|u_0\|_{H^s})^{-10}) $ it holds
   $$
\|u_m\|_{Z^s_T} \lesssim \|u_0\|_{H^s}(1+\|u_0\|_{H^s}^2) , \quad \quad
\lim_{K\to +\infty} \sup_{m\in \N} \|P_{\ge K} u_m(t) \|_{L^\infty_T H^s} =0 \; .
$$
 and
 $$\ (u_{m}, \phi)_{H^{s}}\to (u,\phi) \mbox{  in } C([0,T]) \; .
   $$
This ensures that $ u_m \to u $ in $ C([0,T]; H^s(\T)) $ and completes the proof of the unconditional well-posedness of \eqref{remKdV}.

\subsection{Back to the mKdV equation}\label{back}
 For $ s\ge 0 $ we define the mapping
 $$
 \Psi \; :\;
 \begin{array}{rcl}
 L^\infty_T H^s_x & \longrightarrow  &L^\infty_T H^s_x \\
 u=u(t,x) & \longmapsto & \Psi(u)=\Psi(u)(t,x)=u(t,x+\int_0^t P_0(u^2(\tau)) \, d\tau)
 \end{array}
 $$
It is easy to check that $ \Psi $ is a bijection from  $  L^\infty_T H^s_x $ into itself and also from $ C([0,T]:H^s(\T) $ into itself
 with inverse bijection defined by
$$ \Psi^{-1}(u)=u\Bigl(t,x-\int_0^t P_0(u^2(\tau)) \, d\tau\Bigr) \, .
$$ Moreover, for $ s\ge 1/3$, it is not too hard to check that $ u\in L^\infty_T H^s_x $
 is a solution of \eqref{weakmKdV} with $ F(u)=u^3 $ if and only if $ \Psi(u)\in  L^\infty_T H^s_x $ is a solution to  \eqref{weakmKdV} with $ F(u)=u^3-3P_0(u^2) $.  Finally, we claim that $ \Psi $ and $ \Psi^{-1} $ are continuous from $ C([0,T]:H^s(\T) $ into itself. Indeed, let $ \{v_n\}_{n\ge 1} \subset C([0,T]:H^s(\T)) $ that converges to $ v$ in $C([0,T]:H^s(\T)) $. Then denoting $\int_0^t P_0(v_n^2)(s)\, ds $ by $ \alpha_n(t) $ and
  $ \int_0^t P_0(v^2)(s) \, ds $ by $ \alpha(t) $ , it is easy to check that
\begin{equation}\label{az}
\lim_{n\to \infty}  \sup_{t\in [0,T]} (\alpha_n(t)-\alpha(t)) =0
\end{equation}
  and
\begin{eqnarray*}
\sup_{t\in [0,1]} \| \Psi(v_n)(t)-\Psi(v)(t) \|_{H^s}&  \le  & \sup_{t\in [0,1]} \Bigl\|v_n(t, \cdot+\alpha_n(t)) -
v(t, \cdot+\alpha_n(t)) \Bigr\|_{H^s} \\
& & +\sup_{t\in [0,1]}  \Bigl\|v(t, \cdot+\alpha_n(t)) -
v(t, \cdot+\alpha(t)) \Bigr\|_{H^s}
\end{eqnarray*}
It is clear that the first term of the right-hand side of the above estimate converges to $0 $. Now, the second term can be rewritten  as
\begin{eqnarray*}
 I_n= \sup_{t\in [0,1]}  \Bigl( \sum_{k\in\Z} \Bigl| k^{s} (e^{ik \alpha_n(t)}-e^{ik\alpha(t)}) \widehat{v}(t,k)\Bigr|^2\Bigr)^{1/2}
\end{eqnarray*}
Since $ v\in C([0,T]:H^s(\T)) $,  $\{v(t)  :  t\in [0,T] \} $ is a compact set of $ H^s(\T) $ and thus
$$
\lim_{N\to \infty} \sup_{t\in [0,T]} \sum_{|k|\ge N} |k|^{2s} |\widehat{v}(t,k)|^2 =0 \; ,
$$
which combined with \eqref{az} ensures that $ \lim_{n\to\infty} I_n=0 $ and completes the proof of the desired continuity result.

These properties of $\Psi $ combined with the unconditional local  well-posedness of the renormalized mKdV equation in $  H^s(\T) $, clearly leads to Theorem \ref{main}.\vspace*{4mm} \\
\noindent \textbf{Acknowledgments.} The authors are very grateful to Professor Tsutsumi for pointing out a flaw in a first version of this work.
 They are also very grateful to the  anonymous Referee who  pointed out  some flaws in a previous  version of this work and greatly improved the present version with numerous helpful suggestions and comments.
 L.M and S.V were partially supported by the ANR project GEO-DISP.

\end{document}